\documentclass[graybox,envcountsame,envcountsect]{svmult}
\usepackage{amsmath,amssymb,amsbsy}
\usepackage{stmaryrd}
\usepackage{tikz}
\usepackage{graphicx}

\usepackage{mathptmx}       
\usepackage{helvet}         
\usepackage{courier}        
\usepackage{type1cm}        
\usepackage{makeidx}         
\usepackage{graphicx}        
\usepackage{multicol}        
\usepackage[bottom]{footmisc}

\makeindex             

\def\AA{{\mathcal A}}

\def\CC{{\mathcal C}}

\def\II{{\mathcal I}}
\def\JJ{{\mathcal J}}
\def\KK{{\mathcal K}}

\def\OO{{\mathcal O}}

\def\WW{{\mathcal W}}

\def\IN{{\mathbb{N}}}
\def\IQ{{\mathbb{Q}}}
\def\IR{{\mathbb{R}}}
\def\IS{{\mathbb{S}}}

\def\TO{\Longrightarrow}
\def\In{\subseteq}

\def\into{\hookrightarrow}

\def\prefix{\sqsubseteq}

\def\mto{\rightrightarrows}

\def\id{{\rm id}}

\def\dom{{\rm dom}}

\def\diam{{\rm diam}}

\def\mind{{\rm mind}}

\def\Baire{{\IN^\IN}}

\def\Tr{{\rm Tr}}

\def\LPO{\text{\rm\sffamily LPO}}
\def\LLPO{\text{\rm\sffamily LLPO}}
\def\WKL{\text{\rm\sffamily WKL}}
\def\RCA{\text{\rm\sffamily RCA}}
\def\ACA{\text{\rm\sffamily ACA}}

\def\BCT{\text{\rm\sffamily BCT}}

\def\IVT{\text{\rm\sffamily IVT}}

\def\BWT{\text{\rm\sffamily BWT}}

\def\BFT{\mbox{\rm\sffamily BFT}}

\def\B{\text{\rm\sffamily B}}
\def\I{\text{\rm\sffamily I}}

\def\C{\mbox{\rm\sffamily C}}
\def\ConC{\mbox{\rm\sffamily CC}}
\def\UC{\mbox{\rm\sffamily UC}}

\def\ConC{\mbox{\rm\sffamily CC}}
\def\XC{\mbox{\rm\sffamily XC}}
\def\AoUC{\mbox{\rm\sffamily AoUC}}
\def\ACC{\mbox{\rm\sffamily ACC}}
\def\LPO{\mbox{\rm\sffamily LPO}}
\def\LLPO{\mbox{\rm\sffamily LLPO}}

\def\MLPO{\mbox{\rm\sffamily MLPO}}
\def\MP{\text{\rm\sffamily MP}}

\def\MCT{\text{\rm\sffamily MCT}}

\def\K{\text{\rm\sffamily K}}

\def\Low{\text{\rm\sffamily L}}
\def\CL{\text{\rm\sffamily CL}}
\def\KL{\text{\rm\sffamily KL}}

\def\A{\text{\rm\sffamily A}}
\def\R{\text{\rm\sffamily R}}

\def\J{\text{\rm\sffamily J}}

\def\G{\text{\rm\sffamily G}}
\def\Z{\text{\rm\sffamily Z}}

\def\DNC{\text{\rm\sffamily DNC}}
\def\ACC{\text{\rm\sffamily ACC}}
\def\CFC{\text{\rm\sffamily CFC}}

\def\MLR{\text{\rm\sffamily MLR}}
\def\WWKL{\text{\rm\sffamily WWKL}}
\def\GEN{\text{\rm\sffamily GEN}}

\def\PC{\text{\rm\sffamily PC}}
\def\COH{\text{\rm\sffamily COH}}
\def\PA{\text{\rm\sffamily PA}}

\def\NASH{\text{\rm\sffamily NASH}}
\def\RDIV{\text{\rm\sffamily RDIV}}

\def\EC{\text{\rm\sffamily EC}}
\def\ATR{\text{\rm\sffamily ATR}}
\def\PWCC{\text{\rm\sffamily PWCC}}
\def\BIM{\text{\rm\sffamily BIM}}
\def\MAX{\text{\rm\sffamily MAX}}

\def\HBC{\text{\rm\sffamily HBC}}
\def\KL{\text{\rm\sffamily KL}}
\def\RT{\text{\rm\sffamily RT}}
\def\VCT{\text{\rm\sffamily VCT}}
\def\Int{\text{\rm\sffamily Int}}
\def\BGK{\text{\rm\sffamily BGK}}

\def\leqT{\mathop{\leq_{\mathrm{T}}}}

\newcommand{\leqW}{\mathop{\leq_{\mathrm{W}}}}
\def\equivW{\mathop{\equiv_{\mathrm{W}}}}

\def\leqSW{\mathop{\leq_{\mathrm{sW}}}}

\def\equivSW{\mathop{\equiv_{\mathrm{sW}}}}

\def\nleqW{\mathop{\not\leq_{\mathrm{W}}}}

\def\lW{\mathop{<_{\mathrm{W}}}}
\def\lSW{\mathop{<_{\mathrm{sW}}}}

\def\bigtimes{\mathop{\mathsf{X}}}

\def\stars{*_{\rm s}\;\!}

\def\r{{\rm r}}
\newcommand{\dash}{\mbox{-}}

\begin{document}

\title*{Weihrauch Complexity in Computable Analysis}
\titlerunning{Weihrauch Complexity in Computable Analysis} 
\author{Vasco Brattka, Guido Gherardi and Arno Pauly}
\institute{Vasco Brattka \at Faculty of Computer Science, Universit\"at der Bundeswehr M\"unchen, Germany and\\ Department of Mathematics \& Applied Mathematics, University of Cape Town, South Africa\\ \email{Vasco.Brattka@cca-net.de}
\and Guido Gherardi \at Dipartimento di Filosofia e Comunicazione, Universit\`{a} di Bologna, Italy\\ \email{Guido.Gherardi@unibo.it}
\and Arno Pauly \at Department of Computer Science, Swansea University, United Kingdom\\ \email{a.m.pauly@swansea.ac.uk}
}

%
%
\maketitle

\abstract*{We provide a self-contained introduction into Weihrauch complexity and its applications to computable analysis. 
                This includes a survey on some classification results and a discussion of the relation to other approaches.}
\abstract{We provide a self-contained introduction into Weihrauch complexity and its applications to computable analysis. 
                This includes a survey on some classification results and a discussion of the relation to other approaches.}

\section{The Algebra of Problems}
\label{sec:problems}

The Weihrauch lattice offers a framework to classify the uniform computational content of problems and theorems
from analysis and other areas of mathematics. 
This framework can be seen as an attempt to create a calculus of mathematical problems,
very much in spirit of Kolmogorov's interpretation of intuitionistic logic \cite{Kol32}.

We express mathematical problems with the help of partial multi-valued functions $f:\In X\mto Y$,
which are just relations $f\In X\times Y$. It has turned out to be fruitful for our approach to think of these relations as
input-output oriented multi-valued functions $f:\In X\mto Y$.
We consider $\dom(f)=\{x\in X:f(x)\not=\emptyset\}$ as the set of admissible instances $x$ of the problem $f$, 
and we consider the corresponding set of function values $f(x)\In Y$ as the set of possible results.
In the case of single-valued $f$ we identify $f(x)$ with the corresponding singleton. 
An example of a mathematical problem that the reader can have in mind as a prototypical case is the zero problem.
Obviously, many problems in mathematics can be expressed in terms
of solutions of equations of type $f(x)=0$ with a continuous $f:X\to\IR$. We formalize this problem.

\begin{example}[Zero problem]
\label{ex:zero}
Let $X$ be a topological space and let $\CC(X)$ denote the set of continuous $f:X\to\IR$.
The {\em zero problem} $\Z_X:\In\CC(X)\mto X,f\mapsto f^{-1}\{0\}$ is
the problem to find a solution $x\in X$ of an equation of type $f(x)=0$,
given a continuous function $f:X\to\IR$. 
The set $\dom(\Z_X)$ of admissible
instances of this problem is the
set of all continuous functions $f$ with a non-empty zero set $f^{-1}\{0\}$. 
The set $\Z_X(f)=f^{-1}\{0\}$ of solutions is the set of all zeros of $f$.
\end{example}

Mathematical problems can be combined in various natural ways to 
obtain new problems. The following definition lists a number of some
typical algebraic operations that we are going to use.
By $X\sqcup Y:=(\{0\}\times X)\cup(\{1\}\times Y)$ we denote the {\em disjoint union}.
By $X^*:=\bigcup_{i=0}^\infty (\{i\}\times X^i)$ we denote the {\em set of words} over $X$, where $X^i:=\bigtimes_{j=1}^i X$ stands for
the $i$--fold {\em Cartesian product} of $X$ with itself with $X^0:=\{()\}$. Here $()$ stands for the empty tuple or word.
By $\overline{X}:=X\cup\{\bot\}$ we denote the {\em completion} of $X$, where $\bot\not\in X$.
We use the set {\em natural numbers} $\IN=\{0,1,2,...\}$.

\begin{definition}[Algebraic operations]
\label{def:algebraic-operations}
Let $f:\In X\mto Y$, $g:\In Z\mto W$ and\linebreak 
$h:\In Y\mto Z$ be multi-valued functions. We define
the following operations (for exactly those inputs given by the specified domains):
\begin{enumerate}
\item $h\circ f:\In X\mto Z$, $(h\circ f)(x):=\{z\in Z:(\exists y\in f(x))\;z\in h(y)\}$ and\\ $\dom(h\circ f):=\{x\in \dom(f):f(x)\In\dom(h)\}$ \hfill (composition)\item $f\times g:\In X\times Z\mto Y\times W, (f\times g)(x,z):=f(x)\times g(z)$ and\\
$\dom(f\times g):=\dom(f)\times\dom(g)$ \hfill (product)
\item $f\sqcup g:\In X\sqcup Z\mto Y\sqcup W$, $(f\sqcup g)(0,x):=\{0\}\times f(x)$, $(f\sqcup g)(1,z):=\{1\}\times g(z)$ and
$\dom(f\sqcup g):=\dom(f)\sqcup\dom(g)$\hfill (coproduct)
\item $f\boxplus g:\In X\sqcup Z\mto\overline{Y}\times\overline{W}$, $(f\boxplus g)(0,x):=f(x)\times\overline{W}$, $(f\boxplus g)(1,z):=\overline{Y}\times g(z)$ and \\ 
$\dom(f\boxplus g):=\dom(f)\sqcup\dom(g)$ \hfill (box sum)
\item $f\sqcap g:\In X\times Z\mto Y\sqcup W, (f\sqcap g)(x,z):=f(x)\sqcup g(z)$ and\\
$\dom(f\sqcap g):=\dom(f)\times\dom(g)$ \hfill (meet)
\item $f+g:\In X\times Z\mto \overline{Y}\times \overline{W}, (f+g)(x,z):=(f(x)\times\overline{W})\cup(\overline{Y}\times g(z))$ and\\
$\dom(f+g):=\dom(f)\times\dom(g)$ \hfill (sum)
\item $f^*:\In X^*\mto Y^*,f^*(i,x):=\{i\}\times f^i(x)$ and\\
$\dom(f^*):=\dom(f)^*$ \hfill (finite parallelization)
\item $\widehat{f}:\In X^\IN\mto Y^\IN,\widehat{f}(x_n)_n:=\bigtimes_{i\in\IN} f(x_i)$ and\\
$\dom(\widehat{f}):=\dom(f)^\IN$ \hfill (parallelization)
\end{enumerate}
\end{definition}

Here $f^i:=\bigtimes_{j=1}^i f$ denotes the $i$--fold product of $f$
with itself, where $f^0=\id_{X^0}$. 
It is important to point out that the appropriate definition of the domain of $h\circ f$ is crucial. If $x\in\dom(h\circ f)$, then we require that all
possible results $y\in f(x)$ of $f$ upon input of $x$ are supported by $h$,
i.e., $f(x)\In\dom(h)$. This definition of composition corresponds to
our understanding of multi-valued functions as computational problems.\footnote{The way we define composition
turns the multi-valued functions into morphisms of a specific category~\cite{Pau17} that is not identical to the usual category of relations.}
We often write for short $hf$ for the composition $h\circ f$.

The reader might notice some relations between the resource oriented
interpretation of linear logic and the way we combine mathematical
problems (see section \ref{subsec:linear-logic}). Indeed, the following intuitive interpretation of some of our algebraic
operations is useful:

\begin{enumerate}
\item The composition $h\circ f$ applies both problems consecutively, first $f$ and then $h$.
\item The product $f\times g$ provides both problems $f$ and $g$ in parallel. For each instance  one obtains solutions of both $f$ and $g$.
\item The coproduct $f\sqcup g$ provides both problems $f$ and $g$ alternatively. For each instance one can select to obtain either a solution of $f$ or of $g$.
\item The meet $f\sqcap g$ provides either $f$ or $g$.  For each instance one either obtains a solution for $f$ or for $g$;
        one learns a posteriori which one it is, but one cannot control in advance which one it will be.
\item The sum $f+ g$ provides two potential solutions for given instances of $f$ and $g$, at least one of which has to be correct. 
\item The finite parallelization $f^*$ allows arbitrarily many finite applications of $f$ in parallel, and with each instance one can select how many applications are to be used in parallel.
\item The parallelization $\widehat{f}$ allows countably many applications of $f$ in parallel.
\end{enumerate}

Given the above list of operations we can derive other algebraic operations.

\begin{definition}[Juxtaposition]
For $f:\In X\mto Y$ and $g:\In X\mto Z$ we denote by $(f,g):\In X\mto Y\times Z$ the 
{\em juxtaposition} of $f$ and $g$, which is defined by $(f,g):=(f\times g)\circ\Delta_X$,
where $\Delta_X:X\mto X\times X,x\mapsto (x,x)$ denotes the {\em diagonal} of $X$.
\end{definition}

Given two problems $f$ and $g$ we want to express what it means that $f$ solves $g$.

\begin{definition}[Solutions]
\label{def:solutions}
Let $f,g:\In X\mto Y$ be multi-valued functions. We define
$f\prefix g:\iff \dom(g)\In\dom(f)\mbox{ and }(\forall x\in\dom(g))\;f(x)\In g(x)$.
In this situation we say that $f$ {\em solves} $g$, $f$ is a {\em strengthening} of $g$ and $g$ is a {\em weakening} of $f$.
\end{definition}

Intuitively, $f\prefix g$ means that all instances of $g$ are also instances of $f$, and on all these common instances
$f$ yields a possible solution of $g$. It is clear that the relation $\prefix$ yields a preorder, i.e., it is reflexive and transitive.

Many theorems give rise to mathematical problems.
In general, a theorem of the logical form 
\[(\forall x\in X)(x\in D\TO (\exists y\in Y)P(x,y))\]
translates into the problem
\[F:\In X\mto Y,x\mapsto\{y\in Y:P(x,y)\}\mbox{ with }\dom(F):=D.\]
That is, $F$ plays the r\^ole of a multi-valued Skolem function for the statement of the theorem.
The problem $F$ measures the difficulty of finding a suitable $y$, given $x$, whereas the condition
encapsulated in $D$ is a purely classical promise that is not meant to bear any constructive content. 
As an example we mention the intermediate value theorem.

\begin{example}[Intermediate value theorem]
$\IVT:\In\CC[0,1]\mto[0,1],f\mapsto f^{-1}\{0\}$, where $\dom(\IVT)$ 
contains all $f\in\CC[0,1]$ with $f(0)\cdot f(1)<0$, is called the {\em intermediate value theorem}.
It is easy to see that $\Z_{[0,1]}\prefix\IVT$ holds.
\end{example}

\subsubsection*{Bibliographic Remarks}

\begin{petit}
Algebraic operations on multi-valued functions have been used frequently in computable analysis.
For instance, composition in the way defined here, product, juxtaposition and parallelization have been used by Brattka~\cite{Bra99b,Bra03}.
The coproduct operation and finite parallelization was introduced by Pauly in~\cite{Pau10}.
The meet operation was introduced by Brattka and Gherardi~\cite{BG11}, and the box sum was introduced by Dzhafarov~\cite{Dzh18}.
Inspired by the definition of the box sum, the definition of the sum from Brattka, Gherardi and H\"olzl~\cite{BGH15a} appears here
for the first time in a modified version that has better properties.
The category of multi-valued functions was studied by Pauly~\cite{Pau17}.
\end{petit}

\section{Represented Spaces}
\label{sec:represented}

In this section we want to provide the data types that we will use for problems $f:\In X\mto Y$.
For a purely topological development of our theory it would be sufficient to consider topological spaces $X$ and $Y$.
However, since we want to discuss computability properties too, we need slightly more structure on the spaces $X$ and $Y$,
and this structure is provided by representations.

\begin{definition}[Represented spaces]
A {\em represented space} $(X,\delta)$ is a set $X$ together with a surjective 
partial function $\delta:\In\IN^\IN\to X$.
\end{definition}

If $\delta(p)=x$ then we call $p$ a {\em name} for $x$, and we reserve the word {\em representation} for the map $\delta$ itself.
We endow Baire space $\IN^\IN$ with its usual product topology of the discrete topology on $\IN$ and
we always assume that a represented space $(X,\delta_X)$ is endowed with the final topology $\OO(X)$ induced by $\delta_X$ on $X$, which
is the largest topology on $X$ that turns $\delta_X$ into a continuous map.
In this situation $\delta_X$ is automatically a quotient map.
Typically, we will deal with {\em admissible} representations\footnote{See the chapter ``\emph{Admissibly Represented Spaces and QCB-Spaces}'' by Schr\"oder in this book for more details.} 
$\delta_X$ that are not just quotient maps but they are even closer linked to the
topology $\OO(X)$.
In the following we will often just write for short $X$ for a represented space
if the representation is clear from the context or not needed explicitly. 
We can now formally define problems.

\begin{definition}[Problems]
We call partial multi-valued functions $f:\In X\mto Y$ on represented spaces $X,Y$
for short {\em problems}.
\end{definition}

Properties of problems such as computability and continuity can easily be introduced
via realizers.

\begin{definition}[Realizer]
Given represented spaces $(X,\delta_X)$, $(Y,\delta_Y)$, a problem $f:\In X\mto Y$
and a function $F:\In\IN^\IN\to\IN^\IN$, we define
$F\vdash f:\iff\delta_YF\prefix f\delta_X$.
In this situation  we say that $F$ is a {\em realizer} of $f$.
\end{definition}

In other words, $F$ is a realizer of $f$ if $\delta_YF$ solves $f\delta_X$. 
Obviously, this concept depends on the underlying represented spaces and the
notation $F\vdash f$ is only justified when these are clear from the context.
%

On Baire space $\IN^\IN$ it is clear what a continuous function $F:\In\IN^\IN\to\IN^\IN$ is. 
Computability of such functions can be defined via Turing machines in a well-known way.
Such properties can now easily be transfered to problems via realizers.

\begin{definition}[Computability and continuity]
\label{def:computability}
A problem $f$ is called {\em computable} ({\em continuous}), if it has a computable (continuous) realizer.
\end{definition}

We warn the reader that the resulting notion of continuity for single-valued functions is not automatically the topological notion of continuity that is induced by the final topologies of the representations. However, 
every total single-valued function $f:X\to Y$ on represented spaces that is continuous in our sense is also continuous in the usual
topological sense with respect to the final topologies, and in all
our applications we will use admissible representations for which these two notions even coincide.

Two representations $\delta_1,\delta_2$ of the same set $X$ are called {\em equivalent} if the identity $\id:(X,\delta_1)\to(X,\delta_2)$
and its inverse are computable. It is easy to see that equivalent representations yield the same notion of computability and continuity.

By $\CC(X,Y)$ we denote the set of continuous functions $f:X\to Y$ in terms of Definition~\ref{def:computability}.
The category of represented spaces is Cartesian closed, and the same holds for the category of admissibly represented spaces. In particular, we have
canonical ways of defining product and function space representations. 

In order to define those, we use pairing functions. We define a pairing function $\langle,\rangle:\IN^\IN\times\IN^\IN\to\IN^\IN$
by $\langle p,q\rangle(2n):=p(n)$ and $\langle p,q\rangle(2n+1):=q(n)$ for $p,q\in\IN^\IN$ and $n\in\IN$.
We define a pairing function of type $\langle,\rangle:(\IN^\IN)^\IN\to\IN^\IN$ by
$\langle p_0,p_1,p_2,...\rangle\langle n,k\rangle:=p_n(k)$ for all $p_i\in\IN^\IN$ and $n,k\in\IN$,
where $\langle n,k\rangle$ is the standard Cantor pairing defined by $\langle n,k\rangle:=\frac{1}{2}(n+k+1)(n+k)+k$. 
Finally, we note that by $np$ we denote the concatenation of a number $n\in\IN$ with a sequence $p\in\IN^\IN$.

We assume that we have some standard representation $\Phi$ of (a sufficiently large class\footnote{It suffices to consider all continuous functions $f:\In\IN^\IN\to\IN^\IN$
with $G_\delta$--domain since any continuous function can be extended to such a function.}) of continuous
functions, i.e., for any such function $f:\In\IN^\IN\to\IN^\IN$ there is a $p\in\IN^\IN$ with $f=\Phi_p$. For total functions this representation yields the exponential in the category of (admissibly) represented spaces and satisfies natural versions of the utm- and smn-theorems.
For computable $p$ one obtains the computable functions $\Phi_p$ with natural domains (see \cite{Wei00} for details). 
For $p\in\IN^\IN$ we denote by $p-1\in\IN^\IN\cup\IN^*$ the sequence or word that is formed as concatenation of 
$p(0)-1$, $p(1)-1$, $p(2)-1$,... with the understanding that $-1=()$ is the empty word.

\begin{definition}[Constructions on representation]
\label{def:operations-representations}
Let $(X,\delta_X)$ and $(Y,\delta_Y)$ be represented spaces. We define
\begin{enumerate}
\item $\delta_{X\times Y}:\In\IN^\IN\to X\times Y$, $\delta_{X\times Y}\langle p,q\rangle:=(\delta_X(p),\delta_Y(q))$
\item $\delta_{X\sqcup Y}:\In\IN^\IN\to X\sqcup Y$, $\delta_{X\sqcup Y}(0p):=(0,\delta_X(p))$ and $\delta_{X\sqcup Y}(1p):=(1,\delta_Y(p))$
\item $\delta_{X^*}:\In\IN^\IN\to X^*$, $\delta_{X^*}(n\langle p_1,p_2,...,p_n\rangle):=(n,(\delta_X(p_1),\delta_X(p_2),...,\delta_X(p_n)))$
\item $\delta_{X^\IN}:\In\IN^\IN\to X^\IN$, $\delta_{X^\IN}\langle p_0,p_1,p_2,...\rangle:=(\delta_X(p_n))_{n\in\IN}$ 
\item $\delta_{\CC(X,Y)}:\In\IN^\IN\to\CC(X,Y)$ by $\delta_{\CC(X,Y)}(p)=f:\iff \Phi_p\vdash f$ 
\item $\delta_{\overline{X}}:\IN^\IN\to\overline{X}$, $\delta_{\overline{X}}(p):=\delta_X(p-1)$ if $p-1\in\dom(\delta_X)$ and $\delta_{\overline{X}}(p):=\bot$ otherwise.
\end{enumerate}
\end{definition}

Many spaces that occur in analysis are actually computable metric spaces.
For the definition we assume that the reader knows the notion of a computable (double) sequence of real numbers.

\begin{definition}[Computable metric spaces and Cauchy representations]
\begin{enumerate}
\item
A {\em computable metric space} $(X,d,\alpha)$ is a separable metric space $(X,d)$ with metric $d:X\times X\to\IR$ and a dense sequence $\alpha:\IN\to X$
such that $d\circ(\alpha\times\alpha):\IN^2\to \IR$ is a computable double sequence of real numbers.
\item
We define the {\em Cauchy representation} $\delta_X:\In\IN^\IN\to X$  by $\delta_X(p):=\lim_{n\to\infty}\alpha p(n)$
and $\dom(\delta_X)=\{p\in\IN^\IN:(\forall i>j)\;d(\alpha p(i),\alpha p(j))<2^{-j}\}$.
\end{enumerate}
\end{definition}

A standard numbering of the rational numbers $\IQ$ and the Euclidean metric yields the standard Cauchy representation $\delta_\IR$ 
of real numbers. Cauchy representations are examples of admissible representations, and for such representations continuity in the usual
topological sense and continuity defined via realizers coincides.
In particular, $\CC(\IR):=\CC(\IR,\IR)$ is the usual set of continuous functions. 
In the following we consider $\IN,\IR,[0,1],2^\IN,\IN^\IN$ and similar spaces often as computable metric spaces in the
straightforward sense without further mentioning this fact.
A {\em computable Banach space} is just a computable metric space that is additionally a Banach space
and such that the linear operations are computable. If the space is additionally a Hilbert space, then it is called a {\em computable Hilbert space}.

A non-metrizable space that we occasionally need is Sierpi\'nski space $\IS=\{0,1\}$, which is endowed with the topology $\OO(\IS)=\{\emptyset,\IS,\{1\}\}$.
By $\widehat{n}\in\IN^\IN$ we denote the constant sequence with value $n\in\IN$.

\begin{definition}[Sierpi\'nski space]
Let $\delta_\IS:\IN^\IN\to\IS$ be defined by $\delta_\IS(\widehat{0})=0$ and $\delta_\IS(p)=1$ for all $p\not=\widehat{0}$.
\end{definition}

We close this section with a discussion of computability properties of subsets.
The most important notion for us is that of a {\em co-c.e.\ closed set}.
Given a computable metric space $(X,d,\alpha)$ we denote by $B(x,r):=\{y\in X:d(x,y)<r\}$ the open ball
with center $x\in X$ and radius $r\geq0$. More specifically, we denote by $B_{\langle n,\langle i,k\rangle\rangle}:=B(\alpha(n),\frac{i}{k+1})$
a basic open ball.

\begin{definition}[Co-c.e.\ closed subsets]
Let $X$ be a computable metric space. Then $A\In X$ is called {\em co-c.e.\ closed},
if $X\setminus A=\bigcup_{n\in\IN}B_{p(n)}$ for some computable $p\in\IN^\IN$.
\end{definition}

For $X=\IN^\IN$ the co-c.e.\ closed subsets are also known as {\em $\mathrm\Pi^0_1$--classes}.
By $\AA(X)$ we denote the set of closed subsets of a topological space $X$.
The definition of co-c.e.\ closed subsets of computable metric spaces $X$ 
directly leads to a representation $\psi_-$ of the set $\AA(X)$ 
defined by $\psi_-(p):=X\setminus \bigcup_{n=0}^\infty B_{p(n)}$.
We denote the represented space $(\AA(X),\psi_-)$ for short by $\AA_-(X)$. 
We now formulate two equivalent characterizations of co-c.e.\ closed sets.
For every set $A\In X$ we denote its {\em characteristic function} by
$\chi_{A}:X\to\IS$, and it is defined by $\chi_A(x)=1:\iff x\in A$.

\begin{proposition}[Co-c.e.\ closed sets]
\label{prop:closed}
Let $X$ be a computable metric space and let $A\In X$. Then the following are equivalent:
\begin{enumerate}
\item $A$ is co-c.e.\ closed, 
\item $A=f^{-1}\{0\}$ for some computable $f:X\to\IR$,
\item $\chi_{X\setminus A}:X\to\IS$ is computable.
\end{enumerate}
These equivalences are uniform, i.e., the maps $\AA_-(X)\to\CC(X,\IS),A\mapsto\chi_{X\setminus A}$
and $\CC(X)\to\AA_-(X),f\mapsto f^{-1}\{0\}$ are computable and admit (in the second case multi-valued) computable right inverses.
\end{proposition}

The third characterization has the advantage that it is the most general of these three, and it works even for
arbitrary represented spaces $X$. Hence, for such spaces we define $\psi_-$ by
$\psi_-(p)=A:\iff \delta_{\CC(X,\IS)}(p)=\chi_{X\setminus A}$.
We denote the corresponding represented space $(\AA(X),\psi_-)$ also by $\AA_-(X)$.
Due to Proposition~\ref{prop:closed} this notation is consistent with the special definition for computable metric spaces $X$ above.
Besides the notion of a co-c.e.\ closed subset we also need the notion of a co-c.e.\ compact subset.

\begin{definition}[Computable compact subsets]
Let $X$ be a computable metric space and let $K\In X$ be compact. 
\begin{enumerate}
\item $K$ is called {\em co-c.e.\ compact}, if $\{\langle\langle n_1,...,n_k\rangle,k\rangle\in\IN:K\In\bigcup_{i=1}^kB_{n_i}\}$ is c.e.
\item $K$ is called {\em computably compact}, if $K$ is co-c.e.\ compact and there exists a computable sequence that is dense in $K$.
\end{enumerate}
\end{definition}

Obviously, a computable metric space is computably compact if and only if it is co-c.e.\ compact.
Similarly as in the case of closed sets we can derive a representation $\kappa_-$ of the set $\KK(X)$ of compact subsets
that is based on (1) and a representation $\kappa$ of compact sets that is based on (2).
By $\KK_-(X)$ we denote the represented space $(\KK(X),\kappa_-)$
Once again there is a more general representation that works for arbitrary represented spaces,
but we will not formalize this representation here. 


\subsubsection*{Bibliographic Remarks}

\begin{petit}
The theory of representations and of computable functions on represented spaces was developed by Kreitz and Weihrauch~\cite{KW85,KW87,WK87,Wei87},
who also introduced the notion on an admissible representation. Admissible representation in a more general sense have been further studied by Schr\"oder~\cite{Sch02},
who also recognized the relevance of Sierpi\'nski space in this context. Computable metric spaces were first introduced by Lacombe~\cite{Lac59}.
Represented spaces are used as a basic framework for computable analysis~\cite{Wei00,BHW08,Pau16}.
Computability properties of subsets of computable metric spaces were studied since Lacombe~\cite{Lac55,Lac57},
and discussions of corresponding representations can be found in~\cite{KW87,WK87,BW99,BP03,Pau16}.
Proposition~\ref{prop:closed} for metric spaces is taken from Brattka and Presser~\cite[Theorem~3.10, Corollary~3.14]{BP03}. The mentioned representations for compact subsets of metric spaces are also studied in~\cite{BP03}.
Representations of subsets for general represented spaces have been studied by Pauly~\cite{Pau16}.
\end{petit}

\section{The Weihrauch Lattice}
\label{sec:Weihrauch}

We now want to define Weihrauch reducibility as a way to compare problems with each other.
The goal is that $f\leqW g$ expresses the fact that $f$ can be computed by a single application of $g$.
We will need two variants of such a reducibility. 
By ${\id:\IN^\IN\to\IN^\IN}$ we denote the {\em identity} of Baire space. 
For other sets $X$ we usually add an index $X$ and write the {\em identity} as $\id_X:X\to X$.
For $F,G:\In\IN^\IN\to\IN^\IN$ we define $\langle F,G\rangle(p):=\langle F(p),G(p)\rangle$.

\begin{definition}[Weihrauch reducibility]
Let $f$ and $g$ be problems. We define:
\begin{enumerate}
\item $f\leqW g:\iff (\exists$ computable $H,K:\In\IN^\IN\to\IN^\IN)(\forall G\vdash g)\; H\langle \id,GK\rangle\vdash f$.
\item $f\leqSW g:\iff (\exists$ computable $H,K:\In\IN^\IN\to\IN^\IN)(\forall G\vdash g)\;HGK\vdash f$.
\end{enumerate}
We say that $f$ is {\em (strongly) Weihrauch reducible} to $g$, if $f\leqW g$ ($f\leqSW g$) holds. 
\end{definition}

The diagram in Figure~\ref{fig:Weihrauch} illustrates Weihrauch reducibility and its strong counterpart.
It is easy to see that $f\leqSW g$ implies $f\leqW g$. It is also easy to see that $\leqW$ and $\leqSW$ are preorders, i.e., they are reflexive and transitive. 
We denote the corresponding equivalences by $\equivW$ and $\equivSW$, respectively,
and we use  the symbols $\lW$ and $\lSW$ for strict reducibilities, respectively. 
Similar reducibilities can be defined if the notion of computability is replaced by continuity or other suitable categories.
A more categorical characterization of Weihrauch reducibility that does neither mention realizers nor Baire space is given by the following proposition.

\begin{proposition}
\label{prop:GM09}
Let $f:\In X\mto Y$ and $g:\In Z\mto W$ be problems. Then:
\begin{enumerate}
\item $f\leqW g$ if and only if there are computable $h:\In V\times W\mto Y$ and $k:\In X\mto V\times Z$ for some represented space $V$ such that $h\circ (\id_V\times g)\circ k\prefix f$.
\item $f\leqSW g$ if and only if there are computable $h:\In W\mto Y$ and $k:\In X\mto Z$ such that $h\circ g\circ k\prefix f$.
\end{enumerate}
\end{proposition}

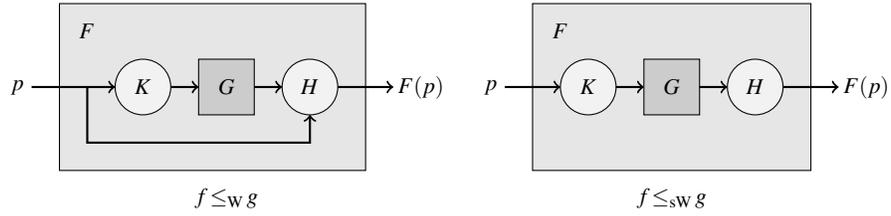
\begin{figure}[tb]
\begin{center}
\begin{tikzpicture}[scale=.37,auto=left]

\draw[style={fill=black!10}] (-4,6) rectangle (7,0);
\draw[style={fill=black!20}]  (1,4) rectangle (3,2);
\draw[style={fill=black!5}]  (-1,3) ellipse (1 and 1);
\draw[style={fill=black!5}]  (5,3) ellipse (1 and 1);
 
\node at (-1,3) {$K$};
\node at (5,3) {$H$};
\node at (2,3) {$G$};
\node at (-3,5) {$F$};
\node at (-5.5,3) {$p$};
\node at (9,3) {$F(p)$};
\node at (2,-1) {$f\leqW g$};

\draw[->,thick] (0,3) -- (1,3);
\draw[->,thick] (-5,3) -- (-2,3);
\draw[->,thick] (3,3) -- (4,3) ;
\draw[->,thick] (6,3) -- (8,3);
\draw[->,thick] (-3,3) -- (-3,1) -- (5,1) -- (5,2);

\draw[style={fill=black!10}] (13,6) rectangle (23,0);
\draw[style={fill=black!20}]  (17,4) rectangle (19,2);
\draw[style={fill=black!5}]  (15,3) ellipse (1 and 1);
\draw[style={fill=black!5}]  (21,3) ellipse (1 and 1);
 
\node at (15,3) {$K$};
\node at (21,3) {$H$};
\node at (18,3) {$G$};
\node at (14,5) {$F$};
\node at (11.5,3) {$p$};
\node at (25,3) {$F(p)$};
\node at (18,-1) {$f\leqSW g$};

\draw[->,thick] (16,3) -- (17,3);
\draw[->,thick] (12,3) -- (14,3);
\draw[->,thick] (19,3) -- (20,3) ;
\draw[->,thick] (22,3) -- (24,3);

\end{tikzpicture}
\end{center}
\ \\[-1cm]
\caption{Weihrauch reducibility and strong Weihrauch reducibility}
\label{fig:Weihrauch}
\end{figure}

Even though the proof of Proposition~\ref{prop:GM09} is elementary, there is a subtle point in it.
Namely the proof requires a version of the axiom of choice.
In fact, we are freely using the axiom of choice, and mostly we invoke the following version.

\begin{svgraybox}
The {\em axiom of choice for Baire space}: every problem $f$ has a realizer $F$.
\end{svgraybox}

The fact that Weihrauch reducibility captures the idea of using $g$ exactly once in the course
of the computation is stated in the following theorem that we only formulate in intuitive terms here:

\begin{theorem}[Generalized Turing oracles]
\label{thm:Tavana-Weihrauch}
$f\leqW g$ holds if and only if $f$ can be computed on a (generalized) Turing machine that uses
exactly one application of $g$ in the course of its computation.
\end{theorem}

We emphasize that $f\leqW g$ actually {\em requires} that the oracle $g$ is used once in the course of the computation of $f$.
Hence, using the oracle $g$ can actually be an obstacle if the domain of $g$ contains only complicated points.

We note that a characterization of strong Weihrauch reducibility analogous to Theorem~\ref{thm:Tavana-Weihrauch} would
require discarding all results that were obtained in the course of the computation other than the result of the application of the oracle $g$.
This would rather be an unnatural way of using oracles, and it indicates why ordinary Weihrauch reducibility is a more appropriate
concept from this perspective.

The relation between strong and ordinary Weihrauch reducibility is similar to the relation between one-one and many-one reducibility
in classical computability theory, and it can be expressed using the notion of 
a cylinder.

\begin{definition}[Cylinder]
A problem $f$ is called a {\em cylinder} if $\id\times f\leqSW f$.
\end{definition}

It is clear that $f\leqSW\id\times f$ and $\id\times f\equivW f$ hold for all problems $f$,
whereas $\id\times f\leqSW f$ is a specific property of $f$ that allows to ``feed the input through to $f$.''

\begin{proposition}[Cylinder]
\label{prop:cylinder}
A problem $f$ is a cylinder if and only if for all problems $g$ the following holds: $g\leqW f\iff g\leqSW f$.
\end{proposition}

It is important to mention that the definitions of $\leqW$ and $\leqSW$ are invariant under
the replacement of represented spaces by equivalent ones \cite[Lemma~2.1]{BG11}.
The equivalence classes induced by $\equivW$ and $\equivSW$ are called {\em Weihrauch degrees} and {\em strong Weihrauch degrees}, respectively.
The reducibilities $\leqW$ and $\leqSW$ naturally extend to these degrees.

Most algebraic operations defined in Definition~\ref{def:algebraic-operations} are monotone with 
respect to (strong) Weihrauch reducibility. We say that a binary operation $\Box$ on problems is {\em monotone} with respect to $\leqW$,
if for all problems $f_0,f_1,g_0$ and $g_1$ condition 1.\ holds, and a unary operation ${\;}^\Box$ on problems is called a {\em closure operator} with respect to $\leqW$,
if for all problems $f,g$ condition 2.\ holds:
\begin{enumerate}
\item ($f_0\leqW f_1\mbox{ and }g_0\leqW g_1)\TO f_0\Box g_0\leqW f_1\Box g_1$\hfill (monotone)
\item $f\leqW f^\Box, f^{\Box\Box}\leqW f^\Box\mbox{ and }(f\leqW g\TO f^\Box\leqW g^\Box)$\hfill (closure operator)
\end{enumerate}
Analogously to monotone, we define {\em antitone} with a reversed order on one side. 
Monotonicity and closure operators with respect to $\leqSW$ are defined analogously.

\begin{proposition}[Monotonicity and closure operators]
\label{prop:mon-closure}
We obtain:
\begin{enumerate}
\item The binary operations $\times$, $\sqcup$, $\sqcap$, $\boxplus$ and $+$ are all monotone with respect to $\leqW$ and $\leqSW$.
\item The unary operation $\;^*$ is a closure operator with respect to $\leqW$ and monotone with respect to $\leqSW$.
\item The unary operation $\ \widehat{\;}\ $ is a closure operator with respect to $\leqW$ and $\leqSW$.
\end{enumerate}
In particular, all the mentioned operations extend to operations on degrees.
\end{proposition}

It is an obvious question whether there is any least and any greatest Weihrauch degree.
The first question is easy to answer.

\begin{definition}[Special Weihrauch degrees]
By ${\mathbf 0}$ we denote the (strong) Weihrauch degree of the nowhere defined problems, and by ${\mathbf 1}$ we 
denote the Weihrauch degree of the identity $\id$.
\end{definition}

It is easy to see that ${\mathbf 0}$ is exactly the class of all nowhere defined problems, and it is
the least (strong) Weihrauch degree. The class $\mathbf{1}$ characterizes the computable problems in the sense
that $f\leqW\mathbf{1}$ holds if and only if $f$ is computable. 
In many respects $\mathbf{0}$ and $\mathbf{1}$ behave algebraically like the numerical constants $0$ and $1$.

The question whether there is a greatest Weihrauch degree is less straightforward to answer. 
If we do not accept the axiom of choice for Baire space, then the class of problems without realizer
form a natural top element. Since we are accepting the axiom of choice, this natural top element is not
available, and we can only add an additional top element to the Weihrauch degrees.\footnote{See Brattka and Pauly~\cite{BP18} for a more detailed discussion.}

If one is not interested in classifying specific problems with general types $X, Y$ as they appear in analysis,
but if one rather wants to study the structure of Weihrauch degrees as such,
then it is sufficient to consider problems of type $f:\In \IN^\IN\mto\IN^\IN$ on Baire space.
We make this slightly more precise.

\begin{lemma}[Realizer version]
\label{lem:realizer}
Let $(X,\delta_X)$ and $(Y,\delta_Y)$ be represented spaces and let $f:\In X\mto Y$ be a problem.
Then the {\em realizer version} $f^{\rm r}:\In\IN^\IN\mto\IN^\IN$ of $f$ is defined by
$f^\r:=\delta_Y^{-1}\circ f\circ\delta_X$. We have $f^\r\equivSW f$.
\end{lemma}

This means that every (strong) Weihrauch degree has a representative of type $f^\r:\In\IN^\IN\mto\IN^\IN$.
By $\WW$ and $\WW_{\rm s}$ we denote the set of Weihrauch degrees and strong Weihrauch degrees, respectively, 
both restricted to problems on Baire space.\footnote{We use the restriction to Baire space for our formal definition of $\WW$ and $\WW_{\rm s}$, 
since the class of all (strong) Weihrauch degrees of problems with arbitrary type does not form a set.}

\begin{theorem}[Weihrauch lattice]
\label{thm:lattice}
The Weihrauch degrees $(\WW,\leqW)$ form a distributive lattice with supremum operation $\sqcup$, infimum operation $\sqcap$ and bottom element ${\mathbf 0}$.
\end{theorem}

Also the strong Weihrauch degrees form a lattice structure, albeit a non-distributive one with a different supremum operation.

\begin{theorem}[Strong Weihrauch lattice]
\label{thm:strong-lattice}
The strong Weihrauch degrees $(\WW_{\rm s},\leqSW)$ form a non-distributive lattice with supremum $\boxplus$, infimum operation $\sqcap$ and bottom element ${\mathbf 0}$.
\end{theorem}

\subsubsection*{Bibliographic Remarks}

\begin{petit}
The reducibility that is now called {\em Weihrauch reducibility} was introduced by Klaus Weihrauch in the late 1980s and 
appeared in two unpublished technical reports \cite{Wei92a,Wei92c}.
He supervised six PhD and MSc projects on this topic (by von Stein \cite{Ste89}, Mylatz \cite{Myl92,Myl06},
Brattka \cite{Bra93}, Hertling \cite{Her96d}, Pauly \cite{Pau07}). 
Most of this material remained unpublished, and the concept featured only in a few early publications by Brattka, Weihrauch and Gherardi~\cite{Bra99,Wei00,Bra05,Ghe06a,BG09}.
In a more abstract categorical setting a concept related to Weihrauch reducibility has independently been studied by Hirsch~\cite{Hir90}.
Theorem~\ref{thm:Tavana-Weihrauch} is due to Tavana and Weihrauch~\cite{TW11}, and we refer the readers
to the reference for a precise formulation.
The subject took a turn when Gherardi and Marcone~\cite{GM09} defined Weihrauch reducibility in its full generality for problems, and
they promoted Weihrauch complexity as a uniform version of reverse mathematics.
Brattka and Gherardi~\cite{BG11,BG11a} continued to study the subject and discovered that the structure is a lower semilattice,
whereas Pauly~\cite{Pau10a} independently provided the coproduct operation and discovered that the structure is a distributive upper semilattice.
The notion of a cylinder is taken from Brattka and Gherardi~\cite{BG11}.
Several authors continued to investigate the subject from the perspective of reverse mathematics, among them
Dorais, Dzhafarov, Hirst, Mileti and Shafer \cite{DDH+16,Dzh15,Dzh18}  and Hirschfeldt and Jockusch \cite{Hir15,HJ16}.
Hirschfeldt and Jockusch also introduced a generalized version of Weihrauch reducibility that has a built-in closure under composition.
Dzhafarov provided the box sum operation $\boxplus$ and proved that the strong Weihrauch degrees form a non-distributive lattice~\cite{Dzh18}.
Independently, a polynomial-time version of Weihrauch complexity was used by Kawamura and Cook to classify
the uniform computational complexity of problems in analysis \cite{Kaw10,KC10}.
\end{petit}

\section{Algebraic and Topological Properties}
\label{sec:algebraic}

In this section we discuss a number of algebraic and topological notions and their interactions that turned out to be 
fruitful for the study of the Weihrauch lattice.
We mention that while $f\leqW{\mathbf 1}$ characterizes the computable problems $f$, also
the relation ${\mathbf 1}\leqW f$ bears some meaning.

\begin{definition}[Pointedness]
We call a problem $f$ {\em pointed}, if $\id\leqW f$ holds.
Analogously, we can define {\em strong pointedness} with the help of $\leqSW$ instead of $\leqW$.
\end{definition}

It is easy to see that the pointed problems are exactly those with a computable point in their domain.
By definition $f^*$ is always pointed since $f^0=\id_{\{()\}}$.
We introduce some further terminology that can be expressed with the help of the 
algebraic operations.

\begin{definition}[Idempotency and parallelizability]
Let $f$ be a problem.
\begin{enumerate}
\item We call $f$ {\em idempotent} if $f\times f\equivW f$.
\item We call $f$ {\em parallelizable} if $\widehat{f}\equivW f$.
\end{enumerate}
Analogously, we define {\em strong idempotency} and {\em strong parallelizability} with the help of $\equivSW$ instead of $\equivW$.
\end{definition}

Whether or not a problem is idempotent or parallelizable might be hard to prove in some instances.
In the following example the first statement is relatively easy to obtain, whereas the second one is harder to prove
(see Theorems~\ref{thm:IVT} and \ref{thm:CC-not-idempotent}).

\begin{example}
\label{ex:idempotent-parallelizable}
$\widehat{\IVT}\equivSW\Z_{[0,1]}$ and hence $\Z_{[0,1]}$ is (strongly) parallelizable, but $\IVT$ is not idempotent.
\end{example}

The following result captures some easy observations. Pointedness is involved here, since
$f^0=\id_{\{()\}}$ is pointed for every problem $f$.

\begin{proposition}[Idempotency and parallelizability]
\label{prop:idempotency-parallelizability}
Let $f$ be a problem. Then:
\begin{enumerate}
\item $f$ (strongly) parallelizable $\TO f$ (strongly) idempotent.
\item $f$ pointed and idempotent $\iff f^*\equivW f$.
\item $f$ strongly pointed and strongly idempotent $\iff f^*\equivSW f$.
\end{enumerate}
\end{proposition}

A less obvious result relates idempotency and parallelizability. 
In order to formulate this result, we need another definition.

\begin{definition}[Finite tolerance]
A {\em problem} $f:\In\IN^\IN\mto\IN^\IN$ is called {\em finitely tolerant} if there is a computable partial function $T:\In\IN^\IN\to\IN^\IN$ such that
for all $p,q\in\dom(f)$ and $k\in\IN$ with $(\forall n\geq k)(p(n)=q(n))$ it follows that $r\in f(q)$ implies $T\langle r,k\rangle\in f(p)$.
More generally, a problem $g:\In X\mto Y$ can be called {\em finitely tolerant}, if there is some finitely tolerant $f:\In\IN^\IN\mto\IN^\IN$ with $f\equivW g$.
\end{definition}

Intuitively, finite tolerance means that for two almost identical inputs and a solution for one of these inputs we can compute a solution for the other input. 
The squashing theorem relates products $g\times f$ to parallelizations $\widehat{g}$ of problems.

\begin{theorem}[Squashing theorem]
\label{thm:squashing}
For $f,g:\In\IN^\IN\mto\IN^\IN$ we obtain:
\begin{enumerate}
\item If $\dom(f)=\IN^\IN$ and $f$ is finitely tolerant, then $g\times f\leqW f\TO \widehat{g}\leqW f$.
\item If $\dom(f)=2^\IN$ and $f$ is finitely tolerant, then $g\times f\leqSW f\TO \widehat{g}\leqSW f$.
\end{enumerate}
\end{theorem} 

We obtain the following immediate corollary.

\begin{corollary}
\label{cor:squashing}
Let $f:\In\IN^\IN\mto\IN^\IN$ be finitely tolerant. Then we obtain:
\begin{enumerate}
\item For $\dom(f)=\IN^\IN$: $f$ idempotent $\iff f$ parallelizable.
\item For $\dom(f)=2^\IN$: $f$ strongly idempotent $\iff f$ strongly parallelizable.
\end{enumerate}
\end{corollary} 

Another property that turned out to be quite useful is join-irreducibility.
We recall that a problem $f$ is called {\em join-irreducible} in the lattice theoretic sense
if $f\leqW g\sqcup h$ implies $f\leqW g$ or $f\leqW h$ for all problems $g,h$. We need a countable version
of this property. For this purpose we first need to define countable coproducts.
For a sequence $(X_i)_{i\in\IN}$ of sets we define the disjoint union by $\bigsqcup_{i=0}^\infty X_i:=\bigcup_{i=0}^\infty(\{i\}\times X_i)$.
Now we can define the countable coproduct.

\begin{definition}[Countable coproduct]
Let $f_i:\In X_i\mto Y_i$ be problems for all $i\in\IN$. Then we define
$\bigsqcup_{i=0}^\infty f_i:\In\bigsqcup_{i=0}^\infty X_i\mto\bigsqcup_{i=0}^\infty Y_i$
by $\bigsqcup_{i=0}^\infty f_i(n,x):=\{n\}\times f_n(x)$.
\end{definition}

Now we are prepared to define countable irreducibility.

\begin{definition}[Countable irreducibility]
A problem $f$ is called {\em countably irreducible} if for every sequence $(g_i)_{i\in\IN}$ of problems: $f\leqW\bigsqcup_{i=0}^\infty g_i\TO(\exists i) f\leqW g_i$. 
Likewise we can define {\em strong countable irreducibility}
with $\leqSW$ in place of $\leqW$.
\end{definition}

It is clear that every countably irreducible\footnote{We note that countable irreducibility is not identical to what is sometimes called {\em ${\mathrm\sigma}$-join-irreducibility} since
the countable coproduct is not necessarily a countable supremum, as we will see in Theorem~\ref{thm:completeness}.} 
problem is join-irreducible.
Another notion that turned out to be fruitful in this context is the notion of a fractal.
Roughly speaking, a fractal is a problem that exhibits its full power even if we zoom arbitrarily deep into its domain.

\begin{definition}[Fractal]
A problem $f$ is called a {\em fractal}, if there is a problem $F:\In\IN^\IN\mto\IN^\IN$
such that $F\equivW f$ and $F|_A\equivW F$ holds for every clopen $A\In\IN^\IN$ with $A\cap\dom(F)\not=\emptyset$.
Likewise we define a {\em strong fractal} with $\equivSW$ instead of $\equivW$.
A {\em total (strong) fractal} is a (strong) fractal where $F$ can be chosen to be total.
\end{definition}

One reason why fractals are useful is captured in the following observation.

\begin{proposition}[Fractals]
\label{prop:fractal}
Every (strong) fractal is (strongly) countably irreducible.
\end{proposition}

Some natural problems in the Weihrauch lattice are densely realized in the following sense.

\begin{definition}[Densely realized]
Let $(X,\delta_X)$, $(Y,\delta_Y)$ be represented spaces.
A problem $f:\In X\mto Y$ is called {\em densely realized} if $f^{\rm r}(p)=\delta_Y^{-1}\circ f\circ\delta_X(p)$ is
dense in $\dom(\delta_Y)$ for all $p\in\dom(f\circ\delta_X)$.
\end{definition}

We note that this notion depends on the representations chosen. 
It turns out that all problems with discrete output below densely realized problems with totally represented output are computable.

\begin{proposition}[Densely realized]
\label{prop:densely-realized}
Let $f:\In X\mto Y$ be densely realized, where $Y$ is a represented space with total representation 
and let $g:\In Z\mto\IN$ be a problem.
If $g\leqW f$ holds, then $g$ is computable.
\end{proposition}

\subsubsection*{Bibliographic Remarks}

\begin{petit}
The notions of pointedness, idempotency and parallelizability were introduced by Brattka and Gherardi in~\cite{BG11}.
Dorais, Dzhafarov, Hirst, Mileti and Shafer~\cite{DDH+16} defined finitely tolerant problems and proved
the squashing theorem (a proof for Theorem~\ref{thm:squashing} exactly as stated here can be found in~\cite{Rak15}).
Countable irreducibility was first considered by Brattka, de Brecht and Pauly~\cite{BBP12},
who also implicitly defined fractals that were later used by Brattka, Gherardi and Marcone~\cite{BGM12}.
Densely realized problems have been introduced by Brattka, Hendtlass and Kreuzer~\cite{BHK17a}
and Proposition~\ref{prop:densely-realized} is due to Brattka and Pauly~\cite{BP18}.
\end{petit}

\section{Completeness, Composition and Implication}
\label{sec:structure}

Another obvious question regarding the Weihrauch lattice is whether the lattice is complete or more generally, which suprema and infima exist.
A mostly negative answer is given by the following result.

\begin{theorem}[Suprema and infima]
\label{thm:completeness}
No non-trivial countable suprema exist in the Weihrauch lattice, i.e., a sequence $(f_n)_{n\in\IN}$ of problems has
a supremum if and only if this supremum is already a supremum of $(f_n)_{n\leq k}$ for some $k\in\IN$.
Some non-trivial countable infima exist in the Weihrauch lattice, others do not exist.
\end{theorem}

In particular, the Weihrauch lattice is not complete in the lattice theoretic sense.
We can also conclude that $\bigsqcup_{n=0}^\infty f_n$ is typically not the supremum of $\{f_n:n\in\IN\}$ unless it is already a supremum
of $\bigsqcup_{n=0}^k f_n$ for some $k\in\IN$.

However, it turns out that some important suprema and infima exist in the Weihrauch lattice.
We are particularly interested in composition and implication.
The composition $f\circ g$ of problems as it has been defined in Definition~\ref{def:algebraic-operations} is not
an operation on degrees in the same sense as the other algebraic operations extend to degrees. It requires that the output
type of $g$ fits to the input type of $f$, and even if the types fit, the operation does not need to be monotone.
On the other hand, it is natural to consider a Weihrauch degree $f*g$
that captures exactly what can be achieved when one first applies $g$, possibly followed by some computation, and then one applies $f$.
That the maximal Weihrauch degree that can be built in this way always exists is the first statement of the following theorem.
The second statement captures the minimal degree $(g\to f)$ that is needed in advance of $g$ in order to compute $f$.
In some sense $(g\to f)$ measures how much harder $f$ is to compute than $g$.

\begin{theorem}[Compositional product and implication]
\label{thm:compositional-product-implication}
Let $f$ and $g$ be problems. The following Weihrauch degrees exist:
\begin{enumerate}
\item $f*g:=\max_{\leqW}\{f_0\circ g_0:f_0\leqW f,g_0\leqW g\}$ \hfill (compositional product)
\item $(g\to f):=\min_{\leqW}\{h:f\leqW g*h\}$ \hfill (implication)
\end{enumerate}
Maximum and minimum are understood with respect to $\leqW$. Only such $f_0$ and $g_0$ are considered that can be composed.
\end{theorem}

By definition $*$ and $\to$ are operations on degrees. It is easy to see that $*$ is even a monotone operation,
whereas $\to$ is antitone in the first component and monotone in the second component.
In order to prove  Theorem~\ref{thm:compositional-product-implication} it is useful to define a specific representative of the degree $f*g$
that we denote by $f\star g$.
For $F,G:\In\IN^\IN\to\IN^\IN$ we define $\langle F\times G\rangle\langle p,q\rangle:=\langle F(p),G(q)\rangle$.

\begin{definition}[Compositional product]
Let $f$ and $g$ be problems. We define\linebreak 
$f\star g:\In\IN^\IN\mto\IN^\IN$ by
$(f\star g)\langle p,q\rangle:= \langle\id\times f^\r\rangle\circ\Phi_p\circ g^\r(q)$ 
for all $p,q\in\IN^\IN$.
\end{definition}

This definition captures the intuition that in between $g$ and $f$ there is another possible computation $\Phi_p$.
Of course, this definition has not the same set-theoretic flavor as that of the other operations in Definition~\ref{def:algebraic-operations}, and it
is not a definition that we typically work with. It is mostly needed in order to prove Theorem~\ref{thm:compositional-product-implication},
and the working definition of the compositional product $f*g$ is the one given in Theorem~\ref{thm:compositional-product-implication}.
The following result captures another interesting property of $f\star g$.

\begin{proposition}
\label{prop:compositional-product-cylinder}
$f*g\equivW f\star g$ and $f\star g$ is always a cylinder. If $f$ and $g$ are fractals, then so is $f\star g$.
\end{proposition}

We can also define a strong version of the compositional product. This operation has been studied less  
and is only known to exist in specific cases. In fact, since $f\star g$ is always a cylinder, we directly
obtain the following corollary of Theorem~\ref{thm:compositional-product-implication} and Proposition~\ref{prop:compositional-product-cylinder}.

\begin{corollary}
\label{cor:strong-compositional-product}
$f\stars\! g:=\max_{\leqSW}\!\{f_0\circ g_0\!:\!f_0\leqSW\! f, g_0\leqSW\! g\}$ exists for cylinders $f,g$.
\end{corollary}

The maximum $f\stars g$ exists also in some cases where $f,g$ are not cylinders, but we do not claim that it exists in general. 
The following result summarizes some algebraic properties of compositional products and implications.

\begin{proposition}[Algebraic properties]
\label{prop:algebraic-properties}
\begin{enumerate}
\item $*$ is associative but not commutative, $\to$ is neither associative nor commutative.
\item $\stars$ is associative whenever all occurring degrees actually exist.
\end{enumerate}
\end{proposition}

The operations $+,\sqcap,\boxplus,\sqcup,\times$ and $*$ are typically ordered as given.

\begin{proposition}[Order of algebraic operations]
\label{prop:order}
We obtain:
\begin{enumerate}
\item $f+g\leqSW f\sqcap g\leqSW f\boxplus g\leqSW f\sqcup g$ and $f\times g\leqSW f\star g$ for all problems $f,g$.
\item $f\sqcup g\leqW f\times g$ and $f\boxplus g\leqSW f\times g$ for all pointed problems $f,g$.
\end{enumerate}
\end{proposition}

The following result expresses in which way compositional product and implication are adjoints of each other.

\begin{proposition}[Adjointness]
\label{prop:composition-implication}
$f\leqW g*h\iff(g\to f)\leqW h$.
\end{proposition}

In the language of lattice theory this result can be expressed such that $(\WW,\geq_{\mathrm W},*)$ is right residuated,
and the residual operation is exactly $\to$. It follows from Example~\ref{ex:CR-CN} that $(\WW,\geq_{\mathrm W},*)$ is not left residuated
and that $(\WW,\geq_{\mathrm W},\times)$ is not residuated. 
The following result expresses that the Weihrauch lattice is not residuated with respect to the lattice operations $\sqcup,\sqcap$.

\begin{theorem}[Brouwer and Heyting algebras]
\label{thm:Brouwer-Heyting}
The Weihrauch lattice $\WW$ is neither a Brouwer algebra nor a Heyting algebra.
\end{theorem}

Brouwer algebras can be seen as models of intermediate logics that are in between classical logic
and intuitionistic logic.

\subsubsection*{Bibliographic Remarks}

\begin{petit}
Theorems~\ref{thm:completeness} and \ref{thm:Brouwer-Heyting} are due to Higuchi and Pauly~\cite{HP13},
who also discussed several variants of the Weihrauch lattice in this regard.
The compositional product was introduced by Brattka, Gherardi and Marcone~\cite{BGM12}.
The implication was introduced by Brattka and Pauly~\cite{Pau16}, who also proved most other results in this section.
They also studied many further algebraic properties of the Weihrauch lattice, including distributivity laws.
The results on strong compositional products are taken from Brattka, Hendtlass and Kreuzer~\cite{BHK17a}.
\end{petit}

\section{Limits and Jumps}
\label{sec:jumps}

A map of particular importance in the Weihrauch lattice is the limit map. Given a Hausdorff space $X$, we define
the {\em limit map of the space $X$} and the {\em the limit map} (of Baire space) by
\begin{enumerate}
\item $\lim\nolimits_X:\In X^\IN\to X,(x_n)_{n\in\IN}\mapsto\lim_{n\to\infty}x_n$,
\item $\lim:\In\IN^\IN\to\IN^\IN,\langle p_0,p_1,p_2,...\rangle\mapsto\lim_{n\to\infty}p_n$.
\end{enumerate}
The domain of $\lim_X$ consists of all converging sequences in $X$.
In the special case of Baire space, we use a tupling with $p_i\in\IN^\IN$ on the input side for mere
reasons of convenience.
By $\lim_\Delta$ we denote the restriction of $\lim$ to eventually constant sequences.
It is easy to see that $\lim_\IN\equivW\lim_\Delta$.
It has been noticed that limit maps can be used to characterize limit computable functions and
functions computable with finitely many mind changes.

\begin{proposition}[Limit computability and finite mind change computability]
\label{prop:limit-computable}
For problems $f$ we obtain:
\begin{enumerate}
\item $f\leqW\lim\iff f$ limit computable.
\item $f\leqW\lim_\IN\iff f$ computable with finitely many mind changes.
\end{enumerate}
\end{proposition}

Limit computability and computability with finitely many mind changes can be defined directly
with Turing machines that allow two-way output tapes. In the case of limit computable problems
the Turing machine can change the content of each output cell finitely many times before
it has to stabilize, in the case of problems that are computable with finitely many mind changes,
the entire output has to stabilize after finitely many changes. These concepts are well-known
from learning theory.

One might ask whether $\lim_X$ for other spaces $X$ yields different classes of computable
problems, but for many spaces $X$ this is not the case. We recall that a computable metric
space $X$ is called {\em rich}, if there is a {\em computable embedding} $\iota:2^\IN\into X$, i.e., 
$\iota$ is injective, and $\iota$ and its partial inverse $\iota^{-1}$ are computable.

\begin{proposition}[Limits]
\label{prop:limits}
$\lim_X\equivSW\lim$ for all rich computable metric spaces $X$.
\end{proposition}

Examples of rich computable metric spaces are $2^\IN,\IN^\IN,\IR,\IR^\IN,[0,1],[0,1]^\IN$, etc.
This justifies also the more generic notation $\lim$ for the limit operation on Baire space.
An interesting property of $\lim_X$ is its behavior under composition.
The following result on $\lim_\IN$ can be proved with the help of Theorem~\ref{thm:independent-choice},
but is also easy to see directly. 

\begin{proposition}[Composition]
\label{prop:limN-closure}
$\lim_\IN*\lim_\IN\equivW\lim_\IN$, i.e., problems that are computable with finitely many mind changes are closed under composition.
\end{proposition}

The situation for iterations of $\lim$ is very different. 
For a problem $f$ we denote by $f^{[n]}$ the $n$--fold iteration of the compositional product of $f$ with itself, i.e.,
$f^{[0]}\equivW\id$, $f^{[1]}\equivW f$, $f^{[2]}\equivW f* f$, etc. 
By iterations of $\lim$ one climbs up the Borel hierarchy
with every further application of a limit. By ${\mathrm\Sigma^0_n}$ we denote the corresponding Borel class of subsets
of $\IN^\IN$, i.e., ${\mathrm\Sigma^0_1}$ is the class of open subsets, ${\mathrm\Sigma^0_2}$ is the class of $F_\sigma$--subsets and so forth.
A function $F:\In\IN^\IN\to\IN^\IN$ is called {\em ${\mathrm\Sigma^0_n}$--measurable}, if preimages $F^{-1}(U)$ of open sets $U$
are ${\mathrm\Sigma^0_n}$--sets relative to $\dom(F)$. Analogously, {\em effective ${\mathrm\Sigma^0_n}$--measurable} is defined
if the preimage can be uniformly computed from a description of $U$. The ${\mathrm\Sigma^0_1}$--measurable functions $F$ are exactly
the continuous ones, and the effectively ${\mathrm\Sigma^0_1}$--measurable functions $F$ are exactly the computable ones. 
We can transfer concepts of measurability to problems via realizers.

\begin{definition}[Effective Borel measurability] Let $n\geq1$.
A problem $f$ is called {\em (effectively) ${\mathrm\Sigma^0_{n}}$--measurable} if it has
a realizer with the same property.
\end{definition}

It can be proved that for computable metric spaces $X$ and $Y$ and total functions $f:X\to Y$ this yields just
the usual (effectively) ${\mathrm\Sigma^0_n}$--measurable functions as they are known in descriptive set theory~\cite{Bra05}.
The measurable problems can also easily be characterized in the Weihrauch lattice.

\begin{theorem}[Effective Borel measurability]
\label{thm:Borel}
$f\leqW\lim^{[n]}\iff f$ is effectively ${\mathrm\Sigma^0_{n+1}}$--measurable,
for all problems $f$ and $n\in\IN$.
\end{theorem}

This theorem can be relativized. We write $f\leq_{\rm W}^p g$ if $f$ is Weihrauch reducible to $g$
with respect to some oracle $p\in\IN^\IN$, which means that the reduction functions $H,K$ are computable relative to $p$.
Then $f\leq_{\rm W}^p\lim^{[n]}$ holds for some $p\in\IN^\IN$ if and only if $f$ is ${\mathrm\Sigma^0_{n+1}}$--measurable.
Theorem~\ref{thm:Borel} shows that the Weihrauch lattice yields a refinement of the effective Borel hierarchy very much in the same
way as many-one reducibility yields a refinement of the Kleene hierarchy.
We summarize some of the obvious algebraic properties of $\lim$.

\begin{proposition}
\label{prop:algebraic-lim}
$\lim$ is a cylinder, strongly parallelizable, strongly idempotent, finitely tolerant, a strong fractal and (strongly) countably irreducible. 
\end{proposition}

It is useful to know that there are many problems that are equivalent to $\lim$. We mention only a few.
By $\J:\IN^\IN\to\IN^\IN,p\mapsto p'$ we denote the Turing jump operation, which is injective as a function on Baire space. 
By $\EC:\IN^\IN\to 2^\IN$ we denote
the function that translates enumerations of sets into their characteristic functions, and by $\LPO:\IN^\IN\to\{0,1\}$ we denote the
{\em limited principle of omniscience}\footnote{$\LPO$ is also Weihrauch equivalent to the identity $\id:\IS\to\{0,1\}$.}:
\[\EC(p)(n):=\left\{\begin{array}{ll}
   1 & \mbox{if $(\exists k\in\IN)\;p(k)=n+1$}\\
   0 & \mbox{otherwise}
\end{array}\right.\mbox{ and }
\LPO(p):=\left\{\begin{array}{ll}
1 & \mbox{if $(\exists k)\; p(k)=0$}\\
0 & \mbox{otherwise}
\end{array}\right.\]

We also use $\inf,\sup:\In\IR^\IN\to\IR$.
We obtain the following result that lists some important members of the equivalence class of $\lim$.

\begin{theorem}[Limit]
\label{thm:lim}
$\lim\equivSW\inf\equivSW\sup\equivSW\J\equivSW\EC\equivSW\widehat{\LPO}\equivSW\widehat{\lim\nolimits_\IN}$.
\end{theorem}

We now use this limit operation to define the jump of a represented space.

\begin{definition}[Jump of a represented spaces]
Let $(X,\delta)$ be a represented space. Then we define its {\em jump} $(X',\delta')$ by $X':=X$
and $\delta':=\delta\circ\lim$. Likewise, $(X^{(n)},\delta^{(n)})$ denotes the $n$--fold jump.
\end{definition}

That is, in the new represented space $X'$ a name of $x$ with respect to $\delta'$ is a 
sequence that converges to a name in the sense of $\delta$. Hence, names in $(X',\delta')$
typically carry less computably accessible information than names in $(X,\delta)$.
Now the jump of a problem is just the same problem but with the input space
replaced by its jump.

\begin{definition}[Jump of a problem]
Let $f:\In X\mto Y$ be a problem. Then its {\em jump} $f':\In X'\mto Y$ is defined
to be the same problem with the modified input space $X'$.
Likewise $f^{(n)}:\In X^{(n)}\mto Y$ denotes the $n$--fold jump for $n\in\IN$ with $f^{(0)}:=f$.
\end{definition}

Since the jump $f'$ has to work with a weaker type of input information, it is typically
harder to compute $f'$ than $f$. 
The study of jumps provides one reason why it is important to keep track of strong Weihrauch reductions.
Jumps are monotone with respect to strong Weihrauch reductions, but not with respect to ordinary Weihrauch reductions in general.

\begin{proposition}[Monotonicity]
\label{prop:jump-monotone}
For all problems $f,g$ it both hold: $f\leqSW f'$ and also $f\leqSW g\TO f'\leqSW g'$.
\end{proposition}

Trivial examples such as a constant function show that $f\equivSW f'$ can happen.
For $f\leqW g$ all possible reductions between $f'$ and $g'$ can occur, the order can even be reversed, i.e., $g'\lW f'$ can happen~\cite[Figure~2]{BHK17a}.
Surprisingly, there is also a certain inverse of Proposition~\ref{prop:jump-monotone} that one can prove. 
We recall that by $p'$ we denote the Turing jump of $p$.
Using this concept we can phrase the following theorem.

\begin{theorem}[Inverting jumps]
\label{thm:jump-inversion}
$f'\leq_{\rm W}^p g'\TO f\leq_{\rm W}^{p'} g$ holds for all problems $f,g$ and $p\in\IN^\IN$. 
An analogous statement holds if Weihrauch reducibility is replaced by strong Weihrauch reducibility in both occurrences.
\end{theorem}

The property that $f\leq_{\rm W}^p g$ holds for some oracle $p$ is equivalent to the continuous
version of Weihrauch reducibility, where the two reduction functions $H,K$ just need to be continuous.
By Theorem~\ref{thm:jump-inversion} continuous separations are particularly useful, since
they automatically carry over to jumps. 

The following result summarizes some algebraic properties of the jump.

\begin{proposition}[Algebraic properties]
\label{prop:jump-algebraic}
We obtain
$f'\times g'\equivSW(f\times g)'$, $\widehat{f'\ }\equivSW\widehat{f}\,'$,\linebreak
${f'\sqcap g'}\equivSW{(f\sqcap g)'}$, $f'\sqcup g'\leqSW(f\sqcup g)'$ and ${f'}^*\leqSW {f^*}'$
for all problems $f,g$.
\end{proposition}

One can see that coproducts do not commute with jumps in general, since jumps are join-irreducible.

\begin{proposition}[Finite tolerance and fractality]
\label{prop:tolerance-fractality}
$f'$ is finitely tolerant, a strong fractal and (strongly) countably irreducible for every problem $f$.
\end{proposition}

Sometimes it is useful to have the following characterizations of the jump.

\begin{proposition}[Cylinder]
\label{prop:jump-cylinder}
$f'\equivSW f\stars\lim$, and if $f$ is a cylinder, then $f'$ is a cylinder and $f'\equivW f'\times\lim\equivW f*\lim$.
\end{proposition}

In particular, $f\stars\lim$ always exists.
We continue with a discussion of some invariant properties. 
We call a class $P$ of problems {\em invariant}, if $f\leqW g$ and $g\in P$ implies $f\in P$.
Likewise, we define {\em strong invariance}. 
We list a number of examples of (strongly) invariant properties that easily follow
from results of this section.

\begin{corollary}[Invariance]
\label{cor:invariance}
The following properties of problems are (strongly) invariant:
continuity, computability, limit computability, (effective) ${\mathrm\Sigma^0_n}$--measurability, computability with finitely many mind changes,
non-uniform computability (i.e., the class of problems that have some computable output for every computable input in the domain).
\end{corollary}

Sometimes it is also useful to use numerical quantities that are preserved
by Weihrauch reducibility. Besides the level $n$ of (effective) ${\mathrm\Sigma^0_n}$--measurability,
we can also use the number of mind changes that are required to compute a problem.
Let $\mind(f)$ denote the minimal number $n\in\IN$ that a Turing machine with two-way output needs 
in order to compute $f$ with at most $n$ mind changes on all inputs (if such a number exists).
This property is invariant in the following sense.

\begin{proposition}[Mind changes]
\label{prop:mind}
$f\leqW g\TO\mind(f)\leqW\mind(g)$ for all problems $f,g$ for which $\mind(f),\mind(g)$ exist.
\end{proposition}

Another numerical quantity that is useful as an invariant for strong Weihrauch reducibility
is the cardinality of a problem.

\begin{definition}[Cardinality]
For every problem $f:\In X\mto Y$ we denote by $\# f$ the maximal cardinality (if it exists) of a set $M\In\dom(f)$
such that $\{f(x):x\in M\}$ contains pairwise disjoint sets. 
\end{definition}

It is easy to see that the following holds.

\begin{proposition}[Cardinality]
\label{prop:cardinality}
$f\leqSW g\TO \#f\leq\#g$ for all problems $f,g$ with existing cardinality.
\end{proposition}

Every cylinder $f$ needs to satisfy $\#f\geq|\IN^\IN|$, since $\#\id=|\IN^\IN|$,
where $|X|$ denotes the {\em cardinality} of the set $X$.
For instance, it is easy to see that $\#\lim_\IN=|\IN|$ and $\#\lim_\Delta=|\IN^\IN|$.
We obtain the following.

\begin{example}
\label{ex:limit}
$\lim_\IN\equivW\lim_\Delta$ and $\lim_\IN\lSW\lim_\Delta$.
Moreover, $\lim_\Delta$ is a cylinder, whereas $\lim_\IN$ is not.
\end{example}

Next we mention that $\LPO$ is in a certain sense
the weakest discontinuous problem among all {\em single-valued} problems.

\begin{theorem}[Discontinuous single-valued problems]
\label{thm:LPO}
$\LPO\leq_{\rm W}^p f$ for some oracle $p\iff f$ is discontinuous, for $f:X\to Y$ on computable metric spaces $X,Y$.
\end{theorem}

We close this section with the following result that shows that for (certain well-behaved) linear closed operators
there is a dichotomy: either they are bounded and computable or $\lim$ is reducible to them.

\begin{theorem}[Linear operators]
\label{thm:linear}
Let $T:\In X\to Y$ be a linear closed operator on computable Banach spaces $X,Y$.
Let $(e_n)_{n\in\IN}$ be a computable sequence in $\dom(T)$ whose linear span
is dense in $\dom(T)$ and such that $(T(e_n))_{n\in\IN}$ is computable.
\begin{enumerate}
\item If $T$ is bounded, then $T$ is computable.
\item If $T$ is unbounded, then $\lim\leqW T$.
\end{enumerate}
\end{theorem}

\subsubsection*{Bibliographic Remarks}

\begin{petit}
The relation between Weihrauch reducibility and the Borel hierarchy was established by Brattka, including 
Proposition~\ref{prop:limits} and Theorem~\ref{thm:Borel}~\cite{Bra05,Bra93} (see also~\cite{BR17}).
Limit computable functions and functions that are computable by finite mind changes were originally
introduced to computable analysis by Ziegler~\cite{Zie07,Zie07a}.
The relation of these classes to the Weihrauch lattice was studied by Brattka and Gherardi~\cite{BG11,BG11a}, including Corollary~\ref{cor:invariance}
and Proposition~\ref{prop:mind} and by Brattka, de Brecht and Pauly~\cite{BBP12},
including Proposition~\ref{prop:limN-closure}. Theorem~\ref{thm:lim} collects results from the last mentioned groups of authors.
The concept of a jump of a problem was introduced by Brattka, Gherardi and Marcone~\cite{BGM12}, who also 
proved most related results that are included here, except Theorem~\ref{thm:jump-inversion} that is due to
Brattka, H\"olzl and Kuyper~\cite{BHK17},
and the proof is essentially based on a jump control theorem by Brattka, Hendtlass and Kreuzer~\cite{BHK17a}.
The number of mind changes has a topological counter part, namely the level of a problem
that has been studied by Hertling~\cite{Her96d,Her96b}; it essentially measures the level in the Hausdorff difference hierarchy.
Proposition~\ref{prop:cardinality} is due to Brattka, Gherardi and H\"olzl~\cite[Proposition~3.6]{BGH15a}.
Theorem~\ref{thm:LPO} is due to Weihrauch~\cite[Theorem~3.7]{Wei92c}.
Theorem~\ref{thm:linear} is due to Brattka~\cite[Theorem~4.3]{Bra99} and can be seen as a uniform version of the first main theorem of Pour-El and Richards~\cite{PR89}.
\end{petit}

\section{Choice}
\label{sec:choice}

The choice problem $\C_X$ of a given space $X$ is the problem of finding a point in a given closed $A\In X$.
By choosing appropriate spaces $X$ one obtains several important Weihrauch degrees. 

\begin{definition}[Choice]\label{closedchoice}
The problem $\C_X:\In\AA_-(X)\mto X,A\mapsto A$ 
with $\dom(\C_X):=\{A:A\not=\emptyset\}$ is called the {\em choice problem} of the represented space $X$.
\end{definition}

Here the description $A\mapsto A$ of the map is to be read such that on the input side $A\in\AA_-(X)$ is a point of the input space, whereas on the output side it is a subset $A\In X$ of possible results.
Typically, $X$ will be a computable metric space, and the reader can think of closed
sets being represented by enumerations of balls whose union exhausts the complement, as described in the first item of Proposition~\ref{prop:closed}.

Likewise, one can use the second characterization of Proposition~\ref{prop:closed} to define a representation of closed sets via preimages of continuous functions $f:X\to\IR$.
The uniformity statement in Proposition~\ref{prop:closed} yields the conclusion that the choice problem is nothing but the zero problem
that we introduced in Example~\ref{ex:zero}:

\begin{corollary}
\label{cor:choice-zero}
$\C_X\equivSW\Z_X$ for every computable metric space $X$.
\end{corollary}

Certain relations between spaces transfer to the corresponding
choice problems. We mention two such properties in the following result.

\begin{proposition}[Subsets and surjections]
\label{prop:injection-surjection}
Let $X,Y$ be represented spaces.
\begin{enumerate}
\item If $A\In X$ is co-c.e.\ closed, then $\C_A\leqSW\C_X$.
\item If there is a computable surjection $s:X\to Y$, then $\C_Y\leqSW\C_X$. 
\end{enumerate}
\end{proposition}

The choice problem has been studied in many variants that are typically restrictions to 
closed subsets with certain extra properties. We list a number of examples.

\begin{definition}[Variants of choice]
\begin{enumerate}
\item $\UC_X$ is $\C_X$ restricted to singletons \hfill ({\em unique choice})
\item $\ConC_X$ is $\C_X$ restricted to connected sets \hfill ({\em connected choice})
\item $\PWCC_X$ is $\C_X$ restricted to pathwise connected sets \hfill ({\em pathw.\ connected choice})
\item $\XC_X$ is $\C_X$ restricted to convex sets \hfill ({\em convex choice})
\item $\PC_X$ is $\C_X$ restricted to sets with positive measure \hfill ({\em positive choice})
\item $\AoUC_X$ is $\C_X$ restricted to sets of the form $\{x\}$ or $X$ \hfill ({\em all-or-unique choice})
\item $\ACC_X$ is $\C_X$ restricted to sets of the form $X\setminus\{x\}$ or $X$ \hfill ({\em all-or-co-unique choice})
\item $\CFC_X$ is $\C_X$ restricted to co-finite sets \hfill ({\em co-finite choice})
\end{enumerate}
\end{definition}

In some of these examples some additional structure is required on $X$.
For instance, for convex choice one would assume that $X$ is a vector space, and for positive choice
one would expect that $X$ is endowed with a fixed Borel measure.
In the case of $\IN^\IN$ and $\IR$ we assume that the product measure of the geometric probability measure on $\IN$ and the Lebesgue measure are used, respectively. 
In the case of $\UC_X$ and $\AoUC_X$ we assume that $X$ is a $T_1$--space, and in the case of $\ACC_X$ and $\CFC_X$ we assume that $X$ is endowed with a discrete topology.
The choice problem $\C_X$ is a fractal for many spaces $X$ and often a total fractal for compact $X$.

\begin{proposition}[Fractality]
\label{prop:choice-fractal}
\begin{enumerate}
\item $\C_\IN,\C_\IR,\PC_\IR$ and $\C_{\IN^\IN}$ are fractals, 
\item $\C_{2^\IN},\PC_{2^\IN},\ConC_{[0,1]}$ and $\XC_{[0,1]^{n+1}}$ are total fractals for all $n\in\IN$.
\end{enumerate}
In particular, all the mentioned problems are countably irreducible and hence join-irreducible. 
\end{proposition} 

While it is easy to see that $\C_X$ is a cylinder for many spaces $X$,
it follows from the fact that there are only countably many pairwise different sets of
positive measure that $\#\PC_{2^\IN}=\#\PC_{\IR}=\#\PC_{\IN^\IN}=\#\C_\IN=|\IN|$,
and hence all the mentioned problems are not cylinders.
It requires more sophisticated arguments to show that $\ConC_{[0,1]}$ is not a cylinder
despite the fact that $\#\ConC_{[0,1]}=|\IN^\IN|$.

\begin{proposition}[Cylinders]
\label{prop:choice-cylinder}
\begin{enumerate}
\item $\C_{2^\IN},\C_\IR,\C_{\IN^\IN}$ are cylinders,
\item $\C_\IN,\PC_{2^\IN},\PC_\IR,\PC_{\IN^\IN}$ and $\ConC_{[0,1]}$ are not cylinders.
\end{enumerate}
\end{proposition}

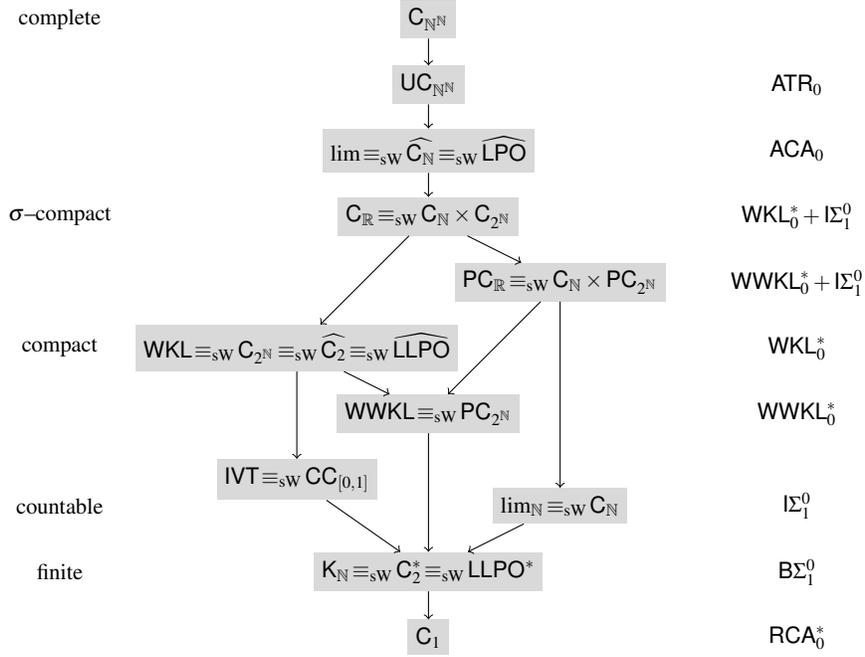
\begin{figure}[tb]
\begin{center}
\begin{tikzpicture}[scale=0.7,every node/.style={fill=black!15}]

\node (v6) at (6.5,4.5) {$\lim_\IN\equivSW\C_\IN$};
\node (v5) at (4,3.25) {$\K_\IN\equivSW\C_2^*\equivSW\LLPO^*$};
\node (v4) at (4,6.25) {$\WWKL\equivSW\PC_{2^\IN}$};
\node (v3) at (1.5,7.5) {$\WKL\equivSW\C_{2^\IN}\equivSW\widehat{\C_2}\equivSW\widehat{\LLPO}$};
\node (v30) at (1.5,5) {$\IVT\equivSW\ConC_{[0,1]}$};
\node (v2) at (4,10) {$\C_{\IR}\equivSW\C_\IN\times\C_{2^\IN}$};
\node (v20) at (6.5,8.75) {$\PC_{\IR}\equivSW\C_\IN\times\PC_{2^\IN}$};
\node (v0) at (4,11.25) {$\lim\equivSW\widehat{\C_\IN}\equivSW\widehat{\LPO}$};
\node (v00) at (4,13.75) {$\C_{\IN^\IN}$};
\node (v01) at (4,12.5) {$\UC_{\IN^\IN}$};
\node(v8) at (4,2) {$\C_1$};

\draw [->] (v2) edge (v3);
\draw [->] (v20) edge (v6);
\draw [->] (v2) edge (v20);
\draw [->] (v20) edge (v4);
\draw [->] (v6) edge (v5);
\draw [->]  (v5) edge (v8);
\draw [->] (v00) edge (v01);
\draw [->] (v01) edge (v0);
\draw [->] (v0) edge (v2);
\draw [->] (v3) edge (v4);
\draw [->] (v3) edge (v30);
\draw [->] (v30) edge (v5);
\draw [->] (v4) edge (v5);

\node[style={fill=blue!0}] at (11,2) {$\RCA_0^*$};
\node[style={fill=blue!0}] at (11,3.25) {$\B{\mathrm\Sigma^0_1}$};
\node[style={fill=blue!0}] at (11,4.5) {$\I{\mathrm\Sigma^0_1}$};
\node[style={fill=blue!0}] at (11,11.25) {$\ACA_0$};
\node[style={fill=blue!0}] at (11,12.5) {$\ATR_0$};
\node[style={fill=blue!0}] at (11,7.5) {$\WKL_0^*$};
\node[style={fill=blue!0}] at (11,10) {$\WKL_0^*+\I{\mathrm\Sigma^0_1}$};
\node[style={fill=blue!0}] at (11,8.75) {$\WWKL_0^*+\I{\mathrm\Sigma^0_1}$};
\node[style={fill=blue!0}] at (11,6.25) {$\WWKL_0^*$};

\node[style={fill=blue!0}] at (-3,7.5) {compact};
\node[style={fill=blue!0}] at (-3,10) {$\sigma$--compact};
\node[style={fill=blue!0}] at (-3,13.75) {complete};
\node[style={fill=blue!0}] at (-3,4.5) {countable};
\node[style={fill=blue!0}] at (-3,3.25) {finite};
\end{tikzpicture}
\end{center}
\ \\[-0.5cm]
\caption{Basic choice problems together with corresponding reverse mathematics systems (see subsection~\ref{subsec:reverse}) and topological properties (every arrow indicates a strong Weihrauch reduction; no additional ordinary Weihrauch reductions hold besides those that follow from transitivity. The arrows in the diagram are pointing into the direction of computations and implicit logical implications and hence in the inverse direction of the corresponding reductions.)}
\label{fig:choice}
\end{figure}

We note that there is also a choice problem $\K_X:\In\KK_-(X)\mto X,K\mapsto K$ that is called {\em compact choice}.
Unlike the other choice problems we do not just restrict $\C_X$ to compact sets here, but we also increase
the input information, i.e., the input set $K$ is actually described as a compact set.

With the help of the choice problem $\C_X$ for different spaces $X$ we obtain 
several important Weihrauch degrees. In the following result we indicate how
the most important choice problems appear naturally as upper bounds for
certain topological properties of the underlying space.
We call $X$ {\em computably countable} if there is a computable surjection $s:\IN\to X$,
and we say that a computable metric space is {\em computably $\sigma$--compact},
if there is a computable sequence $(K_i)_{i\in\IN}$ of compact sets $K_i\In X$ such that
$X=\bigcup_{i\in\IN}K_i$. 

\begin{proposition}[Spaces]
\label{prop:choice-spaces}
Let $X$ be a computable metric space.
\begin{enumerate}
\item $\C_X\leqSW\C_{\IN^\IN}$ if $X$ is complete,
\item $\C_X\leqSW\C_\IR$  if $X$ is computably $\sigma$--compact,
\item $\C_X\leqSW\C_{2^\IN}$ if $X$ is computably compact,
\item $\C_X\leqSW\C_\IN$ if $X$ is computably countable,
\item $\C_X\leqSW\K_\IN$ if $X$ is finite.
\end{enumerate}
\end{proposition}

We will discuss these cones in individual subsections below.
The diagram in Figure~\ref{fig:choice} displays several basic choice problems in the Weihrauch lattice.
The corresponding systems from reverse mathematics are discussed later in subsection~\ref{subsec:reverse}.

The reader who is mostly interested in classifications of theorems in analysis can continue reading 
in section~\ref{sec:classifications} from here on. 
In the remainder of this section we continue discussing systematically choice principles and their properties.

\subsection{Composition and Non-Determinism}

In this section we discuss a different perspective on choice that can be seen
as a type conversion and that is related to non-determinism.
Firstly, we define non-deterministic computability with some advice space $R$.

\begin{definition}[Non-deterministic computability]
\label{def:non-deterministic}
Let $(X,\delta_X)$, $(Y,\delta_Y)$ be represented spaces and $R\In\Baire$.
Then $f:\In X\mto Y$ is called {\em non-deterministically computable with advice space $R$},
if there exist computable functions
$F:\In \Baire\to\Baire$ and $S:\In\Baire\to\IS$ such that $\langle\dom(f\delta_X)\times R\rangle\In\dom(S)$ and
for each $p\in\dom(f\delta_X)$:
\begin{enumerate}
\item $R_p:=\{r\in R:S\langle p,r\rangle=0\}\not=\emptyset$,
\item $r\in R_p \TO\delta_YF\langle p,r\rangle\in f\delta_X(p)$.
\end{enumerate}
If $R=2^\IN$, then we say for short that $f$ is {\em non-deterministically} computable.
If we strengthen in this case the first condition to $\mu(R_p)>0$ with the uniform measure $\mu$,
then we say that $f$ is {\em Las Vegas computable}.
\end{definition}

Intuitively, the machine can access an arbitrary oracle $r\in R$ besides the input $p$.
Here $R_p$ is the set of successful oracles for input $p$. On input $p$ together with such successful oracles $r\in R_p$
the computable realizer $F$ produces  a correct output. On the other hand, the computable $S$ eventually rejects unsuccessful
oracles $r$, i.e, $S\langle p,r\rangle=1$ for such $r$.
The importance of non-determinism in our context is based on the following observation.

\begin{theorem}[Non-determinism]
\label{thm:non-deterministic}
$f\leqW\C_R\iff f$ is non-deterministically computable with advice space $R$, for every $R\In\IN^\IN$.
\end{theorem}

In the case of $R=\IN$ it is not too hard to see that we obtain exactly the functions
that are computable with finitely many mind changes.
We summarize some important classes of functions that can be characterized
by an appropriate version of choice.

\begin{corollary}[Notions of computability]
\label{cor:notions-computability}
Let $f$ be a problem. Then:
\begin{enumerate}
\item $f\leqW\C_\IN\iff f$ is computable with finitely many mind changes,
\item $f\leqW\C_{2^\IN}\iff f$ is non-deterministically computable,
\item $f\leqW\PC_{2^\IN}\iff f$ is Las Vegas computable.
\end{enumerate}
In particular, all the given properties of $f$ are invariant.
\end{corollary}

One benefit of characterizing the choice problem with the help
of non-deterministic computations is that it is very easy to consider
compositions of non-deterministic computations, and hence one obtains
a simple proof of the following result that is much harder to prove directly.

\begin{theorem}[Independent choice]
\label{thm:independent-choice}
$\C_R*\C_S\leqW\C_{R\times S}$ and 
$\PC_R*\PC_S\leqW\PC_{R\times S}$ for all $R,S\In\IN^\IN$.
\end{theorem}

In the case of positive choice one needs an invocation of Fubini's theorem
besides the composition of the two non-deterministic computations.
We obtain the following important corollary that, in particular, shows that the
notions of computability listed in Corollary~\ref{cor:notions-computability} are very natural.

\begin{corollary}[Composition]
\label{cor:choice-composition}
$\C_\IN,\C_{2^\IN},\C_\IR,\C_{\IN^\IN},\UC_{\IN^\IN},\PC_{2^\IN},\PC_\IR$ and $\PC_{\IN^\IN}$ 
are closed under compositional product and hence, in particular, idempotent.
\end{corollary}

\subsection{Choice on Natural Numbers}

An important equivalence class is the class of choice on natural numbers.
We summarize some of its characterizations. In particular, we use the complementary minimum function 
$\min^{\rm c}:\In\IN^\IN\to\IN,p\mapsto\min\{n\in\IN:(\forall k)\;p(k)\not=n\}$  
and the maximum function $\max:\In\IN^\IN\to\IN,p\mapsto\max\{p(n):n\in\IN\}$.

\begin{theorem}[Choice on $\IN$]
\label{thm:CN-strong}
$\UC_\IN\equivSW\C_\IN\equivSW\C_\IQ\equivSW\lim\nolimits_\IN\equivSW\min^{\rm c}\equivSW\max$.
\end{theorem}

Here $\IQ$ can be endowed with the discrete or the Euclidean topology.
The ordinary Weihrauch degree of $\C_\IN$ has some further members that
occasionally appear. 

\begin{theorem}
\label{thm:CN-weak}
$\CFC_\IN\equivW\C_\IN\equivW\lim_\Delta$ and $\CFC_\IN\lSW\C_\IN\lSW\lim\nolimits_\Delta$.
\end{theorem}

Here the strictness results follow from Proposition~\ref{prop:cardinality}
since $\#\CFC_\IN=1$, $\#\C_\IN=|\IN|$ and $\#\lim_\Delta=|\IN^\IN|$.

An important result related to choice on natural numbers shows that it can not contribute anything to the
computation of total fractals, if it is applied first (possibly followed by another problem). 

\begin{theorem}[Choice elimination]
\label{thm:CN-elimination}
$f\leqW g*\C_\IN\TO f\leqW g$, for every total fractal $f$ and every problem $g$.
\end{theorem}

The proof of this theorem is based on the Baire category theorem.
In light of Proposition~\ref{prop:choice-fractal} we obtain the following corollary.

\begin{corollary}[Separations]
\label{cor:CN-separations}
$\ConC_{[0,1]}\nleqW\C_\IN$ and $\PC_{2^\IN}\nleqW\C_\IN$.
\end{corollary}

We use the identification $n=\{0,1,...,n-1\}$ for all $n\in\IN$,
and we also consider the finite choice problems $\C_n$.
It is clear that $\C_0\equivW{\bf 0}$ and $\C_1\equivW{\bf 1}$.
The particular case of $\C_2$ is related to $\LLPO$, which is the counterpart
of the {\em lesser limited principle of omniscience} as it is known from constructive analysis.

\[\LLPO:\In2^\IN\mto\{0,1\},\LLPO(p)\ni\left\{\begin{array}{ll}
0 &\mbox{$\iff(\forall n)\;p(2n)=0$}\\
1 &\mbox{$\iff(\forall n)\;p(2n+1)=0$}
\end{array}\right.\]
with $\dom(\LLPO):=\{p\in2^\IN:p(k)\not=0$ for at most one $k\}$.

\begin{proposition}[Principles of omniscience]
\label{prop:LLPO-LPO}
$\C_2\equivSW\LLPO\lW\LPO\lW\C_\IN$.
\end{proposition}

Likewise, one can define problems $\MLPO_n$ that are equivalent to $\C_n$ and problems 
$\LPO_n$ that are equivalent to $\ACC_n$. These yield an increasing and a decreasing chain of problems,
respectively. 

\begin{proposition}[Finite choice]
\label{prop:finite-choice}
For every $n>2$ and every $p\in\IN^\IN$ we obtain\linebreak
$\ACC_\IN<_{\rm W}^p\ACC_{n+1}<_{\rm W}^p\ACC_{n}<_{\rm W}^p\ACC_2=\C_2<_{\rm W}^p\C_n<_{\rm W}^p\C_{n+1}<_{\rm W}^p\C_\IN$.
\end{proposition}

The mere fact that $\C_{n+1}\nleqW\C_n$ holds, follows since
$\mind(\C_n)=n-1$ for all $n\geq1$.
While the power of choice increases with the finite cardinality, we can
compensate cardinality by sufficiently many parallel copies of $\C_2$, as the following result shows.

\begin{theorem}[Cardinality versus products]
\label{thm:cardinality-products}
$\C_{n+1}\leqSW\C_{2}^n$ for all $n\in\IN$.
\end{theorem}

It is easy to see that also $\C_2^n\leqSW\C_{2^n}$ holds.
This implies the second equivalence in the following result.

\begin{proposition}[Compact choice]
\label{prop:KN}
$\K_\IN\equivSW\C_2^*\equivSW\C_n^*$ for all $n\geq2$.
\end{proposition}

We note that $\C_2^*\lW\C_\IN$, since $\C_\IN$ is a fractal by Proposition~\ref{prop:choice-fractal}
and hence countably irreducible.

\begin{corollary}[Compact versus closed choice]
\label{cor:KN}
$\K_\IN\lW\C_\IN$.
\end{corollary}

We use the minimum function $\min:\In\IN^\IN\to\IN,p\mapsto\min\{p(n):n\in\IN\}$
in order to express the last result of this subsection.

\begin{proposition}[Minimum]
\label{prop:min}
$\LPO^*\equivSW\min$.
\end{proposition}

\subsection{Choice on Cantor Space}

Choice on Cantor space $2^\IN$ is closely related to weak K\H{o}nig's lemma,
which states that every infinite binary tree $T\In2^*$ has an infinite path $p\in 2^\IN$ (formally, a binary tree is a subset of $2^*$ closed downward with respect to the partial order induced by the prefix relation).
By $\Tr$ we denote the set of binary trees $T\In2^*$ (represented by their characteristic
functions ${\chi_T:{2^*}\to 2}$), and by $[T]$ we denote the set of infinite paths $p\in2^\IN$ of $T$.
Now we formalize weak K\H{o}nig's lemma as the problem $\WKL:\In\Tr\mto2^\IN,T\mapsto[T]$,
where $\dom(\WKL)$ consists of all infinite binary trees.

It is well-known that the map $[.]:\Tr\to\AA_-(2^\IN),T\mapsto[T]$ is computable, surjective and
it has a computable multi-valued right-inverse. This yields the first equivalence in the following
theorem.

\begin{theorem}[Choice on Cantor space]
\label{thm:choice-cantor}
$\WKL\equivSW\!\C_{2^\IN}\!\equivSW\!\C_X\!\equivSW\!\widehat{\C_2}\!\equivSW\!\widehat{\LLPO}$ for every rich computably compact computable metric space $X$.
\end{theorem}

In particular, $\C_{2^\IN}$ is parallelizable, and the problem of finding a path in a binary tree
can be reduced to countably many binary choices, i.e., $\WKL\leqW\widehat{\C_2}$.
In the following corollary we list the choice problem for some important examples of rich computably compact computable metric spaces.

\begin{corollary}
\label{cor:C2N-equiv}
$\C_{2^\IN}\equivSW\C_{[0,1]^\IN}\equivSW\C_{[0,1]^n}$ for all $n\geq1$.
\end{corollary}

Similarly as for choice on natural numbers there is a choice elimination result for $\C_{2^\IN}$.
This result can be proved using compactness properties.

\begin{theorem}[Choice elimination]
\label{thm:C2N-elimination}
$f\leqW\C_{2^\IN}*g\TO f\leqW g$ for every single-valued problem $f:X\to Y$ with a 
computable metric space $Y$ and every $g$.
\end{theorem}

This result can also be generalized to admissibly represented spaces $Y$. 
We obtain the following important special case.

\begin{corollary}[Single-valuedness]
\label{cor:C2N-single-valued}
$f\leqW\C_{2^\IN}\TO f$ computable, for all single-valued problems $f:X\to Y$ with a computable metric space $Y$.
\end{corollary}

In particular this applies to $f=\UC_{2^\IN}$. Since $\lim_\IN$ is a single-valued problem
in the equivalence class of $\C_\IN$, we also get the following conclusion.

\begin{corollary}[Separation]
\label{cor:C2N-separation}
$\C_\IN\nleqW\C_{2^\IN}$.
\end{corollary} 

The so-called weak weak K\H{o}nig's lemma $\WWKL$ is $\WKL$ restricted to trees
$T$ such that $\mu([T])>0$. It is easy to see that it is equivalent to $\PC_{2^\IN}$.

\begin{theorem}[Positive choice on Cantor space]
\label{thm:PC2N}
$\WWKL\equivSW\PC_{2^\IN}\equivSW\PC_{[0,1]}$.
\end{theorem}

One can use a result of Jockusch and Soare~\cite[Theorem~5.3]{JS72} that essentially shows that $\WKL$ cannot be computed with an advice set of positive measure, 
in order to separate $\PC_{2^\IN}$ and $\C_{2^\IN}$.

\begin{proposition}[Positive choice versus choice]
\label{prop:PC-C}
$\PC_{2^\IN}\lW\C_{2^\IN}$.
\end{proposition}

Since $\C_2\leqW\PC_{2^\IN}\lW\C_{2^\IN}\equivW\widehat{\C_2}$, it is clear that $\PC_{2^\IN}$
is not parallelizable.

\begin{corollary}[Parallelizability]
\label{cor:PC2N-parallelizability}
$\PC_{2^\IN}$ is not parallelizable and $\widehat{\PC_{2^\IN}}\equivSW\C_{2^\IN}$.
\end{corollary}

Also quantitative versions of $\WWKL$ have been considered, and by $\varepsilon\dash\WWKL$
for $\varepsilon\in(0,1)$
we denote $\WWKL$ restricted to trees with $\mu([T])>\varepsilon$.

\begin{theorem}[Quantitative $\WWKL$]
\label{thm:quantitative-WWKL}
$\varepsilon\dash\WWKL\leqW\delta\dash\WWKL\iff\varepsilon\geq\delta$.
\end{theorem}

We continue with the discussion of further special versions of choice related to $\C_{2^\IN}$.
It is not obvious at all that connected choice $\ConC_{[0,1]^n}$ is in the same equivalence class
as $\C_{2^\IN}$ from dimension $n\geq2$. 

\begin{theorem}[Connected choice]
\label{thm:CC}
$\ConC_{[0,1]^n}\equivSW\PWCC_{[0,1]^{n+1}}\equivSW\C_{2^\IN}$ for $n\geq2$.
\end{theorem}

The map $A\mapsto (A\times[0,1]\times\{0\})\cup(A\times A\times[0,1])\cup([0,1]\times A\times\{1\})$
shows that one can map each closed subset $A\In[0,1]$ to a pathwise connected closed 
subset $B\In[0,1]^3$, and given a point in the latter set one can reconstruct a point in the former set.
This proves the previous statement on $\PWCC_{[0,1]^n}$ and $\ConC_{[0,1]^n}$ for $n\geq3$.
Only the two-dimensional case needs a more sophisticated argument, and in the case of $\PWCC_{[0,1]^2}$
the Weihrauch degree is not known.

\begin{problem}[Pathwise connected choice]
\label{prob:PWCC}
Does $\PWCC_{[0,1]^2}\equivW\C_{2^\IN}$ hold?
\end{problem}

The one-dimensional case of connected choice yields the degree of the intermediate value theorem.

\begin{theorem}[Intermediate value theorem]
\label{thm:IVT}
$\ConC_{[0,1]}\equivSW\IVT$.
\end{theorem}

We mention a fact that was already stated in Example~\ref{ex:idempotent-parallelizable}.

\begin{theorem}[Idempotency]
\label{thm:CC-not-idempotent}
$\ConC_{[0,1]}$ is not idempotent.
\end{theorem}

While connected choice is very stable with respect to the dimension of the space, this is not so for
convex choice as the following result shows.

\begin{theorem}[Convex choice]
\label{thm:convex-choice}
$\XC_{[0,1]^n}\lW\XC_{[0,1]^{n+1}}$ for all $n\in\IN$.
\end{theorem}

Convex choice is not closed under composition, as the following result shows.

\begin{theorem}[Composition of convex choice]
\label{thm:XC-composition}
$\XC_{[0,1]^n}$ is not closed under compositional product $*$ and
${\XC_{[0,1]}*\XC_{[0,1]}\nleqW\XC_{[0,1]^n}}$ for all $n\geq1$.
\end{theorem}

We mention that compact choice $\K_X$ does not lead to anything new on rich computable metric spaces.

\begin{theorem}[Compact choice]
\label{thm:KX}
$\K_X\equivSW\C_{2^\IN}$ for all rich computable metric spa\-ces $X$.
\end{theorem}

In particular, this implies $\K_{2^\IN}\equivSW\K_\IR\equivSW\K_{\IN^\IN}\equivSW\C_{2^\IN}$.
We close this section by mentioning that $\ConC_{[0,1]}$ and $\PC_{2^\IN}$ are both upper bounds
of $\K_\IN$.

\begin{proposition}[Upper bound of compact choice]
\label{prop:KN-upper}
$\K_\IN\leqW\ConC_{[0,1]}\sqcap\PC_{2^\IN}$.
\end{proposition}

\subsection{Choice on Euclidean Space}

In this section we discuss $\C_\IR$ and related problems. 
The basic observation is that $\C_\IR$ can be described with the help of $\C_{2^\IN}$ and $\C_\IN$
in several different ways.

\begin{theorem}[Choice on Euclidean space]
\label{thm:CR}
$\C_\IR\equivSW\C_{\IR^n}\equivSW\C_{2^\IN\times\IN}\equivSW$\linebreak 
${\C_{2^\IN}\times\C_\IN}\equivSW\C_{2^\IN}\star\C_\IN\equivSW\C_\IN\star\C_{2^\IN}$
for all $n\geq1$.
\end{theorem}

The results regarding $\star$ follow with the help of Theorem~\ref{thm:independent-choice}.
We have deliberately used the symbol $\star$ and not $*$, since the degrees with $\star$
are cylinders and hence we obtain strong equivalences.
Theorem~\ref{thm:CR} shows that Theorem~\ref{thm:C2N-elimination} is applicable to $\C_\IR$, and
we obtain the following conclusion.

\begin{corollary}[Single-valuedness]
\label{cor:CR-single}
$f\leqW\C_\IR\TO f\leqW\C_\IN$, for all single-valued problems $f:X\to Y$ with a computable metric space $Y$.
\end{corollary}

This result applies in particular to $\UC_\IR$ and implies $\UC_\IR\equivW\C_\IN$.
Now we discuss an important upper bound on $\C_\IR$. The low basis theorem of Jockusch and Soare
states that every computable infinite binary tree has a low path. 
We recall that $p\in\IN^\IN$ is called {\em low} if $p'\leqT\emptyset'$ holds, i.e., 
if the halting problem relative to $p$ is not more difficult than the ordinary halting problem.
Lowness is represented by the problem $\Low:=\J^{-1}\circ\lim$ since $p$ is low
if and only if there is a computable $q$ such that $p=\Low(q)$. 
It is clear that $\Low\lW\lim$ (since $\J^{-1}$ is computable and not every limit computable $p$ is low).
The following result can be seen as a uniform version of the low basis theorem.

\begin{theorem}[Uniform low basis theorem]
\label{thm:low-basis}
$\C_\IR\lSW\Low$.
\end{theorem}

The strictness follows for instance from Corollary~\ref{cor:CR-single} since $\Low$ is single-valued.
We mention an interesting algebraic example of how infima and suprema of the degrees of $\C_\IN$ and $\C_{2^\IN}$  
interact. 

\begin{example}
\label{ex:CR-CN}
We obtain
\begin{enumerate}
\item $(\C_{2^\IN}\sqcap\C_\IN)*(\C_{2^\IN}\sqcup\C_\IN)\equivW(\C_{2^\IN}\sqcap\C_\IN)\times(\C_{2^\IN}\sqcup\C_\IN)\equivW\C_{2^\IN}\sqcup\C_\IN$.
\item
$(\C_{2^\IN}\sqcup\C_\IN)*(\C_{2^\IN}\sqcap\C_\IN)\equivW\C_{2^\IN}*\C_\IN\equivW\C_{2^\IN}\times\C_\IN$.
\end{enumerate}
\end{example}

We note that $\C_{2^\IN}\sqcup\C_\IN\lW\C_{2^\IN}\times\C_\IN\equivW\C_\IR$ since the right-hand degree is join-irreducible
and since $\C_{2^\IN}$ and $\C_\IN$ are incomparable.
We formulate a counterpart of Theorem~\ref{thm:CR} for $\PC_\IR$.

\begin{theorem}[Positive choice on Euclidean space]
\label{thm:PCR}
$\PC_\IR\equivSW\PC_{2^\IN\times\IN}\equivSW$\linebreak 
${\PC_{2^\IN}\times\C_\IN}\equivW\PC_{2^\IN}*\C_\IN\equivW\C_\IN*\PC_{2^\IN}$
for all $n\geq1$.
\end{theorem}

We mention that in this case we cannot simply replace $*$ by $\star$ and $\equivW$ by $\equivSW$, since
$\PC_\IR$ is not a cylinder.
We note that $\C_\IN\lW\C_\IR\lW\lim\equivW\widehat{\C_\IN}$ implies the following.

\begin{corollary}[Parallelizability]
\label{cor:CR-parallelizability}
$\C_\IR$ and $\PC_\IR$ are not parallelizable, and we obtain $\widehat{\PC_\IR}\equivSW\widehat{\C_\IR}\equivSW\lim$.
\end{corollary}

\subsection{Choice on Baire Space}

Choice on Baire space is the upper bound of all choice problems of complete computable metric spaces.
In fact, we obtain the following.

\begin{theorem}[Non $\sigma$--compact spaces]
\label{thm:non-Ks}
$\C_{\IN^\IN}\equiv_{\rm W}^p\C_X$ for some oracle $p\in\IN^\IN$ if $X$ is a separable
complete metric space that is not $\sigma$--compact.
\end{theorem}

We list a number of choice problems that fall into the equivalence class of $\C_{\IN^\IN}$.
We assume that these spaces are represented as computable metric spaces in the standard way.

\begin{theorem}[Baire space]
\label{thm:CNN-equiv}
$\C_{\IN^\IN}\equivSW\C_{\IR^\IN}\equivSW\C_{\IR\setminus\IQ}\equivSW\C_{\ell_p}\equivSW\C_{\CC[0,1]}$ for all computable $p\geq1$.
\end{theorem}

Also the single-valued problems below $\C_{\IN^\IN}$ have a very natural characterization.

\begin{theorem}[Single-valuedness]
\label{thm:CNN-single-valued}
$f\leqW\C_{\IN^\IN}\iff f$ is effectively Borel measurable, for $f:X\to Y$ on complete computable metric spaces.
\end{theorem}

Similarly as Theorem~\ref{thm:Borel} this result can be relativized.
We mention that it is easy to see that $\C_{\IN^\IN}$ is parallelizable.

\begin{proposition}[Parallelizability]
\label{prop:CNN-parallelizability}
$\C_{\IN^\IN}$ is strongly parallelizable.
\end{proposition}

We briefly mention $\UC_{\IN^\IN}$, the unique version of choice on Baire space.
It is easy to see that $\lim\lW\UC_{\IN^\IN}$ holds. It follows from a basis theorem of Kreisel
that $\UC_{\IN^\IN}$ is strictly weaker than $\C_{\IN^\IN}$.

\begin{proposition}[Unique choice]
\label{pop:UCNN}
$\lim^{(n)}\lW\UC_{\IN^\IN}\lW\C_{\IN^\IN}$ for all $n\in\IN$.
\end{proposition}

We close with the following characterization of positive choice on Baire space.

\begin{theorem}[Positive choice]
\label{thm:PCNN}
$\PC_{\IN^\IN}\equivSW\PC_\IR$.
\end{theorem}

\subsection{Jumps of Choice}

In order to characterize the jump of choice we need the {\em cluster point problem}
$\CL_X:\In X^\IN\mto X,(x_n)_{n\in\IN}\mapsto\{x\in X: x$ is a cluster point of $(x_n)_{n\in\IN}\}$.
This problem fully characterizes the jump of $\C_X$ on computable metric spaces $X$.
If we restrict $\CL_X$ to such sequences $(x_n)_{n\in\IN}$ whose range $\{x_n:n\in\IN\}$
has a compact closure, then we denote it as $\BWT_X:\In X^\IN\mto X$ since it can be seen
as a problem that realizes the Bolzano-Weierstra\ss{} theorem. 

\begin{theorem}[Jump of choice]
\label{thm:jumps-choice}
$\C_X'\equivSW\CL_X$ and $\K_X'\equivSW\BWT_X$ for all computable metric spaces $X$.
\end{theorem}

$\BWT_2=\CL_2$ is also known as the {\em infinite pigeonhole problem}.
Many properties of problems can be transferred to jumps. 
However, this is often not so for properties that involve compositional products.
We recall that with $f^{[n]}$ we denote the $n$--fold compositional product of $f$ with itself
and by $f^{(n)}$ the $n$--fold jump.

\begin{theorem}[Composition]
\label{thm:composition-jump}
We obtain:
\begin{enumerate}
\item $\C_\IN'*\C_\IN'\equivW\C_\IN'$.
\item $\C_{2^\IN}'*\C_{2^\IN}'\equivW\C_{2^\IN}''$ and  more generally ${\C_{2^\IN}'}^{[n]}\equivW\C_{2^\IN}^{(n)}$ for all $n\geq1$.
\item $\PC_{2^\IN}'*\PC_{2^\IN}'\equivW\PC_\IR'*\PC_\IR'\equivW\PC_\IR'$.
\end{enumerate}
\end{theorem}

It is perhaps surprising that compositions behave very differently in the probabilistic case
and in the non-probabilistic case.
The difference between $\C_{2^\IN}$ and $\PC_{2^\IN}$ is also underlined
by the third statement in the following result that strengthens the negative statement of Proposition~\ref{prop:PC-C}.

\begin{theorem}[Separations]
\label{thm:jump-separations}
We obtain for all $n\in\IN$:
\begin{enumerate}
\item $\C_2^{(n+1)}\nleqW\lim^{(n)}$.
\item $\LPO^{(n)}\nleqW\C_{2^\IN}^{(n)}$.
\item $\C_{2^\IN}\nleqW\PC_{2^\IN}^{(n)}$.
\end{enumerate}
\end{theorem}

It follows from the first statement that $\C_2\leqW\PC_{2^\IN}\leqW\C_{2^\IN}$ all climb up the Borel hierarchy one step
with every jump. This statement even holds relative to any oracle, and hence 
$\C_2^{(n)}$ is not ${\mathrm\Sigma^0_{n+1}}$--measurable.
We obtain the following alternating hierarchies.

\begin{theorem}[Alternating hierarchies]
\label{thm:alternating-hierarchies}
For all $n\in\IN$ we obtain
\begin{enumerate}
\item $\C_2^{(n)}\lW\LPO^{(n)}\lW \C_2^{(n+1)}$,
\item $\K_\IN^{(n)}\lW\C_\IN^{(n)}\lW\K_\IN^{(n+1)}$,
\item $\C_{2^\IN}^{(n)}\lW\lim^{(n)}\lW\C_{2^\IN}^{(n+1)}$.
\end{enumerate}
Analogous statements hold with $\leqSW$ in place of $\leqW$.
\end{theorem}

Choice on Baire space is an example of a choice problem that is stable under jump.

\begin{theorem}[Baire space]
\label{thm:CNN}
$\C_{\IN^\IN}'\equivSW\C_{\IN^\IN}$ and $\UC_{\IN^\IN}'\equivSW\UC_{\IN^\IN}$.
\end{theorem}

The following example shows that a straightforward jump inversion theorem does
not hold in the Weihrauch lattice.

\begin{example}
\label{ex:jump-inversion}
$\lim\lW\C_{2^\IN}'\sqcup\C_\IN'$, but there is no problem $f$ with 
$f'\equivW\C_{2^\IN}'\sqcup\C_\IN'$.
\end{example}

The latter holds since $f'$ is join-irreducible and $\C_{2^\IN}'$ and $\C_\IN'$ are incomparable.
While Theorem~\ref{thm:jump-separations} shows that $\C_{2^\IN}$ has no jump of positive
choice as upper bound, this is different for $\ConC_{[0,1]}$ as the following result shows. 

\begin{proposition}[Upper bounds]
\label{prop:upper-bounds-CC-CN}
$\ConC_{[0,1]}\sqcup\C_\IN\leqW\PC_{2^\IN}'\sqcap\C_\IN'$.
\end{proposition}

This result is contrasted by $\ConC_{[0,1]}\nleqW\PC_{2^\IN}\sqcup\C_\IN$,
which holds because $\ConC_{[0,1]}$ is not reducible to any of the problems on the right-hand side and since it
is join-irreducible by Propositions~\ref{prop:choice-fractal} and \ref{prop:fractal}.

For $\frac{1}{2}\dash\WWKL$ we can improve the statement that follows from Corollary~\ref{cor:C2N-single-valued}
by the following result that can be proved with a majority vote argument.

\begin{theorem}[Single-valuedness]
\label{thm:WWKL-single-valued}
$f\leqW\frac{1}{2}\dash\WWKL^{(n)}\TO f$ computable, for all single-valued problems $f:X\to Y$ with a computable metric space $Y$
and $n\in\IN$.
\end{theorem}

We note that $\frac{1}{2}\dash\WWKL$ cannot be replaced by $\WWKL$ in this result, since $\lim_\IN\leqW\WWKL'$
holds as a consequence of Proposition~\ref{prop:upper-bounds-CC-CN}.

\subsection{All-or-Unique Choice}

We briefly discuss all-or-unique choice in this section.
The problem $\AoUC_{[0,1]}$ is located between $\LLPO$ and $\LPO$ and related to
{\em robust division} that is defined as the problem $\RDIV:[0,1]\times[0,1]\mto[0,1]$
with $\RDIV(x,y):=\{\frac{x}{\max(x,y)}\}$ if $y\not=0$ and $\RDIV(x,y):=[0,1]$ otherwise.
We now obtain the following characterization.

\begin{proposition}[All-or-unique choice]
\label{prop:AoUC}
$\C_2\lW\AoUC_{[0,1]}\equivSW\RDIV\lW\LPO$.
\end{proposition}

In some respects $\AoUC_{[0,1]}$ is closer to $\C_2$ than to $\LPO$, 
at least with respect to the following upper bounds.

\begin{theorem}[Upper bound]
\label{thm:AoUC-upper-bound}
$\AoUC_{[0,1]}\leqW\ConC_{[0,1]}\sqcap\PC_{2^\IN}$.
\end{theorem}

In the diagram in Figure~\ref{fig:choice} $\AoUC_{[0,1]}$ would be in a similar
position as $\K_\IN$, however it is incomparable to $\K_\IN$ since $\AoUC_{[0,1]}$ is countably irreducible
and $\K_\IN\nleqW\LPO$.
We continue with a number of separation results that involve $\AoUC_{[0,1]}$.

\begin{theorem}[Separation]
\label{thm:AoUC-separation}
\begin{enumerate}
\item $\XC_{[0,1]}*\AoUC_{[0,1]}\nleqW\XC_{[0,1]^n}$ for all $n\in\IN$.
\item $\C_2*\AoUC_{[0,1]}\nleqW\AoUC_{[0,1]}^*$.
\item $\C_2\times\AoUC_{[0,1]}\nleqW\ConC_{[0,1]}$.
\end{enumerate}
\end{theorem}

These separation results have a number of interesting consequences.
The first statement implies Theorem~\ref{thm:XC-composition}, and the third statement implies Theorem~\ref{thm:CC-not-idempotent}.
Since $\C_2\leqW\AoUC_{[0,1]}$ we can conclude the following from the second statement.

\begin{corollary}[Composition]
\label{cor:AoUC-composition}
$\AoUC_{[0,1]}^*$ is not closed under compositional product.
\end{corollary}

Surprisingly, a composition of $\AoUC_{[0,1]}^*$ with itself yields a new problem that 
is closed under compositional product.

\begin{theorem}[Double composition]
\label{thm:AoUC-composition}
We obtain:\\
$\AoUC_{[0,1]}^**\AoUC_{[0,1]}^*\equivW\AoUC_{[0,1]}^**\AoUC_{[0,1]}^**\AoUC_{[0,1]}^*$.
\end{theorem}

\subsubsection*{Bibliographic Remarks}

\begin{petit}
The study of choice problems in the Weihrauch lattice has been started by Gherardi and Marcone~\cite{GM09} and Brattka and Gherardi~\cite{BG11,BG11a}.
Indirectly, Weihrauch~\cite{Wei92c} already studied choice problems in form of versions of $\MLPO$ and $\LPO$; Propositions~\ref{prop:LLPO-LPO} and \ref{prop:finite-choice} 
are due to him. Theorem~\ref{thm:cardinality-products} is due to Pauly~\cite{Pau10}.
Non-deterministically computable functions have been introduced to computable analysis by
Ziegler~\cite{Zie07,Zie07a}, and the relation to choice problems has been established by Brattka, de Brecht and Pauly~\cite{BBP12},
who also started to study choice problems more systematically. Many results in this section are due to them.
The study of positive choice was initiated by Brattka and Pauly~\cite{BP10} and more systematically continued by Brattka, Gherardi and H\"olzl~\cite{BGH15a}.
This subject was independently studied by Dorais, Dzhafarov, Hirst, Mileti and Shafer~\cite{DDH+16}, who also introduced the quantitative version
of $\WWKL$ and proved Theorem~\ref{thm:quantitative-WWKL}.
Connected choice was mostly studied by Brattka, Le Roux, Miller and Pauly~\cite{BLRMP18} and convex choice by Le Roux and Pauly~\cite{LRP15a},
who also proved Theorem~\ref{thm:CN-elimination}.
The study of jumps of choice is due to Brattka, Gherardi and Marcone~\cite{BGM12}, and the statement on $\K_X$ in Theorem~\ref{thm:jumps-choice}
is due to Brattka, Cettolo, Gherardi, Marcone and Schr\"oder~\cite{BCG+17}.
The statement on positive choice in Theorem~\ref{thm:composition-jump} is due to Bienvenu and Kuyper~\cite{BK17}.
Pauly started the study of all-or-unique choice~\cite{Pau10,Pau11}, and some of the separation results in this 
regard are due to Kihara and Pauly~\cite{KP16a}. 
Theorem~\ref{thm:XC-composition} is due to Kihara \cite{Kih16a}.
\end{petit}

\section{Classifications}
\label{sec:classifications}

In this section we present results on the classification of theorems. Most of these theorems originate from analysis. 
We interpret theorems as problems as explained after Definition~\ref{def:solutions}.
For many theorems one can derive upper bounds using the following observation.

\begin{theorem}[Upper bounds]
\label{thm:upper-bounds}
Let $X,Y$ be represented spaces and $A\In X\times Y$ co-c.e.\ closed.
If 
$(\forall x\in X)(\exists y\in Y)\;(x,y)\in A$
holds, then the corresponding problem $F:X\mto Y,x\mapsto\{y\in Y:(x,y)\in A\}$ satisfies $F\leqW\C_Y$.
\end{theorem}

In combination with Proposition~\ref{prop:choice-spaces} one can thus derive upper bounds on theorems
by exploiting topological properties of $Y$.
We essentially group our classifications according to related choice problems. 

The equivalence class of choice on natural numbers contains many theorems
that are typically proved with the help of the Baire category theorem.

\begin{theorem}[Choice on the natural numbers]
\label{thm:CN}
The following are all Weihrauch equivalent to each other:
\begin{enumerate}
\item Choice on natural numbers $\C_{\IN}$.
\item The Baire category theorem $\BCT_1$.
\item Banach's inverse mapping theorem $\BIM_{\ell_2,\ell_2}$.
\item The open mapping theorem for $\ell_2$.
\item The closed graph theorem for $\ell_2$.
\item The uniform boundedness theorem on non-singleton computable Banach spaces.
\item The Lebesgue covering lemma for $[0,1]$.
\item The partial identity from continuous function to analytic functions.
\end{enumerate}
\end{theorem}

In most cases these theorems are interpreted as problems in a straightforward way.
We only provide some examples and refer the reader to the references for exact definitions.
For instance, Banach's inverse mapping theorem on computable Banach spaces $X,Y$ is formalized as 
$\BIM_{X,Y}:\In\CC(X,Y)\to\CC(Y,X),T\mapsto T^{-1}$, 
restricted to bijective, linear, bounded $T$.
We always obtain $\BIM_{X,Y}\leqW\C_\IN$, and in the case of $X=Y=\ell_2$ the theorem
actually attains the maximal complexity. For finite-dimensional $X,Y$ it is, however, computable.
Similar remarks apply in the cases of the open mapping theorem and the closed graph theorem.
The Baire category theorem and the uniform boundedness theorem are even equivalent 
to $\C_\IN$ for all complete computable metric spaces and non-singleton computable Banach spaces, respectively.

In the case of some theorems it can happen that one logical formulation of the theorem
and the contrapositive formulation carry different computational content. In such a situation it might
not always be clear, which form is more natural, and perhaps both forms have applications.
Such an example is the Baire category theorem that we can formalize at least in two ways.
By $A^\circ$ we denote the {\em interior} of the set $A$.

\begin{enumerate}
\item $\BCT_0:\In\AA_-(X)^\IN\mto X,(A_n)_{n\in\IN}\mapsto\{x\in X:x\not\in\bigcup_{n=0}^\infty A_n\}$,\\
        with $\dom(\BCT_0):=\{(A_n)_{n\in\IN}:A_n^\circ=\emptyset\}$.
\item $\BCT_1:\In\AA_-(X)^\IN\mto\IN,(A_n)_{n\in\IN}\mapsto\{n\in\IN:A_n^\circ\not=\emptyset\}$,\\
        with $\dom(\BCT_1):=\{(A_n)_{n\in\IN}:X=\bigcup_{n=0}^\infty A_n\}$.
\end{enumerate}

While $\BCT_1$ is in the equivalence class of $\C_\IN$, 
it is easy to see that $\BCT_0$ is computable.
Nevertheless the jump of $\BCT_0$ has interesting applications that we mention below. 
Similarly to the Baire category theorem, also the Heine-Borel covering theorem can be formalized in at least 
two ways. Here $\OO(X)$ denotes the set of open subsets of $X$ seen as the complements of the elements of $\AA_-(X)$, i.e., every open set is represented by an enumeration of basic open balls whose union coincides with the set.

\begin{enumerate}
\item $\HBC_0:\In\OO([0,1])^\IN\mto \IN,(U_n)_{n\in\IN}\mapsto\{k\in\IN:[0,1]\In\bigcup_{n=0}^k U_n\}$,\\
        with $\dom(\HBC_0):=\{(U_n)_{n\in\IN}:[0,1]\In\bigcup_{n=0}^\infty U_n\}$.
\item $\HBC_1:\In\OO([0,1])^\IN\mto [0,1],(U_n)_{n\in\IN}\mapsto\{x\in [0,1]:x\not\in\bigcup_{n=0}^\infty U_n\}$,\\
        with $\dom(\HBC_1):=\{(U_n)_{n\in\IN}:(\forall k)\;[0,1]\not\In\bigcup_{n=0}^k U_n\}$.
\end{enumerate}

Once again, it is easy to see that $\HBC_0$ is computable, and $\HBC_1$ is in the equivalence class
of choice on Cantor space.

\begin{theorem}[Choice on Cantor space]
\label{thm:C2N}
The following are all strongly Weihrauch equivalent to each other:
\begin{enumerate}
\item Choice on Cantor space $\C_{2^\IN}$.
\item Weak K\H{o}nig's lemma $\WKL$.
\item The Hahn-Banach theorem.
\item The Heine-Borel covering theorem $\HBC_1$.
\item The theorem of the maximum $\MAX$.
\item The Brouwer fixed point theorem $\BFT_n$ for dimension $n\geq2$.
\item The Brouwer fixed point theorem $\BFT_\infty$ for the Hilbert cube $[0,1]^\IN$.
\item Finding connectedness components of sets $A\In[0,1]^n$ for $n\geq1$.
\item The parallelization $\widehat{\IVT}$ of the intermediate value theorem.
\item Determinacy of Gale-Stewart games in $2^\IN$ with closed winning sets.
\end{enumerate}
\end{theorem}

In the case of the Hahn-Banach theorem the underlying separable Banach space is part of the input information.
No space of maximal complexity is known in this case. For certain spaces (such as computable Hilbert spaces) the
Hahn-Banach theorem is computable.  
For two further theorems mentioned above we provide formalizations as problems.

\begin{enumerate}
\item $\MAX:\CC[0,1]\mto\IR,f\mapsto\{x\in [0,1]:f(x)=\max f([0,1])\}$.
\item $\BFT_n:\CC([0,1]^n,[0,1]^n)\mto[0,1]^n,f\mapsto\{x\in[0,1]^n:f(x)=x\}$.
\end{enumerate}

The Brouwer fixed point theorem of dimension $n=1$ is equivalent to the
intermediate value theorem, i.e., $\BFT_1\equivSW\IVT\equivSW\ConC_{[0,1]}$.
We note that classifications such as the one in Theorem~\ref{thm:C2N} lead to simple
proofs of classically known non-uniform results in computable analysis. We mention some examples.

\begin{corollary}[Non-uniform results]
\label{cor:C2N}
\begin{enumerate}
\item There exists an infinite binary tree without computable paths (Kleene~\cite{Kle52a}).
\item There is a computable function $f:[0,1]\to\IR$ that attains its maximum only at non-computable points $x\in[0,1]$ 
        (Lacombe~\cite[Theorems VI and VII]{Lac57} and Specker~\cite{Spe59}).
\item There is a computable function $f:[0,1]^2\to[0,1]^2$ that has no computable fixed point $x\in[0,1]^2$ (Orevkov~\cite{Ore63} and Baigger~\cite{Bai85}).
\item There is a computable sequence $(f_n)_{n\in\IN}$ of functions $f_n:[0,1]\to\IR$ with \linebreak
        ${f_n(0)\cdot f_n(1)}<0$ for all $n\in\IN$
         such that there is no computable sequence $(x_n)_{n\in\IN}$ with $f_n(x_n)=0$ (Pour-El and Richards~\cite[Example 8a]{PR89}).
\end{enumerate}
\end{corollary}

Once one has one of these negative results, all the others follow immediately by Theorem~\ref{thm:C2N}.
On the other hand, also positive non-uniform results can be derived from Theorem~\ref{thm:C2N}.
For instance, every computable function $f:[0,1]^2\to[0,1]^2$ has a low fixed point.
Such non-uniform results hold analogously for other classifications presented here, but we are not going to discuss them in detail.

A theorem that is often proved with the help of the Brouwer fixed point theorem
is the Nash equilibria existence theorem. Its computational content is significantly
weaker than that of the Brouwer fixed point theorem.

\begin{theorem}[All-or-unique choice]
\label{thm:AoUC}
The following are strongly Weihrauch equivalent to each other:
\begin{enumerate}
\item The finite parallelization $\AoUC_{[0,1]}^*$ of all-or-unique choice. 
\item The Nash equilibria existence theorem $\NASH$.
\end{enumerate}
\end{theorem}

Here $\NASH_{n,m}:\IR^{m\times n}\times\IR^{m\times n}\mto\IR^n\times\IR^m$ is the
map that maps a bi-matrix game $(A,B)$ to a pair of strategies that form a Nash equilibrium of $(A,B)$,
and $\NASH:=\bigsqcup_{n,m\in\IN}\NASH_{n,m}$.

The equivalence class of choice on Euclidean space contains a theorem
that we mention in the following result without further definitions. 

\begin{theorem}[Choice on Euclidean space]
\label{thm:CR-thm}
The following are Weihrauch equivalent to each other:
\begin{enumerate}
\item Choice on Euclidean space $\C_{\IR}$.
\item Frostman's lemma on the existence of measures.
\end{enumerate}
\end{theorem}

The Vitali covering theorem is a theorem that has even been studied in three different
logical versions. We consider $\Int:=(\IQ^2)^\IN$ as the set of sequences $\II=(I_n)_{n\in\IN}$
of rational intervals $I_n=(a,b)$. We say that $\II$ is a {\em Vitali cover} of a set $A\In\IR$ if
for every $x\in A$ and $\varepsilon>0$ there is some $n\in\IN$ with $x\in I_n$ and $\diam(I_n)<\varepsilon$.
We write $\JJ\prefix\II$ if $\JJ$ is a subsequence of $\II$ of pairwise disjoint intervals. We consider the following three formalizations
of the Vitali covering theorem:

\begin{enumerate}
\item $\VCT_0: \In\Int\mto\Int,\II\mapsto\{\JJ: \JJ\prefix\II\mbox{ with }\mu([0,1]\setminus\bigcup\JJ)=0\}$ and\linebreak
	  $\dom(\VCT_0)$ contains all $\II\in\Int$ that are Vitali covers of $[0,1]$.
\item $\VCT_1:\In\Int\mto[0,1],\II\mapsto[0,1]\setminus\bigcup\II$
	  and $\dom(\VCT_1)$ contains all $\II\in\Int$ that are Vitali covers of $\bigcup\II$ and without a $\JJ\prefix\II$ with $\mu([0,1]\setminus\bigcup\JJ)=0$.
\item $\VCT_2 :\In\Int\mto[0,1],\II\mapsto\left\{x\in[0,1]:(\exists\varepsilon>0)(\forall n)(x\not\in I_n\mbox{ or }\diam(I_n)\geq\varepsilon)\right\}$ and
        $\dom(\VCT_2)$ contains all $\II$ without a $\JJ\prefix\II$ with $\mu([0,1]\setminus\bigcup\JJ)=0$.
\end{enumerate}

It turns out that $\VCT_0$ is computable and $\VCT_1$ and $\VCT_2$ are equivalent
to different versions of positive choice. 

\begin{theorem}[Positive choice]
\label{thm:PC}
The following are all strongly Weihrauch equivalent to each other:
\begin{enumerate}
\item Positive choice on Cantor space $\PC_{2^\IN}$.
\item Weak weak K\H{o}nig's Lemma $\WWKL$.
\item The Vitali covering theorem $\VCT_1$.
\end{enumerate}
The following are strongly Weihrauch equivalent to each other:
\begin{enumerate}
\item Positive choice on Euclidean space $\PC_{\IR}$.
\item The Vitali covering theorem $\VCT_2$.
\end{enumerate}
\end{theorem}

Convex choice is equivalent to the Browder-G\"ohde-Kirk fixed point theorem
that is formalized as 
$\BGK_K:\In\CC(K,K)\mto K,f\mapsto\{x\in K:f(x)=x\}$,
where $K\In H$ is compact and convex, $H$ is a computable Hilbert space
and $\dom(\BGK_K)$ consists of all non-expansive continuous maps $f:K\to K$.
More general versions of the theorem have been studied, but for simplicity
we state only this basic result.

\begin{theorem}[Convex choice]
\label{thm:XC}
Let $H$ be a computable Hilbert space and $K\In H$ convex and computably compact.
The following are Weihrauch equivalent to each other:
\begin{enumerate}
\item Convex choice $\XC_{K}$.
\item The Browder-G\"ohde-Kirk fixed point theorem $\BGK_K$ on $K$.
\end{enumerate}
\end{theorem}

Another important equivalence class is that of the limit map.

\begin{theorem}[The limit]
\label{thm:lim-thm}
The following are Weihrauch equivalent to each other:
\begin{enumerate}
\item The limit map $\lim$ on Baire space (or every other rich computable metric space).
\item The monotone convergence theorem $\MCT:\In\IR^\IN\to\IR,(x_n)_{n\in\IN}\mapsto\sup_{n\in\IN}x_n$.
\item The operator of differentiation $d:\In\CC[0,1]\to\CC[0,1],f\mapsto f'$.
\item The Fr{\'e}chet-Riesz representation theorem for $\ell_2$.
\item The Radon-Nikodym theorem.
\item The parallelization $\widehat{\BIM}$ of Banach's inverse mapping theorem.
\item Finding a basis of a countable vector space.
\item Finding a connected component of a countable graph. 
\item The partial identity from infinitely differentiable functions  to Schwartz functions.
\end{enumerate}
\end{theorem}

Of course, the Banach inverse mapping theorem can be replaced by any other problem
from Theorem~\ref{thm:CN}. 
We mention that the reduction $\lim\leqW d$ follows easily with Theorem~\ref{thm:linear}.
Several theorems also fall into the equivalence class
of the jump of choice on Cantor space. Here $\KL$ is defined as $\WKL$ but for finitely
branching trees $T\In\IN^*$.

\begin{theorem}[Jump of choice on Cantor space]
\label{thm:JC2N}
The following are all strongly Weihrauch equivalent to each other:
\begin{enumerate}
\item The jump $\C_{2^\IN}'$ of choice on Cantor space.
\item K\H{o}nig's lemma $\KL$.
\item The Bolzano-Weierstra\ss{} theorem $\BWT_\IR$ on Euclidean space.
\item The Arzel{\'a}-Ascoli theorem for functions $f:[0,1]\to[0,1]$.
\item Determinacy of Gale-Stewart games in $2^\IN$ with winning sets that are differences of open sets.
\end{enumerate}
\end{theorem}

A natural problem that is known to be equivalent to higher jumps of choice on Cantor space
is the parallelization of Ramsey's theorem. We summarize some results on this theorem.
$\RT^k_n:k^{[\IN]^n}\mto2^\IN$ denotes the problem that maps every coloring $c:[\IN]^n\to k$ (of
the $n$--element subsets of $\IN$ with $k$ colors) to an infinite set $H\In\IN$ that is homogeneous for $c$.

\begin{theorem}[Ramsey's theorem]
\label{thm:Ramsey}
$\C_2^{(n)}\lW\RT_k^n\lW\RT_{k+1}^n\lW\C_{2^\IN}^{(n)}$ for all $n,k\geq2$.
The reductions also hold in the case $n=1$, but the first one is not strict in this case.
\end{theorem}

This result can be proved with the help of the squashing theorem (Theorem~\ref{thm:squashing}).
Since $\C_{2^\IN}^{(n)}$ is the parallelization of $\C_2^{(n)}$ we obtain the following corollary.

\begin{corollary}[Ramsey's theorem]
\label{cor:Ramsey}
$\C_{2^\IN}^{(n)}\equivW\widehat{\RT_k^n}$ for all $n\geq1$ and $k\geq2$.
\end{corollary}

Higher levels of the Weihrauch lattices are not yet all too well explored. 
This is currently a topic of further research, and we mention one result along these lines.

\begin{theorem}[Choice on Baire space]
\label{thm:CNN-thm}
The following are Weihrauch equivalent to each other:
\begin{enumerate}
\item Choice on Baire space $\C_{\IN^\IN}$.
\item The perfect subtree theorem.
\end{enumerate}
\end{theorem}

At the end of this section we demonstrate how some problems from computability theory can be classified 
in the Weihrauch lattice. We consider in particular the following (for $X\In\IN$ with at least two elements
and a standard numbering $\varphi^p$ of the computable functions on natural numbers relative to $p$):

\begin{enumerate}
\item $\DNC_X:\IN^\IN\mto\IN^\IN,p\mapsto\{q\in X^\IN:(\forall n)\;\varphi^p_n(n)\not=q(n)\}$.
\item $1\dash\GEN:2^\IN\mto2^\IN,p\mapsto\{q:q$ is $1$--generic relative to $p\}$.
\item $\MLR:2^\IN\mto2^\IN,p\mapsto\{q:q$ is Martin-L\"of random relative to $p\}$.
\item $\PA:2^\IN\mto2^\IN,p\mapsto\{q:q$ is of PA degree relative to $p\}$.
\item $\COH:(2^\IN)^\IN\mto2^\IN,(R_i)_{i\in\IN}\mapsto\{A:A$ is cohesive for $(R_i)_{i\in\IN}\}$.
\end{enumerate}

The first observation is that $\DNC_n$ is just the parallelization of $\ACC_n$.

\begin{theorem}[Diagonal non-computability]
\label{thm:DNC}
$\DNC_n\equivSW\widehat{\ACC_n}$ for all $n\geq2$ and $n=\IN$. 
\end{theorem}

The jump $\BCT_0'$ of the computable version of the Baire category theorem $\BCT_0$
is closely related to $1$--genericity and the problem ${\mathrm\Pi}^0_1\G$ of ${\mathrm\Pi^0_1}$--genericity
that we do not define here.

\begin{theorem}[Genericity]
\label{thm:genericity}
$1\dash\GEN\lW\BCT_0'\equivSW{\mathrm\Pi^0_1}\G\lW\Low$.
\end{theorem}

We note that all the problems from computability theory mentioned here, except $\DNC_n$, are densely realized.
Hence we can apply Proposition~\ref{prop:densely-realized}.

\begin{proposition}
\label{prop:computability-densely}
$1\dash\GEN$, $\MLR$, $\PA$, $\COH$ and $\BCT_0'$ are densely realized and hence $\ACC_\IN$ and $\C_2$ are not
Weihrauch reducible to any of them.
\end{proposition}

This means that these problems are very different from all the theorems from analysis mentioned above
that are all above $\C_2$ in the Weihrauch lattice. 
Hence, it is interesting that some of these densely realized problems can be characterized as implications (i.e., as ``quotients'') of problems above $\C_2$.

\begin{theorem}[Randomness, Peano arithmetic, cohesiveness]
\label{thm:MLR-PA-COH}
We obtain:
\begin{enumerate}
\item $\MLR\equivW(\C_\IN\to\WWKL)\equivW(\C_\IN\to\PC_{2^\IN})$.
\item $\PA\equivW(\C_\IN'\to\WKL)\equivW(\C_\IN'\to\C_{2^\IN})$.
\item $\COH\equivW(\lim\to\KL)\equivW(\lim\to\C_{2^\IN}')$.
\end{enumerate}
\end{theorem}

We close this section by mentioning that one can apply results from computability theory
such as the theorem of van Lambalgen to conclude that some of the above mentioned problems are closed under composition.

\begin{proposition}
\label{prop:MLR-GEN-composition}
$\MLR$ and $1\dash\GEN$ are closed under compositional product $*$.
\end{proposition}

\subsubsection*{Bibliographic Remarks}

\begin{petit}
The intermediate value theorem and theorems from functional analysis were studied by Brattka and Gherardi~\cite{BG11a,Bra09,Bra06}.
The Lebesgue covering lemma for $[0,1]$ has been classified by Brattka, Gherardi and H\"olzl~\cite{BGH15a}.
The equivalence of the Hahn-Banach theorem and weak K\H{o}nig's lemma was proved by Gherardi and Marcone~\cite{GM09,Bra08b}. 
The Brouwer fixed point theorem and the problem of finding connectedness components was studied by Brattka, Le Roux, Miller and Pauly~\cite{BLRMP18}.
The theorem of the maximum has been studied by Brattka~\cite{Bra16} and the Fr\'echet-Riesz representation theorem by Brattka and Yoshikawa~\cite{Bra16,BY06}.
The classification of the Nash equilibria existence theorem is due to Pauly~\cite{Pau10,Pau11}.
Frostman's lemma was studied by Pauly and Fouch\'e~\cite{PF17} and Vitali's covering theorem by
Brattka, Gherardi, H\"olzl and Pauly~\cite{BGHP17}.
The Browder-G\"ohde-Kirk fixed point theorem was classified by Neumann~\cite{Neu15}.
The operator of differentiation was studied by von Stein~\cite{Ste89}, and the degree of many operations on sets
that are not mentioned here were classified by Brattka and Gherardi~\cite{BG09}.
The identities to analytic and Schwartz functions have been classified by Pauly and Steinberg~\cite{PS18}.
The analysis of the Radon-Nikodym theorem is due to Hoyrup, Rojas and Weihrauch~\cite{HRW12}. 
The problems of finding a basis of a countable vector space and of finding a connected component of a countable graph
were studied by Gura, Hirst and Mummert~\cite{GHM15,HM17}.
The Bolzano-Weierstra\ss{} theorem was studied by Brattka, Gherardi and Marcone~\cite{BGM12}.
K\H{o}nig's lemma was studied by Brattka and Rakotoniaina~\cite{BR17} and Gale-Stewart games by Le Roux and Pauly~\cite{LRP15}.
Ramsey's theorem was studied in the Weihrauch lattice
by Dorais, Dzhafarov, Hirst, Mileti and Shafer~\cite{DDH+16,Dzh15,Dzh16}, by Brattka and Rakotoniaina~\cite{BR17,Rak15},
by Patey~\cite{Pat16a} and by Hirschfeldt and Jockusch~\cite{HJ16}.
The uniform content of problems for partial and linear orders that are closely related to Ramsey's theorem for pairs
was studied by Astor, Dzhafarov, Solomon and Suggs~\cite{ADSS17}.
The classification of the prefect subtree theorem was initiated by Marcone~\cite{BKM+16}, and the result mentioned here is unpublished.
The results on problems from computability theory including a systematic study of the Baire category theorem
are due to Brattka, Hendtlass and Kreuzer~\cite{BHK17a,BHK18} and Brattka and Pauly~\cite{BP18}.
The problem of diagonally non-computable functions was also studied by Higuchi and Kihara~\cite{HK14a}.
\end{petit}

\section{Relations to Other Theories}
\label{sec:relations}

In this section we discuss very briefly the relation between Weihrauch complexity and other theories
and we provide some further references.

\subsection{Linear Logic}
\label{subsec:linear-logic}

There is an apparent similarity between some algebraic operations on problems and the resource-oriented
interpretation of some logical operations in (intuitionistic) linear logic that was noticed early on. Table~\ref{tab:linear-logic} provides
a dictionary on these relations.

\begin{table}[htb]
\begin{center}
\begin{tabular}{ll}
{logical operation in linear logic\ } & {algebraic operation on problems}\\\hline
$\otimes$ multiplicative conjunction  & $\times$ product \\
$\&$ additive conjunction                & $\sqcup$ coproduct\\
$\oplus$ additive disjunction           & $\sqcap$ infimum\\
\rotatebox[origin=c]{180}{$\&$} multiplicative disjunction         & $+$ sum\\
$!$ bang		                                    & $\widehat{\ }$ parallelization, $^*$ finite parallelization
\end{tabular}
\end{center}
\caption{Linear logic versus the algebra of problems}
\label{tab:linear-logic}
\end{table}
However, it seems that other algebraic operations on problems, such as the compositional product $*$, do not have
any obvious counterpart in the standard approach to linear logic, but could be seen as a non-commutative conjunction.
There does not seem to be any straightforward interpretation
of the Weihrauch lattice as a model for (intuitionistic) linear logic. 

Several researchers have independently noticed that G\"odel's Dialectica interpretation has some formal similarity
to Weihrauch reducibility. This observation has not yet been formally exploited.

\subsection{Medvedev Lattice and Many-One and Turing Semilattices}
\label{subsec:Medvedev}

The Medvedev lattice has also been considered as a calculus of problems. Here problems are understood to
be subsets $A,B\In\IN^\IN$ of Baire space and $A$ is called {\em Medvedev reducible} to $B$, in symbols $A\leq_{\rm s} B$,
if there exists a partial computable function ${F:\In\IN^\IN\to\IN^\IN}$ with $B\In\dom(F)$ such that $F(B)\subseteq A$. 
The supremum operation of this lattice is defined by $A\oplus B:=\langle A,B\rangle$ and the infimum operation 
by $A\otimes B:=0A\cup 1B$. 

The relation between the Weihrauch lattice and the Medvedev lattice can be expressed from both perspectives:
\begin{enumerate}
\item The Medvedev lattice is a special case of the Weihrauch lattice for problems $f:\In\IN^\IN\mto\IN^\IN$ that are constant.
\item The Weihrauch lattice is a generalization of the Medvedev lattice for ``relativized" problems $A_p\In\IN^\IN$ that depend on a parameter $p\in\IN^\IN$.
 \end{enumerate}

This point of view translates into a formal embedding of the Medvedev preorder into the Weihrauch preorder (the first mentioned one in Theorem~\ref{thm:Medvedev}). 
In fact, we can embed the Medvedev lattice also order reversing into the Weihrauch lattice,
and we list both embeddings here.

\begin{theorem}[Embedding the Medvedev lattice]
\label{thm:Medvedev}
Let $A,B\In\IN^\IN$.
\begin{enumerate}
\item $c_A:\IN^\IN\mto\IN^\IN,p\mapsto A$ satisfies $A\leq_{\rm s}B\iff c_A\leqW c_B$ with $c_{A\oplus B}\equivW c_A\times c_B$ and $c_{A\otimes B}\equivW c_A\sqcap c_B$.
\item $d_A:\In\IN^\IN\to\IN^\IN,p\mapsto \widehat{0}$ with $\dom(d_A):=A$ satisfies $A\leq_{\rm s}B\iff d_B\leqW d_A$ with $d_{A\oplus B}\equivW d_A\sqcap d_B$ and $d_{A\otimes B}\equivW d_A\sqcup d_B$.
\end{enumerate}
In both cases all Weihrauch reductions and equivalences can be replaced by strong ones, in which case $\sqcup$ has to be replaced by $\boxplus$ in 2.
\end{theorem}

The second reverse embedding is even a lattice embedding since it preserves suprema and infima in the reverse order. 
The first embedding is also a lattice embedding if considered as an embedding into the parallelized Weihrauch degrees.
These embeddings were studied by Brattka and Gherardi~\cite{BG11}, Higuchi and Pauly~\cite{HP13} and Dzhafarov~\cite{Dzh18}.
Since the Turing degrees and the enumeration degrees can be embedded into the Medvedev lattice, it follows that they can also
be embedded into the Weihrauch lattice via the above mentioned embeddings.

We note that the Medvedev lattice has been used by Downey, Greenberg, Jock\-usch, Milans, Lewis and others~\cite{DGJM11,JL13} in order to study problems
from computability theory, such as $\MLR, 1\dash\GEN, \PA$ and $\DNC_n$ in their unrelativized form (for computable inputs).
The advantage of the Weihrauch lattice is that these problems can be studied in this lattice together with problems such as
$\WKL$ and $\WWKL$ that depend on parameters (i.e., the input tree) in an essential way.
Finally, we mention that the Muchnik lattice, which is the non-uniform counterpart of the Medvedev lattice has also been 
used to classify problems, see Simpson~\cite{Sim15} .

Also the many-one semilattice can be embedded into the Weihrauch lattice, albeit in a slightly less natural way than the Turing semilattice.
The construction starts with a non-canonical choice of two Turing incomparable points. As usual we denote {\em many-one reducibility}
between sets $A,B\In\IN$
by $\leq_{\rm m}$ and we recall that $A\oplus B:=\{2n:n\in A\}\cup\{2n+1:n\in B\}$ is the supremum with respect to many-one reducibility.

\begin{proposition}[Embedding of the many-one semilattice]
\label{prop:many-one}
Let $p,q\in\IN^\IN$ be Turing incomparable and, for $A\In\IN$, define
$m_A : \mathbb{N} \to \{p,q\}$ by $m_A(n) = p:\iff n \in A$. 
Then we obtain $A\leq_{\rm m}B\iff m_A\leqW m_B$ and $m_{A\oplus B}\equivW m_A\sqcup m_B$ for all $A,B\In\IN$,
i.e., $A\mapsto m_A$ is a join-semilattice embedding.
\end{proposition}

\subsection{Reverse Mathematics}
\label{subsec:reverse}

Reverse mathematics is a proof theoretic approach that aims to classify theorems according to axioms that are needed to prove
these theorems in second-order arithmetic~\cite{Sim09}. Many theorems from various areas of mathematics have been classified
in this approach. Most axiom systems that are used in reverse mathematics have counterparts in the Weihrauch lattice (see also Figure~\ref{fig:choice}):

\begin{itemize}
\item $\B{\mathrm\Sigma^0_n}$ (${\mathrm\Sigma^0_n}$--boundedness): $\B{\mathrm\Sigma^0_2}$ is equivalent to the {\em regularity principle} 
        $\R{\mathrm\Sigma^0_1}$ over a very weak system~\cite{HP93}, and it
        corresponds to $\K_\IN'$ by Theorem~\ref{thm:jumps-choice}. Hence $\B{\mathrm\Sigma^0_n}$ can be seen as counterpart of $\K_\IN^{(n-1)}$.
\item $\I{\mathrm\Sigma^0_n}$ (${\mathrm\Sigma^0_n}$--induction) is equivalent to the {\em least number principle} $\Low{\mathrm\Pi^0_n}$ over a very weak 
         system~\cite{HP93}.
         $\Low{\mathrm\Pi^0_1}$ directly translates into $\min^{\rm c}$ as a problem and hence $\I{\mathrm\Sigma^0_n}$ corresponds to $\C_\IN^{(n-1)}$ by Theorem~\ref{thm:CN}.
\item $\RCA_0^*$ (recursive comprehension) stands for the usual system $\RCA_0$ but with $\I{\mathrm\Sigma^0_0}$ instead of $\I{\mathrm\Sigma^0_1}$. 
        It corresponds to $\C_1$ (the computable problems).
\item $\WKL_0^*$ and $\WWKL_0^*$, by which we mean $\RCA_0^*$ plus weak K\H{o}nig's lemma and weak weak K\H{o}nig's lemma, respectively,  
         correspond directly to the problems $\WKL$ and $\WWKL$.
\item $\ACA_0$ (arithmetic comprehension) corresponds to the problems $\lim$ (and its finite compositions $\lim^{[n]}$ with $n\in\IN$).
         Sometimes, a uniform version $\ACA_0'$ of $\ACA_0$ is used~\cite{Hir15}, which corresponds to $\bigsqcup_{n\in\IN}\lim^{[n]}$.
\item $\ATR_0$ (arithmetical transfinite recursion) corresponds to $\UC_{\IN^\IN}$ and $\C_{\IN^\IN}$ (see Theorem~\ref{thm:CNN-thm}). This topic is still very much research in progress.
\end{itemize}
Counterparts in the Weih\-rauch lattice of higher systems such as ${\mathrm\Pi^1_1}\dash\C\A_0$ (${\mathrm\Pi^1_1}$--comprehension) have not yet been systematically studied.
By Theorem~\ref{thm:alternating-hierarchies} we have $\K_\IN^{(n)}\lW\C_\IN^{(n)}\lW\K_\IN^{(n+1)}$ in analogy to 
$\B{\mathrm\Sigma^0_n}\leftarrow\I{\mathrm\Sigma^0_n}\leftarrow\B{\mathrm\Sigma^0_{n+1}}$.

Reverse mathematics is based on a proof theoretic approach, whereas classifications
in the Weihrauch lattice are based on a computational approach.
Besides this we note the following distinguishing features: 
\begin{enumerate}
\item {\em Resource sensitivity}: classifications in reverse mathematics do not distinguish between a single, a finite number of consecutive applications
         or a finite number of parallel applications of a theorem, since classical logic is used (opposed to linear logic).
\item {\em Uniformity}: classifications in reverse mathematics only capture the non-uniform content of problems, i.e., the way output parameters
         depend on input parameters in the worst case. Again this is due to the usage of classical logic (opposed to intuitionistic logic).
\end{enumerate}
For instance, a number of theorems that are non-uniformly computable in the sense that there is a computable output for every
computable input are provable over $\RCA_0$ in reverse mathematics, even though they are not computable in a uniform way.
This includes the intermediate value theorem $\IVT$, the Baire category theorem $\BCT_1$ and others. 
Due to the lack of uniformity reverse mathematics can also not distinguish between theorems and their
contrapositive forms. For instance the version $\HBC_0$ of the Heine-Borel covering theorem is computable, while $\HBC_1\equivW\WKL$. 
In reverse mathematics, the Heine-Borel theorem is equivalent to $\WKL_0$ over $\RCA_0$ irrespectively of whether we consider
the analogue of $\HBC_0$ or $\HBC_1$. In other words:
classifications in reverse mathematics automatically capture the most complicated contrapositive form. 
         
It is remarkable that despite these explicable differences most classifications in the Weihrauch lattice can be seen 
as uniform and resource sensitive refinements of classifications in reverse mathematics.
This seems to confirm a ``computations as proofs'' paradigm (opposed to the well-known ``proofs as computations'' paradigm in intuitionistic logic).
         
\subsection{Constructive Reverse Mathematics}
\label{subsec:constructive}

Constructive reverse mathematics, as proposed by Ishihara~\cite{Ish06}, classifies problems in the Bishop approach
to constructive analysis that is based on intuitionistic logic.
Due to the usage of intuitionistic logic this approach is fully uniform,
but it is even less resource sensitive compared to classical reverse mathematics.
This is due to the fact that typically the {\em axiom of countable choice} can be used freely, which amounts
to a free usage of parallelization in the Weihrauch lattice. In this sense, the Weihrauch complexity approach
is closer to a hypothetical version of constructive reverse mathematics with intuitionistic linear logic. 
The classifications in constructive reverse mathematics are captured by the equivalence to certain
constructively unacceptable principles:

\begin{enumerate}
\item $\LLPO$ (the {\em lesser limited principle of omniscience}) is the theorem that corresponds to our problem $\LLPO$.
         In presence of countable choice it corresponds to $\WKL$ by Theorem~\ref{thm:choice-cantor}.
\item $\LPO$ (the {\em limited principle of omniscience}) is the theorem that corresponds to our problem $\LPO$.
        In presence of countable choice and due to the availability of composition it corresponds to $\lim^{(n)}$ with $n\in\IN$ by Theorem~\ref{thm:lim}
        (and hence to $\ACA_0$ in classical reverse mathematics).
\item $\MP$ ({\em Markov's principle}), ${\mathsf{BD}}\dash\IN$ (the {\em boundedness problem}) and some other principles that are rejected in 
        constructive analysis correspond to computable (and hence continuous) problems in the Weihrauch lattice.
\end{enumerate}

In conclusion, this means that the Weihrauch complexity approach is finer than constructive reverse mathematics in terms of resource sensitivity,
but coarser when it comes to distinctions that are based on computable principles such as $\MP$ and ${\mathsf{BD}}\dash\IN$.
In order to translate these heuristic observations into formal theorems, one needs to fix an axiomatic framework for constructive analysis.
Some results in this direction have been obtained by Kuyper~\cite{Kuy17}.

\subsection{Other Reducibilities}
\label{subsec:reducibilities}

Hirschfeldt and Jockusch~\cite{Hir15,HJ16,Sol16} have introduced a number of further reducibilities that are related to Weihrauch reducibility.
For one, there are non-uniform versions of Weihrauch reducibility and strong Weihrauch reducibility, which are called 
{\em computable reducibility} and {\em strong computable reducibility}, in symbols $\leq_{\rm c}$ and $\leq_{\rm sc}$,
as well as a reducibility $\leq_\omega$ that is based on Turing ideals.
On the other hand, they introduced a concept of {\em generalized Weihrauch reducibility} that has a built-in closure under composition.
This operation can be formalized as a closure operator $f\mapsto f^\diamond$ in the Weihrauch lattice.
Likewise, Brattka and Gherardi~\cite{BG11} and Higuchi and Pauly~\cite{HP13} studied variants of Weihrauch reducibility with a 
built-in parallelization.
These and further reducibilities allow to interpolate between Weihrauch complexity and reverse mathematics in the sense
that one can choose a reduction that captures a particular degree of uniformity and resource sensitivity (see Figure~\ref{fig:reducibility}).

\begin{figure}[htb]
\begin{center}
\begin{tikzpicture}[scale=.4,auto=left]
\useasboundingbox  rectangle (29,4.5);

\node at (6,4) {$f\leqSW g$};
\node at (13,4) {$f\leqW g$};
\node at (20,2) {$f\leq_{\rm c}g$};
\node at (13,2) {$f\leq_{\rm sc}g$};
\node at (27,2) {$f\leq_{\rm\omega} g$};
\node at (20,4) {$f\leq_{\rm gW} g$};

\node (v1) at (8,4) {};
\node (v2) at (11,4) {};
\node (v3) at (15,4) {};
\node (v4) at (18,4) {};
\node (v5) at (15,2) {};
\node (v6) at (18,2) {};
\node (v7) at (11,2) {};
\node (v8) at (22,4) {};
\node (v10) at (22,2) {};
\node (v9) at (25,2) {};
\draw [->] (v1) edge (v2);
\draw [->] (v3) edge (v4);
\draw [->] (v5) edge (v6);
\draw [->] (v1) edge (v7);
\draw [->] (v8) edge (v9);
\draw [->] (v10) edge (v9);
\draw [->] (v3) edge (v6);

\node at (2,4) {\footnotesize uniform};
\node at (8,0.5) {\footnotesize resource sensitive};
\node at (2,2) {\footnotesize non-uniform};
\node at (24.5,0.5) {\footnotesize closed under composition};
\end{tikzpicture}
\end{center}
\ \\[-0.8cm]
\caption{Implications between notions of reducibility}
\label{fig:reducibility}
\end{figure}
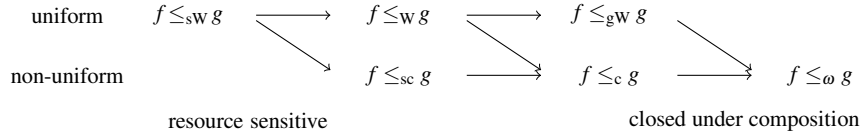

Yet another related reducibility that originates from descriptive set theory is {\em Wadge reducibility}, which is defined via
preimages. Given $A,B\In\IN^\IN$, we say that $A$ is {\em Wadge reducible} to $B$, in symbols $A\leqW B$, if there exists
a continuous function $f:\IN^\IN\to\IN^\IN$ such that $A=f^{-1}(B)$. Hence, Wadge reducibility is the (topological) analogue of 
many-one reducibility on Baire space. Weihrauch reducibility can be seen as a (computable) analogue of this reduction
for multi-valued functions.

In early work by Weihrauch~\cite{Wei92a,Wei92c} and by Hertling~\cite{Her96d} mostly
the continuous version of Weihrauch reducibility was considered.
In particular, Hertling~\cite{Her96d} completely characterized continuous
(strong) Weihrauch and Wadge degrees of certain functions with
discrete image in terms of preorders on labeled forests.
Kudinov, Selivanov and Zhukov~\cite{KSZ10} and Hertling and Selivanov~\cite{HS14a}
have studied the decidability and complexity of some initial
segments of these preorders. Such preorders and versions of Weihrauch reducibility have also been used in descriptive set theory, e.g., by Carroy~\cite{Car13a}.

\subsection{Descriptive Set Theory}
\label{subsec:descriptive}

Descriptive set theory studies the complexity of subsets of and functions between separable complete metric spaces. 
Wadge reducibility has been established as a critical tool here, and reasonable classes of subsets are typically closed downwards under Wadge reducibility. 
For functions, Weihrauch reducibility can play the analogous role. 
As demanded by Moschovakis~\cite{Mos10a}, this treatment covers both the effective and the non-effective case simultaneously, 
with the former implying the latter via relativization.

Many typical classes of functions even have complete problems under Weihrauch reducibility. 
Theorem~\ref{thm:Borel} provides an example, another one is related to $({\mathrm\Sigma^0_{n+2}},{\mathrm\Sigma^0_{n+2}})$--measurability, which is defined such that preimages of ${\mathrm\Sigma^0_{n+2}}$--sets
are ${\mathrm\Sigma^0_{n+2}}$--sets (see Pauly and de Brecht~\cite{PdB14} and Kihara~\cite{Kih15}).

\begin{theorem}[Effective $({\mathrm\Sigma^0_{n+2}},{\mathrm\Sigma^0_{n+2}})$--measurability]
\label{thm:Delta}
$f\leqW\C_\IN^{(n)}\iff f$ is effectively {$({\mathrm\Sigma^0_{n+2}},{\mathrm\Sigma^0_{n+2}})$--measurable}, for all $f:X\to Y$ on
complete computable metric spaces $X,Y$ with $n=0$ and for $f:\IN^\IN\to\IN^\IN$ with $n\in\IN$.
\end{theorem}

There is a subtle but crucial issue with relativization here: 
Relativizing the theorem covers the case where the preimage map of $f$ from ${\mathrm\Sigma^0_{n+2}}$--sets to ${\mathrm\Sigma^0_{n+2}}$--sets is continuous, 
rather than merely being well-defined. For $n\in\{0,1\}$ the theorem of Jayne and Rogers~\cite{JR82} and a theorem by Semmes~\cite{Sem09} 
show that these cases are equivalent. For $n > 1$, the question whether the cases are equivalent is open and equivalent to the generalized 
conjecture of Jayne and Rogers. This is discussed in some more detail in~\cite{Pau14}.

Weihrauch complete problems for function classes correspond to game characterizations in descriptive set theory. 
A general account of the latter is provided by Motto Ros~\cite{MRos12}, and the link to Weihrauch reducibility is made by Nobrega and Pauly~\cite{NP17}.

More generally, the theory of Weihrauch degrees is closely linked to a programme to extend descriptive set theory from Polish spaces to larger classes of spaces, 
such as represented spaces. Such an endeavor was called for and started by Selivanov~\cite{Sel04a}. 
De Brecht introduced the quasi-Polish spaces~\cite{dBre13} and demonstrated that many results from descriptive set theory remain valid in this setting. 
A further extension is possible using the formalism of jump operators (de Brecht~\cite{dBre14}) or computable endofunctors (Pauly and de Brecht~\cite{PdB15}), 
both of which are closely related to each other and to Weihrauch degrees.

\subsection{Other Models of Computability}
\label{subsec:models}

We have already seen in sections~\ref{sec:jumps} and \ref{sec:choice} that other models of computability can be 
characterized in the Weihrauch lattice. This includes the classes of problems that are computable with finitely many mind changes,
limit computable, non-deterministically computable and Las Vegas computable.

There are completely different algebraic models of computability such as the Blum-Shub-Smale machines~\cite{BCSS98} (BSS machines). 
Due to their algebraic nature, the class of functions computable by these
machines lack certain completeness properties and cannot be characterized exactly in the Weihrauch lattice. 
However, some tight upper bounds have been found by Neumann and Pauly~\cite{NP18}. 

\begin{theorem}[Algebraic computation]
\label{thm:algebraic-models}
If $f:\In\IR^*\to\IR^*$ is computable on a BSS machine, then $f\leqW\C_\IN$ and there is a function $f:\In\IR^*\to\IR^*$ that is computable 
        on a BSS machine and satisfies $f\equivW\C_\IN$.
\end{theorem}

Hertling and Weihrauch~\cite{HW94,Her96b} have studied how the number of tests that are performed are related to degeneracies in computations.
Yet a further class of machines can be obtained if one allows infinite computation time of higher order.
Weihrauch computability was generalized to this context by Carl~\cite{Car16a} and Galeotti and Nobrega~\cite{GN17}.




\bibliographystyle{spmpsci_cca}
\bibliography{C:/Users/vbrattka/Dropbox/Bibliography/lit}
There is an electronic version of a bibliography on Weihrauch complexity\footnote{See {\tt http://cca-net.de/publications/weibib.php}}.

\newpage

\section{Appendix: Additional Remarks and Proofs}

In this section\footnote{This section is not supposed to be part of the published version of this survey.} we add some further references
for results that have been presented, and we provide some additional proofs that close some gaps.

\subsection*{Represented Spaces}

We motivate the way in which we have defined the completion in Definition~\ref{def:operations-representations} by the following lemma.

\begin{lemma}[Precompleteness]
\label{lem:precompleteness}
Let $(X,\delta_X)$ be a represented space and $(\overline{X},\delta_{\overline{X}})$ its completion.
\begin{enumerate}
\item $\id:X\into\overline{X},x\mapsto x$ is a computable injection, i.e.,
\begin{enumerate}
\item There is a computable $F:\In\IN^\IN\to\IN^\IN$ with $\delta_X(p)=\delta_{\overline{X}}F(p)$ for $p\in\dom(\delta_X)$.
\item There is a computable $G:\In\IN^\IN\to\IN^\IN$ with $\delta_{\overline{X}}(p)=\delta_XG(p)$ for $p\in\delta_{\overline{X}}^{-1}(X)$.
\end{enumerate}
\item For every computable function $F:\In\IN^\IN\to\IN^\IN$ there exists a total computable function 
$\overline{F}:\IN^\IN\to\IN^\IN$ such that $\delta_{\overline{X}}F(p)=\delta_{\overline{X}}\overline{F}(p)$ for all $p\in\dom(F)$.
\end{enumerate}
\end{lemma}
\begin{proof}
The function $F$ with $F(p):=p+1$ is computable and satisfies the claim, where $(p+1)(n):=p(n)+1$.
The function $G$ with $G(p):=p-1$ is a partial computable function that satisfies the claim.
Given an arbitrary computable $F$, there is a Turing machine that computes $F$.
We modify this Turing machine so that it computes a total function $\overline{F}$ as follows:
upon input $p$ we consecutively write the output symbols of $F(p)$ on the output tape,
but whenever for some fixed time no new output symbol is produced, then we write an additional symbol
$0$ on the output tape. This guarantees that the output is infinite and that $\overline{F}(p)-1=F(p)-1$ for all $p\in\dom(F)$.
\qed
\end{proof}

\begin{proof}[of Proposition~\ref{prop:closed}]
See \cite[Theorem~3.10, Corollary~3.14]{BP03}.
\qed
\end{proof}

\subsection*{The Weihrauch Lattice}

\begin{proof}[of Proposition~\ref{prop:GM09}]
We consider the represented spaces $(X,\delta_X)$, $(Y,\delta_Y)$, $(Z,\delta_Z)$ and $(W,\delta_W)$.
We prove 1. Let $f\leqW g$. Then there are computable functions $H,K:\In\IN^\IN\to\IN^\IN$
such that $H\langle \id,GK\rangle$ is a realizer of $f$ whenever $G$ is a realizer of $g$.
We choose $V:=\IN^\IN$. Now
$k:\In X\mto\IN^\IN\times Z$ with $k:=(\id,\delta_ZK)\circ\delta_X^{-1}$ 
and $h:\In \IN^\IN\times W\mto Y$ with $h:=\delta_Y\circ H\circ\langle\id\times\delta_W^{-1}\rangle$
are computable and we obtain
\begin{eqnarray*}
h\circ(\id\times g)\circ k &=& \delta_Y\circ H\circ\langle\id\times\delta_W^{-1}\rangle\circ(\id\times g)\circ(\id,\delta_ZK)\circ\delta_X^{-1}\\
&=& \delta_Y\circ H\langle\id,\delta_W^{-1}g\delta_ZK\rangle\circ\delta_X^{-1}\prefix f,
\end{eqnarray*}
where the last mentioned relation holds since for every $p\in\dom(\delta_W^{-1}g\delta_Z)$ and
$q\in\delta_W^{-1}g\delta_Z(p)$ there is a realizer $G\vdash g$ with $G(p)=q$, and for this realizer we have $H\langle\id,GK\rangle\vdash f$.

Let now $h:\In V\times W\mto Y$ and $k:\In X\mto V\times Z$ be computable with $h\circ (\id_V\times g)\circ k\prefix f$
for some represented space $V$
and let $H,K:\In\IN^\IN\to\IN^\IN$ be computable realizers of $h,k$, respectively. 
Let $K_1,K_2:\In\IN^\IN\to\IN^\IN$ be such that $K(p)=\langle K_1(p),K_2(p)\rangle$
and let $\pi:\IN^\IN\times\IN^\IN\to\IN^\IN,(p,q)\mapsto\langle p,q\rangle$.
Let $G\vdash g$. We note that this implies $\langle\id\times G\rangle\circ\pi^{-1}\vdash \id_V\times g$
and hence $H\circ\langle \id\times G\rangle\circ\pi^{-1}\circ K\vdash h\circ(\id_V\times g)\circ k\prefix f$.
We obtain
\[
H\circ\langle K_1\times\id\rangle\circ\pi^{-1}\circ\langle\id,G\circ K_2\rangle =H\circ\langle K_1,G\circ K_2\rangle=H\circ\langle \id\times G\rangle\circ\pi^{-1}\circ K\vdash f.
\]
Hence, $H_0:=H\circ\langle K_1\times\id\rangle\circ\pi^{-1}$ and $K_2$ are computable functions
that satisfy $H_0\langle\id,G K_2\rangle\vdash f$. The choice of $H_0$ and $K_2$ is independent of $G$
and hence they witness that $f\leqW g$ holds.

The statement 2.\ can be proved similarly. 
\qed
\end{proof}

The characterization of ordinary Weihrauch reducibility in \cite[Lemma~4.5]{GM09} is not correct,
which was noticed by Peter Hertling.
We provide a counterexample, i.e., we show that there are $f:\In X\mto Y$ and $g:\In W\mto Z$ such that $f\leqW g$ holds, 
but there is no computable $k:\In X\mto Z$ and $h:\In X\times Z\mto Y$ such that $h\circ(\id_X,gk)\prefix f$.

\begin{example}
Let $p,q\in\IN^\IN$ be Turing incomparable and let $\{0\}$ be represented in a non-standard way 
such that $p,q$ are the only two names of $0$. We represent $\{0,1\}$ in the usual way.
We now consider the functions $f:\{0\}\to\IN^\IN,0\mapsto\langle p,q\rangle$ and
$g:\{0,1\}\to\IN^\IN$ with $g(0)=p$, $g(1)=q$. Then $f\leqW g$ holds, since one can computably
distinguish $p,q$. Hence the function $K:\In\IN^\IN\to\IN^\IN$ that sends $p$ to a name of $1$
and $q$ to a name of $0$ is computable. The function $H:\In\IN^\IN\to\IN^\IN$ that maps
$\langle p,q\rangle$ and $\langle q,p\rangle$ both to $\langle p,q\rangle$ is also computable. 
The functions $K,H$ witness the reduction $f\leqW g$. On the other hand, there are no computable $h,k$
with $h\circ(\id_X,gk)\prefix f$: it is easy to check that for each of the three computable $k:\{0\}\mto\{0,1\}$ 
there is no suitable $h$.
\end{example}

\begin{proof}[of Theorem~\ref{thm:Tavana-Weihrauch}]
See~\cite{TW11}.
\qed
\end{proof}

\begin{proof}[of Proposition~\ref{prop:cylinder}]
See \cite[Proposition~3.5, Corollary~3.6]{BG11}.
\qed
\end{proof}

\begin{proof}[of Proposition~\ref{prop:mon-closure}]
For the monotonicity of $\times$ and $\sqcap$ with respect to $\leqW$ and $\leqSW$, see \cite[Propositions~3.2, 3.10]{BG11}.
For the fact that parallelization is a closure operator with respect to $\leqW$ and $\leqSW$, see \cite[Proposition~4.2]{BG11}
and for finite parallelization with respect to the topological version of $\leqW$, see \cite[Theorem~6.2]{Pau10a}.
That $\sqcup$ is monotone for the continuous version of $\leqW$ follows from \cite[Corollary~4.7]{Pau10a}.
The proofs can be transferred to the computable case, and for the monotonicity of $\sqcup$ and of finite parallelization 
the proofs can be transferred also to $\leqSW$.
The fact that $\boxplus$ is monotone with respect to $\leqSW$ follows from Theorem~\ref{thm:strong-lattice} (see below).
In \cite[Proposition~3.12]{Dzh18} it was also proved that $f\sqcup g\equivW f\boxplus g$, which implies that $\boxplus$ is monotone with respect to $\leqW$.
We still need to prove that $+$ is monotone with respect to $\leqW$ and $\leqSW$:
Let $f_i:\In X_i\mto Y_i$ and $g_i:\In Z_i\mto W_i$ be problems with
$f_1\leqW f_2$ via computable $H_1,K_1$ and $g_1\leqW g_2$ via computable $H_2,K_2$. 
We show that $f_1+f_2\leqW g_1+g_2$. Let $\overline{H_1}$ and $\overline{H_2}$ be total
computable extensions of $H_1$ and $H_2$ with respect to $\overline{Y_1}$ and $\overline{W_1}$, respectively.
We define computable $H,K$ via $K\langle p,q\rangle:=\langle K_1(p),K_2(q)\rangle$ and
$H\langle\langle p,q\rangle,\langle r,s\rangle\rangle:=\langle\overline{H_1}\langle p,r\rangle,\overline{H_2}\langle q,s\rangle\rangle$.
Then $f_1+f_2\leqW g_1+g_2$ holds via $H,K$.
Monotonicity of $+$ with respect to $\leqSW$ can be proved analogously.
\qed
\end{proof}

\begin{proof}[of Lemma~\ref{lem:realizer}]
This proof is immediate since $f:\In X\mto Y$ and $f^\r=\delta_Y^{-1}\circ f\circ\delta_X:\In\IN^\IN\mto\IN^\IN$ share exactly
the same realizers $F:\In\IN^\IN\to\IN^\IN$.
\qed
\end{proof}

\begin{proof}[of Theorem~\ref{thm:lattice}]
By \cite[Theorem~3.14]{BG11} $(\WW,\leqW)$ is an upper semilattice with infimum $\sqcap$ and
by \cite[Corollary~4.7, Theorem~4.23]{Pau10a} $(\WW,\leqW)$ is a distributive lower semilattice with supremum $\sqcup$.
Note that the last mentioned result was proved for the continuous version of $\leqW$, but in the case of finite
suprema the proof goes through for the computable version too.
\qed
\end{proof}

\begin{proof}[of Theorem~\ref{thm:strong-lattice}]
This result was proved in \cite{Dzh18}. Note that we use a completion that is defined in a slightly different
way and \cite[Proposition~3.5]{Dzh18} can be replaced by Lemma~\ref{lem:precompleteness}.
The function $e$ defined in \cite{Dzh18} corresponds to function $G$ from that lemma
and the map $F\mapsto\overline{F}$ from that lemma plays the role of $a$ from \cite{Dzh18}.
Compare the monotonicity proof for $+$ in the proof of Proposition~\ref{prop:mon-closure} above.
\qed
\end{proof}

\subsection*{Algebraic and Topological Properties}

\begin{proof}[of Example~\ref{ex:idempotent-parallelizable}]
The intermediate value theorem has been classified in \cite[Theorem~6.2]{BG11a} which yields
a classification of its parallelization (e.g.\ via \cite[Corollary~3.11]{BG11a}, Corollary~\ref{cor:choice-zero} and
Theorem~\ref{thm:choice-cantor}). We note that we have $\C_2\leqSW\IVT$ and hence
even the strong reduction $\Z_{[0,1]}\leqSW\widehat{\IVT}$ holds. The fact that $\IVT$ is
not idempotent has been proved in \cite[Theorem~93]{BLRMP18}.
\qed
\end{proof}

\begin{proof}[of Proposition~\ref{prop:idempotency-parallelizability}]
The fact that any (strongly parallelizable) problem is (strongly) idempotent follows from \cite[Proposition~4.6]{BG11}.
Since $\id_{\{0\}}= f^0\leqSW f^*$, it is clear that any $f^*$ is strongly pointed.
Since $f\times f= f^2\leqSW f^*$, it is also clear that $f$ is (strongly) idempotent if $f\equivW f^*$ (or $f\equivSW f^*$, respectively). 
Vice versa, if $f\times f\equivW f$, then there is a uniform method to show $f^n\leqW f$ for all $n\geq1$
and if $f$ is pointed, then also $f^0\leqW f$ holds. This implies $f^*\leqW f$. An analogous argument works in the strong case.
\qed
\end{proof}

\begin{proof}[of Theorem~\ref{thm:squashing}]
The squashing theorem was first formulated and proved in \cite[Theorem~2.5]{DDH+16}.
A detailed proof of the version for general problems on Baire space as stated here can be found in \cite[Theorem~6.3.3]{Rak15}.
\qed
\end{proof}

\begin{proof}[of Proposition~\ref{prop:fractal}]
This result has first been proved for ordinary Weihrauch reducibility in \cite[Lemma~5.5]{BBP12} with the
notion of a fractal still being implicit and it has been formulated similarly as here in \cite[Proposition~2.6]{BGM12}.
\qed
\end{proof}

\begin{proof}[of Proposition~\ref{prop:densely-realized}]
This result has been proved for $f:\In\IN^\IN\mto\IN^\IN$ in~\cite[Proposition~55]{BP18}.
This implies the result for general $f:\In X\mto Y$ if $Y$ has a total representation $\delta_Y$.
In this case $f^{\rm r}(p)$ is dense in $\dom(\delta_Y)$ if it is dense in $\IN^\IN$.
\qed
\end{proof}

\subsection*{Completeness, Composition and Implication}

\begin{proof}[of Theorem~\ref{thm:completeness}]
See \cite[Propositions~3.15 and 3.16]{HP13}.
\qed
\end{proof}

\begin{proof}[of Theorem~\ref{thm:compositional-product-implication}]
See \cite[Corollary~18 and Theorem~24]{BP18}.
\qed
\end{proof}

\begin{proof}[of Proposition~\ref{prop:compositional-product-cylinder}]
See \cite[Lemma~17, Corollary~18 and Proposition~29]{BP18}.
This deserves some further justification, since $\star$ was defined in a slightly different
way here compared to \cite{BP18}. However, $f\star g\leqW f*g$ is obvious, and
using the notions from \cite{BP18} for $f:\In X\mto Y$ it is not too difficult
to see that $(\id\times f)\circ g^{\rm t}_{\IN^\IN\times X}\leqW f\star g$ holds. 
\qed
\end{proof}

\begin{proof}[of Corollary~\ref{cor:strong-compositional-product}]
See \cite[Lemma~2.5]{BHK17a}.
\qed
\end{proof}

\begin{proof}[of Proposition~\ref{prop:algebraic-properties}]
See \cite[Proposition~32]{BP18} for the statements on $*$ and $\to$ and
\cite[Lemma~2.6]{BHK17a} for the statement regarding $\stars$.
\qed
\end{proof}

\begin{proof}[of Proposition~\ref{prop:order}]
$f\times g\leqW f*g$ was proved in \cite[Lemma~4.3]{BGM12}, and this implies $f\times g\leqSW f\star g$,
since $f\star g$ is a cylinder by Proposition~\ref{prop:compositional-product-cylinder}.
In \cite{BBP12} it was mentioned that $f\sqcup g\leqW f\times g$ holds for pointed $f,g$.
There are counterexamples for the non-pointed case given in the proof of \cite[Proposition~34]{BP18}.
In the proof of \cite[Proposition~3.12]{Dzh18} it is shown that $f\boxplus g\leqSW f\sqcup g$ holds.
Since $f\sqcap g$ is the infimum of $f,g$ with respect to $\leqSW$ and $f\boxplus g$ is the supremum,
it is clear that $f\sqcap g\leqSW f\boxplus g$.
The reduction $f+g\leqSW f\sqcap g$ is easy to see: we use the same input for $f\sqcap g$ as for $f+g$
and we just pair the result of $f\sqcap g$ with $\bot$ in the second or first component, depending on whether
$f\sqcap g$ yields a result for $f$ or for $g$, respectively. 
Finally, also $f\boxplus g\leqSW f\times g$ is easy to see for pointed $f, g$: depending on the instance
of $f\boxplus g$ we forward the instance to $f$ or $g$ and we evaluate the other problem among $f$ and $g$
on some arbitrary computable input. The resulting solution of $f\times g$ is a solution for $f\boxplus g$.
\qed
\end{proof}

\begin{proof}[of Proposition~\ref{prop:composition-implication}]
See \cite[Corollary~25]{BP18}.
\qed
\end{proof}

\begin{proof}[of Theorem~\ref{thm:Brouwer-Heyting}]
See \cite[Theorems~4.1 and 4.9]{HP13}.
\qed
\end{proof}

\subsection*{Limits and Jumps}

\begin{proof}[of Proposition~\ref{prop:limit-computable}]
See \cite[Theorems~7.11 and Corollary~7.15]{BBP12}.
\qed
\end{proof}

\begin{proof}[of Proposition~\ref{prop:limits}]
This follows from~\cite[Proposition~9.1]{Bra05}.
\qed
\end{proof}

\begin{proof}[of Proposition~\ref{prop:limN-closure}]
This follows, for instance, form Corollary~\ref{cor:choice-composition} and
has been proved, for instance, in \cite[Corollary~7.6]{BBP12}.
\qed
\end{proof}

\begin{proof}[of Theorem~\ref{thm:Borel}]
This follows from \cite[Proposition~9.1]{Bra05}; see the proof of \cite[Fact 2.2]{BR17} for a more detailed
reasoning.
\qed
\end{proof}

\begin{proof}[of Proposition~\ref{prop:algebraic-lim}]
Since $\lim\equivSW\widehat{\LPO}$ (e.g., by Theorem~\ref{thm:lim}), it is clear that $\lim$
is strongly parallelizable and hence, in particular, strongly idempotent. This implies that it is
a cylinder, since $\id\times\lim\leqSW\lim\times\lim\leqSW\lim$. 
That $\lim$ is a strong fractal and finitely tolerant is obvious and hence it is (strongly) countably irreducible.
\qed
\end{proof}

\begin{proof}[of Theorem~\ref{thm:lim}]
It follows from \cite[Lemma~6.3]{BG11} that $\widehat{\LPO}\equivSW\EC$
and from \cite[Corollary~3.11]{BG11a} that $\widehat{\LPO}\equivW\widehat{\C_\IN}$ holds (for a definition of $\C_\IN$ see Definition~\ref{closedchoice}).
Since all the problems are cylinders, it is clear that this can be strengthened to a strong Weihrauch equivalence
and hence $\widehat{\LPO}\equivSW\widehat{\lim_\IN}$ follows.
See \cite[Example~3.10]{BBP12} for $\lim\equivW\widehat{\C_\IN}$ and once again this extends
to a strong Weihrauch equivalence since $\lim$ is a cylinder.
It follows from \cite[Lemma~8.9]{BBP12} that $\lim\equivSW\J$.
$\sup\equivW\EC$ follows from \cite[Proposition~3.7]{BG11a} and once again
this extends to a strong Weihrauch equivalence since $\sup$ is a cylinder
(this follows since $\sup_{2^\IN}$ with respect to the lexicographic order is a cylinder
and $\sup_{2^\IN}\leqSW\sup\leqSW\lim_\IR\leqW\sup_{2^\IN}$ holds).
The equivalence $\inf\equivSW\EC$ can be proved analogously.  
\qed
\end{proof}

\begin{proof}[of Proposition~\ref{prop:jump-monotone}]
See \cite[Proposition~5.6]{BGM12}.
\qed
\end{proof}

\begin{proof}[of Theorem~\ref{thm:jump-inversion}]
See \cite[Theorem~11]{BHK17}.
\qed
\end{proof}

\begin{proof}[of Proposition~\ref{prop:jump-algebraic}]
See \cite[Proposition~5.7]{BGM12}.
\qed
\end{proof}

\begin{proof}[of Proposition~\ref{prop:tolerance-fractality}]
It is obvious that every jump $f'$ is finitely tolerant and the rest has been proved in \cite[Proposition~5.8]{BGM12}.
\qed
\end{proof}

\begin{proof}[of Proposition~\ref{prop:jump-cylinder}]
See \cite[Corollary~5.16]{BGM12}.
\qed
\end{proof}

\begin{proof}[of Corollary~\ref{cor:invariance}]
See \cite[Lemma~4.4]{BG11a}; the remaining statements follow since 
all other notions can be characterized as lower cones of some complete
problem in the Weihrauch lattice (with respect to some oracle in the case of continuity).
\qed
\end{proof}

\begin{proof}[of Proposition~\ref{prop:mind}]
See \cite[Lemma~4.4]{BG11a}.
\qed
\end{proof}

\begin{proof}[of Proposition~\ref{prop:cardinality}]
See \cite[Proposition~3.6]{BGH15a}.
\qed
\end{proof}

\begin{proof}[of Example~\ref{ex:limit}]
See \cite[Corollary~7.12]{BBP12} for $\lim_\IN\equivW\lim_\Delta$.
The other statements are obvious.
\qed
\end{proof}

\begin{proof}[of Theorem~\ref{thm:LPO}]
This result has originally been proved in \cite[Theorem~3.7]{Wei92c}.
See also \cite[Proposition~6.2]{BG11}.
\qed
\end{proof}

\begin{proof}[of Theorem~\ref{thm:linear}]
See \cite[Theorem~4.3]{Bra99}.
\qed
\end{proof}

\subsection*{Choice}

\begin{proof}[of Proposition~\ref{prop:injection-surjection}]
See \cite[Proposition~3.7, Corollary~4.3]{BBP12}.
\qed
\end{proof}

\begin{proof}[of Proposition~\ref{prop:choice-fractal}]
See \cite[Corollary~5.6]{BBP12} (and the reasoning before it), which shows that $\C_\IN,\C_{2^\IN},\C_\IR$ and $\C_{\IN^\IN}$ are fractals,
and this can be proved analogously for $\PC_\IR$ and $\PC_{2^\IN}$.
In \cite[Lemma~15.5]{BGH15a} it is shown that $\ConC_{[0,1]}$ is a total fractal, and in \cite[Observation~2]{LRP15a} it is proved that $\XC_{[0,1]^{n+1}}$
is a total fractal. The techniques used to prove totality for the last mentioned result also apply to 
$\C_{2^\IN}$ and $\PC_{2^\IN}$ (compare also \cite[Lemma~4.7]{BBP12}).  
\qed
\end{proof}

\begin{proof}[of Proposition~\ref{prop:choice-cylinder}]
It is easy to see that $\C_{2^\IN}$ and $\C_{\IN^\IN}$ are closed under $\times$ and that
the identity is strongly reducible to any of them (using singletons). Hence, $\id\times\C_X\leqSW\C_X$ for $X\in\{2^\IN,\IN^\IN\}$. That $\C_\IR$ is a cylinder follows then from Theorem~\ref{thm:CR}.
That $\C_\IN$ is not a cylinder follows from $\#\C_\IN=|\IN|$. Likewise, $\PC_{2^\IN},\PC_\IR$ and $\PC_{\IN^\IN}$ are
not cylinders as proved in \cite[Corollary~3.9]{BGH15a}.
That $\ConC_{[0,1]}$ is not a cylinder was proved in \cite[Theorem~9.5]{BLRMP18}.
\qed
\end{proof}

\begin{proof}[of Proposition~\ref{prop:choice-spaces}]
See \cite[Corollary~4.2]{BBP12} for the statements on $\C_{\IN^\IN}$ and $\C_{2^\IN}$,
see \cite[Proposition~4.8]{BBP12} for the statement on $\C_\IR$.
The statement on $\C_\IN$ follows from Proposition~\ref{prop:injection-surjection}.
Likewise, for every computable metric space $X$ with $n$ elements there is a computable surjection $\{0,...,n-1\}\to X$ and
hence $\C_X\leqSW\C_n\leqSW\K_\IN$.
\qed
\end{proof}

\begin{proof}[of Theorem~\ref{thm:non-deterministic}]
See \cite[Theorem~7.2]{BBP12}.
\qed
\end{proof}

\begin{proof}[of Corollary~\ref{cor:notions-computability}]
The first two statements are a corollary of Theorem~\ref{thm:non-deterministic}, and the statement 
on Las Vegas computability can be found in \cite[Corollary~3.4]{BGH15a}.
\qed
\end{proof}

\begin{proof}[of Theorem~\ref{thm:independent-choice}]
See \cite[Theorem~7.3]{BBP12} and \cite[Theorem~4.3]{BGH15a}.
\qed
\end{proof}

\begin{proof}[of Corollary~\ref{cor:choice-composition}]
See \cite[Corollary~7.6]{BBP12} for the statements on choice, and
see \cite[Corollary~4.5]{BGH15a} for the statements on positive choice.
The proof of Theorem~\ref{thm:independent-choice} (see \cite[Theorem~7.3]{BBP12}) goes through in the case
of unique choice too, hence one obtains the result for $\UC_{\IN^\IN}$.
\qed
\end{proof}

\begin{proof}[of Theorem~\ref{thm:CN-strong}]
See \cite[Proposition~3.8]{BGM12} for $\UC_\IN\equivSW\C_\IN\equivSW\lim_\IN$,
\cite[Corollary~4.13]{BBP12} for $\C_\IN\equivSW\C_\IQ$ and
\cite[Proposition~7.1]{BR17} for $\C_\IN\equivSW\min^{\rm c}$.
We still need to prove $\max\equivSW\C_\IN$. It is easy to see that $\max\leqSW\lim_\IN$:
given a sequence of natural numbers with an upper bound, we always repeat the maximal number
that we have seen so far. The limit of the resulting sequence is the maximum of the original sequence.
It is also easy to see that $\C_\IN\leqSW\max$: given a sequence of numbers, we generate a new
sequence where we enumerate the minimal number that we have not seen so far. The maximum
of this sequence is an element that is missing in the original sequence.
\qed
\end{proof}

\begin{proof}[of Theorem~\ref{thm:CN-weak}]
See \cite[Proposition~3.3]{BG11a} for $\CFC_\IN\equivW\C_\IN$
and \cite[Corollary~7.12]{BBP12} for $\C_\IN\equivW\lim_\Delta$.
\qed
\end{proof}

\begin{proof}[of Theorem~\ref{thm:CN-elimination}]
See \cite[Theorem~2.4]{LRP15a}.
\qed
\end{proof}

\begin{proof}[of Corollary~\ref{cor:CN-separations}]
See \cite[Proposition~4.9]{BG11a} for $\ConC_{[0,1]}\nleqW\C_\IN$
and \cite[Proposition~21]{BP10} for $\PC_{2^\IN}\nleqW\C_\IN$.
\qed
\end{proof}

\begin{proof}[of Proposition~\ref{prop:LLPO-LPO}]
See \cite[Example~3.2]{BBP12} for the statement $\C_2\equivSW\LLPO$,
see \cite[Theorem~4.2]{Wei92c} for the original proof that $\LLPO\lW\LPO$ (and \cite[Theorem~7.13]{BG11} for an alternative proof)
and see \cite[Propositions~3.3 and 4.5]{BG11a} for $\LPO\lW\C_\IN$.
\qed
\end{proof}

\begin{proof}[of Proposition~\ref{prop:finite-choice}]
See \cite[Theorem~5.4 and 4.3]{Wei92c}.
\qed
\end{proof}

\begin{proof}[of Theorem~\ref{thm:cardinality-products}]
See \cite[Theorem~32]{Pau10}.
\qed
\end{proof}

\begin{proof}[of Proposition~\ref{prop:KN}]
See \cite[Proposition~10.9]{BGM12}.
\qed
\end{proof}

\begin{proof}[of Corollary~\ref{cor:KN}]
See \cite[Corollary~10.10]{BGM12}.
\qed
\end{proof}

\begin{proof}[of Proposition~\ref{prop:min}]
$\min\leqSW\LPO^*$: given a sequence $p$ with the first element $n$, $n$ parallel applications of $\LPO$ are sufficient in order to find
out whether there is an occurrence of any of the elements $0,...,n-1$ in the sequence $p$. Once we know all the answers, we know the smallest
element in $p$ (without further access to $p$). $\LPO^*\leqSW\min$: given $n$ sequences $p_0,...,p_{n-1}$ we generate a sequence $p$ that starts
with repetitions of the number $\sum_{i=0}^{n-1}2^i$; once we find for some $i<n$ such that $p_i\not=\widehat{0}$, then we subtract $2^i$ from the number 
that we enumerate into the output sequence $p$. Given the minimum in $p$ we know all the answers to $\LPO(p_i)$ for $i<n$ without further access to the input.
\qed
\end{proof}

\begin{proof}[of Theorem~\ref{thm:choice-cantor}]
See \cite[Theorem~8.2]{BG11} for $\WKL\equivSW\widehat{\LLPO}$
and \cite[Example~3.10 and Corollary~4.5]{BBP12} for the other claims.
\qed
\end{proof}

\begin{proof}[of Corollary~\ref{cor:C2N-equiv}]
See \cite[Corollary~4.6]{BBP12} 
\qed
\end{proof}

\begin{proof}[of Theorem~\ref{thm:C2N-elimination}]
See \cite[Theorem~5.1]{BBP12}.
\qed
\end{proof}

\begin{proof}[of Corollary~\ref{cor:C2N-single-valued}]
See \cite[Corollary~8.8]{BG11}.
\qed
\end{proof}

\begin{proof}[of Corollary~\ref{cor:C2N-separation}]
See \cite[Corollary~4.10]{BG11a}.
\qed
\end{proof}

\begin{proof}[of Theorem~\ref{thm:PC2N}]
See \cite[Proposition~8.2]{BGH15a}.
\qed
\end{proof}

\begin{proof}[of Proposition~\ref{prop:PC-C}]
See \cite[Theorem~20]{BP10}, \cite[Corollary~8.5]{BGH15a} and \cite[Proposition~4.2]{DDH+16} for 
an independent proof.
\qed
\end{proof}

\begin{proof}[of Theorem~\ref{thm:quantitative-WWKL}]
See \cite[Proposition~4.7]{DDH+16} and \cite[Corollary~10.2]{BGH15a}. 
\qed
\end{proof}

\begin{proof}[of Theorem~\ref{thm:CC}]
See \cite[Theorems~6.2, 7.1, Corollary~7.2]{BLRMP18}.
\qed
\end{proof}

\begin{proof}[of Theorem~\ref{thm:IVT}]
See \cite[Theorem~6.2]{BG11a}.
\qed
\end{proof}

\begin{proof}[of Theorem~\ref{thm:CC-not-idempotent}]
See \cite[Theorem~9.3]{BLRMP18}.
\qed
\end{proof}

\begin{proof}[of Theorem~\ref{thm:convex-choice}]
See \cite[Corollary~3.31]{LRP15a}.
\qed
\end{proof}

\begin{proof}[of Theorem~\ref{thm:XC-composition}]
See \cite[Theorem~3]{Kih16a}.
\qed
\end{proof}

\begin{proof}[of Theorem~\ref{thm:KX}]
See \cite[Corollaries~10.11 and 10.14]{BGM12}.
\qed
\end{proof}

\begin{proof}[of Proposition~\ref{prop:KN-upper}]
See \cite[Proposition~9.2]{BLRMP18} for the proof of $\K_\IN\leqW\ConC_{[0,1]}$;
$\K_\IN\leqW\PC_{2^\IN}$ is easy to prove.
\qed
\end{proof}

\begin{proof}[of Theorem~\ref{thm:CR}]
See \cite[Corollary~4.9]{BBP12} for most of the claims.\linebreak
$\C_\IN\star\C_{2^\IN}\leqW\C_{\IN}\times\C_{2^\IN}$
and $\C_{2^\IN}\star\C_\IN\leqW\C_{\IN}\times\C_{2^\IN}$
follow by Theorem~\ref{thm:independent-choice}. The strong reduction is then obtained since $\C_{2^\IN}$ is a cylinder by Theorem~\ref{prop:choice-cylinder}, 
The inverse reductions follow since $f\times g\leqSW f\star g$ holds for all problems $f,g$ by Proposition~\ref{prop:order}.
\qed
\end{proof}

\begin{proof}[of Corollary~\ref{cor:CR-single}]
See \cite[Corollary~5.3]{BBP12}.
\qed
\end{proof}

\begin{proof}[of Theorem~\ref{thm:low-basis}]
See \cite[Theorem~8.7]{BBP12}.
\qed
\end{proof}

\begin{proof}[of Example~\ref{ex:CR-CN}]
See \cite[Example~40]{BP18}.
\qed
\end{proof}

\begin{proof}[of Theorem~\ref{thm:PCR}]
See \cite[Corollary~6.4, Theorem~9.3]{BGH15a}.
\qed
\end{proof}

\begin{proof}[of Corollary~\ref{cor:CR-parallelizability}]
This follows from $\C_\IN\lSW\PC_\IR\lSW\C_\IR\lSW\lim\equivSW\widehat{\C_\IN}$.
\qed
\end{proof}

\begin{proof}[of Theorem~\ref{thm:non-Ks}]
See \cite[Corollary~4.10]{BBP12}.
\qed
\end{proof}

\begin{proof}[of Theorem~\ref{thm:CNN-equiv}]
See \cite[Corollary~4.11 and 4.14]{BBP12}.
\qed
\end{proof}

\begin{proof}[of Theorem~\ref{thm:CNN-single-valued}]
See \cite[Theorem~7.7]{BBP12}. 
\qed
\end{proof}

\begin{proof}[of Proposition~\ref{prop:CNN-parallelizability}]
Given a sequence $(A_n)_{n\in\IN}$ of closed sets $A_n\In\IN^\IN$, we can compute
the closed set $A:=\langle A_0\times A_1\times A_2\times...\rangle\In\IN^\IN$. From any point
$a\in A$ we can compute a sequence $(a_n)_{n\in\IN}$ with $a_n\in A_n$ for all $n\in\IN$.
\qed
\end{proof}

\begin{proof}[of Proposition~\ref{pop:UCNN}]
It is known that $\lim\leqW\UC_{\IN^\IN}$ holds \cite{BBP12}.
Since $\UC_{\IN^\IN}$ is closed under composition, we obtain $\lim^{[n]}\leqW\UC_{\IN^\IN}$,
and since $\lim^{[n]}\lW\lim^{[n+1]}$ by  Theorem~\ref{thm:Borel}, we can conclude that $\lim^{[n]}\lW\UC_{\IN^\IN}$.
It is clear that $\UC_{\IN^\IN}\leqW\C_{\IN^\IN}$ and Kihara (personal communication) noted that the
reduction is strict: By a result of Kleene there is a computable tree $T\In\IN^*$ that has
infinite branches but no hyperarithmetic ones.
By Kreisel's basis theorem every isolated member of a computable tree $T\In\IN^*$ is hyperarithmetic.
That is $\UC_{\IN^\IN}$ always has hyperarithmetic solutions for computable inputs, but
$\C_{\IN^\IN}$ does not necessarily have such solutions. Since hyperarithmeticity is invariant
under Weihrauch reducibility, we obtain the strictness of the reduction. 
\qed
\end{proof}

\begin{proof}[of Theorem~\ref{thm:PCNN}]
See \cite[Theorem~5.2]{BK17}.
\qed
\end{proof}

\begin{proof}[of Theorem~\ref{thm:jumps-choice}]
See \cite[Theorems~9.4 and 11.2]{BGM12} and \cite{BCG+17}.
\qed
\end{proof}

\begin{proof}[of Theorem~\ref{thm:composition-jump}]
See \cite[Theorem~8.13]{BGM12} and \cite[Fact~2.3~(6)]{BR17}
for \linebreak
${\C_{2^\IN}'}^{[n]}\equivW\C_{2^\IN}^{(n)}$ and
\cite{BK17} for $\PC_{2^\IN}'*\PC_{2^\IN}'\equivW\PC_\IR'*\PC_\IR'\equivW\PC_\IR'$.
In a similar way one can prove $\C_\IN'*\C_\IN'\equivW\C_\IN'$ (unpublished result of Brattka, H\"olzl and Kuyper).
\qed
\end{proof}

\begin{proof}[of Theorem~\ref{thm:jump-separations}]
See the proof of \cite[Corollary~9.1]{BGH15a} for\linebreak
$\C_2^{(n+1)}\nleqW\lim^{(n)}$,
see \cite[Theorem~12.2]{BGM12} for $\LPO^{(n)}\nleqW\C_{2^\IN}^{(n)}$ and
see \cite[Corollary~14.9]{BGH15a} for $\C_{2^\IN}\nleqW\PC_{2^\IN}^{(n)}$.
\qed
\end{proof}

\begin{proof}[of Theorem~\ref{thm:alternating-hierarchies}]
$\C_2\leqSW\LPO\leqSW\C_2'$ is easy to see; the separation follows from Theorem~\ref{thm:jump-separations}, since $\LPO\leqSW\lim$ by Theorem~\ref{thm:lim}, on the one hand, and $\C_2\leqSW\C_{2^\IN}$ by Proposition~\ref{prop:choice-spaces}.
For the statement on $\K_\IN$ and $\C_\IN$ see \cite[Proposition~7.2]{BR17}.
For the statement on $\C_{2^\IN}$ and $\lim$ see \cite[Fact 2.3(5)]{BR17}.
\qed
\end{proof}

\begin{proof}[of Proposition~\ref{thm:CNN}]
See \cite[Theorem~9.16]{BGM12} for $\C_{\IN^\IN}'\equivSW\C_{\IN^\IN}$.
The other reduction $\UC_{\IN^\IN}'\equivSW\UC_{\IN^\IN}$ can be proved analogously. 
\qed
\end{proof}

\begin{proof}[of Proposition~\ref{prop:upper-bounds-CC-CN}]
See \cite[Proposition~15.7]{BGH15a} for $\ConC_{[0,1]}\leqW\PC_{2^\IN}'$,
see \cite[Proposition~5.21]{BR17} for $\ConC_{[0,1]}\leqW\C_\IN'$ and
see \cite[Corollary~9.7]{BGH15a} for $\C_\IN\leqW\PC_{2^\IN}$.
\qed
\end{proof}

\begin{proof}[of Theorem~\ref{thm:WWKL-single-valued}]
See \cite[Corollary~13.3]{BGH15a}.
\qed
\end{proof}

\begin{proof}[of Proposition~\ref{prop:AoUC}]
See \cite[Proposition~5.2.1.3]{Pau11}.
\qed
\end{proof}

\begin{proof}[of Theorem~\ref{thm:AoUC-upper-bound}]
See \cite[Theorem~16.3]{BGH15a}.
\qed
\end{proof}

\begin{proof}[of Theorem~\ref{thm:AoUC-separation}]
See \cite{Kih16a} for  $\XC_{[0,1]}*\AoUC_{[0,1]}\nleqW\XC_{[0,1]^n}$,
see \cite[Corollary~11]{KP16a} for $\C_2*\AoUC_{[0,1]}\nleqW\AoUC_{[0,1]}^*$,
see \cite[Proposition~17.4]{BGH15a} for $\C_2\times\AoUC_{[0,1]}\nleqW\ConC_{[0,1]}$.
\qed
\end{proof}

\begin{proof}[of Theorem~\ref{thm:AoUC-composition}]
See \cite[Corollary~12]{KP16a}.
\qed
\end{proof}

\subsection*{Classifications}

\begin{proof}[of Theorem~\ref{thm:upper-bounds}]
Meta theorems of this type have been discussed in \cite[Section 8]{BG11a} and
the proof formally follows from \cite[Lemma~8.5]{BG11a} for computable metric spaces
and from \cite[Proposition~4.2(9)]{Pau16} for the general case of represented spaces.
\qed
\end{proof}

\begin{proof}[of Theorem~\ref{thm:CN}]
See \cite[Theorem~5.2]{BG11a} for the Baire category theorem, 
\cite[Theorem~5.4]{BG11a} for Banach's inverse mapping theorem,
\cite[Theorem~5.6]{BG11a} for the open mapping theorem,
\cite[Theorem~5.8]{BG11a} for the closed graph theorem,
\cite[Theorem~5.10]{BG11a} for the uniform boundedness theorem,
\cite[Theorem~11.2]{BGH15a} for the Lebesgue covering lemma and
\cite[Theorem~4 and Figure~4]{PS18} for the partial identity from continuous functions to analytic functions.
\qed
\end{proof}

\begin{proof}[of Theorem~\ref{thm:C2N}]
See \cite{GM09} and \cite[Section~8]{BG11} for weak K\H{o}nig's lemma,
\cite{GM09} for the Hahn-Banach theorem,
\cite[Theorem~27]{Bra16} for the theorem of the maximum $\MAX$,
\cite[Theorems~6.2 and 7.1]{BLRMP18} for the Brouwer fixed point theorem,
\cite[Theorem~3.8]{BLRMP18} for finding connectedness components of sets $A\In[0,1]^n$,
\cite[Section~6]{BG11a} for the parallelization $\widehat{\IVT}$ of the intermediate value theorem and
\cite[Theorem~2]{LRP15} for Determinacy of Gale-Stewart games in $2^\IN$ with closed winning sets.
It is straightforward to see that $\HBC_1\equivW\C_{[0,1]}$ and the latter is
equivalent to $\C_{2^\IN}$ by Theorem~\ref{thm:choice-cantor}.
\qed
\end{proof}

\begin{proof}[of Theorem~\ref{thm:AoUC}]
In \cite[Corollary~40]{Pau10} it was proved that $\NASH\equivW\RDIV^*$
and by Proposition~\ref{prop:AoUC} we have $\RDIV\equivSW\AoUC_{[0,1]}$,
which implies the result for ordinary Weihrauch reducibility. In \cite[Lemma~17.1]{BGH15a}
it was shown that $\NASH$ is a cylinder, which implies that the result also holds for
strong Weihrauch reducibility.
\qed
\end{proof}

\begin{proof}[of Theorem~\ref{thm:CR-thm}]
The result on Frostman's lemma was proved in \cite[Corollary~57]{PF17}.
\qed
\end{proof}

\begin{proof}[of Theorem~\ref{thm:PC}]
Weak weak K\H{o}nig's lemma was classified in \cite{BP10} and \cite[Proposition~8.2]{BGH15a}
and Vitali's covering theorem in \cite{BGHP17}.
\qed
\end{proof}

\begin{proof}[of Theorem~\ref{thm:XC}]
The Browder-G\"ohde-Kirk fixed point theorem was studied in \cite[Theorem~5.8]{Neu15}.
\qed
\end{proof}

\begin{proof}[of Theorem~\ref{thm:lim-thm}]
The monotone convergence theorem was discussed in \cite[Fact~11.26]{BGM12} and \cite[Proposition~3.8]{BG11a},
the operator of differentiation was discussed in \cite{Ste89},
the Fr{\'e}chet-Riesz representation theorem for $\ell_2$ in \cite{Bra16},
the Radon-Nikodym theorem in \cite{HRW12},
the result on the parallelization of Banach's inverse mapping theorem follows from Theorem~\ref{thm:CN},
the result on finding a basis of a countable vector space was considered in \cite[Theorem~12]{HM17},
the result on finding a connected component of a countable graph in \cite[Theorem~6.4]{GHM15} and 
the partial identity from continuous functions on $\IR$ to Schwartz functions was studied in \cite[Theorem~7]{PS18}.
\qed
\end{proof}

\begin{proof}[of Theorem~\ref{thm:JC2N}]
K\H{o}nig's lemma has been classified in \cite[Theorem~5.13]{BR17},
the Bolzano-Weierstra\ss{} theorem $\BWT_\IR$ on Euclidean space has been studied in \cite[Corollary~11.7]{BGM12} and \cite{BCG+17},
the Arzel{\'a}-Ascoli theorem has been studied  in \cite{Kre14},
and determinacy of Gale-Stewart games in $2^\IN$ with winning sets that are differences of open sets
has been classified  in \cite[Theorem~3]{LRP15}.
\qed
\end{proof}

\begin{proof}[of Theorem~\ref{thm:Ramsey}]
See \cite[Theorems 3.5, 4.15 and 4.22]{BR17}.
\qed
\end{proof}

\begin{proof}[of Theorem~\ref{thm:CNN-thm}]
This result is based on personal communication (Marcone) and unpublished so far.
\qed
\end{proof}

\begin{proof}[of Theorem~\ref{thm:DNC}]
See \cite[Theorem~5.2]{BHK17a} and \cite[Proposition~81]{HK14a}.
\qed
\end{proof}

\begin{proof}[of Theorem~\ref{thm:genericity}]
See \cite[Corollaries 8.3, 9.7 and 10.8]{BHK18}.
\qed
\end{proof}

\begin{proof}[of Proposition~\ref{prop:computability-densely}]
See \cite{BHK17a,BHK18}.
\qed
\end{proof}

\begin{proof}[of Theorem~\ref{thm:MLR-PA-COH}]
See \cite[Proposition~58]{BP18} for the characterization of $\MLR$,
\cite[Theorem~6.7]{BHK17a} for the characterization of $\PA$,
\cite[Corollary~14.9]{BHK17a} for the characterization of $\COH$.
\qed
\end{proof}

\begin{proof}[of Theorem~\ref{prop:MLR-GEN-composition}]
See \cite[Footnote 7, Proposition~10.7]{BHK17a}, \cite[Proposition~2.7]{BK17}.
\qed
\end{proof}

\subsection*{Relations to Other Theories}

\begin{proof}[of Theorem~\ref{thm:Medvedev}]
See \cite[Theorem~5.1, Proposition~5.3]{BG11} for the statement regarding the embedding $c$.
See \cite[Lemma~5.6]{HP13} for the statement regarding the embedding $d$ into the Weihrauch lattice 
and \cite{Dzh18} regarding the embedding $d$ into the strong Weihrauch lattice.
\qed
\end{proof}

\begin{proof}[of Proposition~\ref{prop:many-one}]
Any computable partial function $H : \subseteq \mathbb{N} \times \{p,q\} \to \{p,q\}$ is a restriction of the projection to the second component, i.e.,
it satisfies $H(x,y) = y$ by the requirement that $p$ and $q$ are Turing incomparable. 
Thus, the Weihrauch reduction of the form $m_A \leqW m_B$ are precisely witnessed by computable functions 
$K:\IN\to\IN$ satisfying that $K(n) \in B\iff n \in A$. But these are exactly the witnesses for $A \leq_{\rm m} B$, hence we have an order-embedding.
To see that this is even a join-semilattice embedding, we recall that $A\oplus B$ is the supremum of $A$ and $B$ in the many-one join-semilattice 
and $m_{A\oplus B}\equivW m_A\sqcup m_B$.
\qed
\end{proof}

\begin{proof}[of Theorem~\ref{thm:Delta}]
The case $n = 0$ was provided by Pauly and de Brecht~\cite[Corollary~12]{PdB14}, the cases $n > 0$ were provided by Kihara~\cite[Theorem~1.5]{Kih15}.
\qed
\end{proof}

\begin{proof}[of Theorem~\ref{thm:algebraic-models}]
This theorem has been proved in \cite{NP18}.
\qed
\end{proof}

\end{document}